\documentclass[a4paper,10pt,twoside]{article}

\RequirePackage[colorlinks,citecolor=blue,urlcolor=blue]{hyperref}
\usepackage[title]{appendix}
\usepackage{xkeyval}
\usepackage[utf8]{inputenc}
\usepackage[T1]{fontenc}
\usepackage[english]{babel}

\usepackage[affil-it]{authblk}
\usepackage{blindtext}

\usepackage{csquotes}
\usepackage{enumerate}
\usepackage{bm}
\usepackage{verbatim}
\usepackage{indentfirst}
\usepackage{xcolor}
\usepackage[top=5.cm, bottom=5.0cm, left=1.5cm, right=1.5cm]{geometry}
\usepackage{authblk}
\usepackage{mathtools}
\allowdisplaybreaks
\usepackage{framed}
\usepackage{amsmath}
\usepackage{amsthm}	
\usepackage{amsfonts}
\usepackage{mathrsfs}	
\usepackage{amssymb}
\usepackage{dsfont}
\usepackage{bbm}

\theoremstyle{plain}
\newtheorem{theorem}{Theorem}[section]
\newtheorem{lemma}[theorem]{Lemma}
\newtheorem{proposition}[theorem]{Proposition}
\newtheorem{corollary}[theorem]{Corollary}
\theoremstyle{definition}
\newtheorem{definition}[theorem]{Definition}

\newtheorem{example}[theorem]{Example}

\theoremstyle{definition}
\newtheorem{remark}[theorem]{Remark}
\newtheorem*{theorem*}{Theorem (Main result)}
\numberwithin{equation}{section}
\numberwithin{figure}{section}
 \usepackage[nodayofweek]{datetime}

\newcommand{\hmu}{\hat{\mu}}

\newcommand{\rvlaw}[1][]{  {\mathscr{L}}{{(#1)}}  }

\newcommand{\ld}[1][]{ \frac{\delta{#1}}{\delta m}  }
\DeclarePairedDelimiterX{\inp}[2]{\langle}{\rangle}{#1, #2}
\newcommand{\intrd}[1][]{ \int_{\bT^d}  }
\newcommand{\vone}{v^{(1)}}

\newcommand{\regb}[1][]{(\textcolor{red}{\text{Int}-$b$-(${#1,1}$)})}
\newcommand{\lipb}[1][]{(\textcolor{red}{\text{Lip}-$b$-(${#1,1}$)})}

\newcommand{\regtphi}[1][]{(\textcolor{red}{\text{TReg}-$\widetilde{\Phi}$-(${#1},1$)})}

\newcommand{\hregb}[2]{(\textcolor{red}{\text{Int}-$b$-(${#1,#2}$)})}
\newcommand{\hintb}[2]{(\textcolor{red}{\text{Int}-$b$-(${#1,#2}$)})}
\newcommand{\hlipb}[2]{(\textcolor{red}{\text{Lip}-$b$-(${#1,#2}$)})}

\newcommand{\hregphi}[2]{(\textcolor{red}{\text{TReg}-$\Phi$-(${#1,#2}$)})}
\newcommand{\hintphi}[2]{(\textcolor{red}{\text{TInt}-$\Phi$-(${#1,#2}$)})}
\newcommand{\hlipphi}[1][]{(\textcolor{red}{\text{TLip}-$\Phi$-(${#1}$)})}
\newcommand{\hliptphi}{(\textcolor{red}{\text{TLip}-${\widetilde{\Phi}-(1)}$)}}

\newcommand{\na}{n,\infty}

\newcommand{\bE}{\mathbb{E}}

\newcommand{\bN}{\mathbb{N}}
\newcommand{\bP}{\mathbb{P}}

\newcommand{\bT}{\mathbb{T}}

\newcommand{\bR}{\mathbb{R}}
\newcommand{\bZ}{\mathbb{Z}}

\newcommand{\cF}{\mathcal{F}}

\newcommand{\cP}{\mathcal{P}}

\newcommand{\cU}{\mathcal{U}}
\newcommand{\cV}{\mathcal{V}}

\def \eps {\epsilon}

   \def\m{\mu}

\newcommand{\lev}{\left\langle}
\newcommand{\rev}{\right\rangle}

\title{Higher order regularity of nonlinear Fokker-Planck PDEs with respect to the measure component}

\author{Alvin Tse
  \thanks{This research benefited from the support of the ``Chaire Risques Financiers'', Fondation du Risque. \\ \,  \,  Corresponding e-mail: alvin.tse@enpc.fr
  }}

\affil{Universit\'{e} Paris-Est, Cermics (ENPC), INRIA, F-77455 Marne-la-Vall\'{e}e, France}
\date{}

\begin{document}	
\selectlanguage{english}
\maketitle

\begin{abstract}
In this article, we establish a general formula for higher order linear functional derivatives for the composition of an arbitrary smooth functional on the 1-Wasserstein space with the solution of a Fokker-Planck PDE. This formula has important links with the theory of propagation of chaos and mean-field games. \\ ${}$ \\
\begin{center}
 \textbf{R\'{e}sum\'{e}} \\
\end{center}
Dans cet article, nous établissons une formule générale pour les dérivées fonctionnelles linéaires d'ordre supérieur pour la composition d'une fonctionnelle régulière arbitraire sur l'espace 1-Wasserstein avec la solution d'une EDP de Fokker-Planck. Cette formule a des liens importants avec la théorie de la propagation du chaos et des jeux à champ moyen. \\
\end{abstract}




\section{Introduction}
Let $\cP(\bT^d)$ denote the 1-Wasserstein space of probability measures on $\bT^d$, where $\bT^d:= \bR^d/\bZ^d$ denotes the $d$-dimensional torus. In this paper, we consider nonlinear Fokker-Planck PDEs of the form
\begin{equation} \begin{cases}
     \partial_t m + \text{div} ( b(\cdot,m) m) - \Delta m = 0, \quad \quad t \in [0,T], \\
             m(0,\mu) = \mu,
    \end{cases} \label{eq: forward eqn }  \end{equation}
    for some function $b: \bT^d \times \cP(\bT^d) \to \bR^d$ and probability measure $\mu \in \cP(\bT^d)$. This type of equations has been a rich area of research in the last decades. The case in which $b$ does not depend on $m$ has been treated in most classical works, such as Chapter 6 of \cite{bogachev2015fokker}.  In \cite{barbu2018nonlinear}, this type of equations is considered to construct weak solutions to a class of distribution-dependent SDEs. The case corresponding to probability measures on the path space is considered in \cite{huang2017nonlinear}.
    
    Let $\Phi: \cP( \bT^d) \to \bR$ be a continuous function (w.r.t. the topology of  $\cP( \bT^d)$). This paper explores the smoothness w.r.t. the measure  component for function $\cU: [0,T] \times \cP( \bT^d) \to \bR$ defined by 
\begin{equation}
\cU(t, \mu) := \Phi(m(t, \mu)), \label{eq: defofflow}
\end{equation} 
under sufficient regularity of $b$ and $\Phi$. The notion of smoothness that we consider, i.e. the linear functional derivative,  is widely adopted in the literature of McKean-Vlasov equations and mean-field games, such as \cite{cardaliaguet2015master}, \cite{carmona2017probabilistic} and \cite{crisan2017smoothing}. A continuous function (w.r.t. the product topology of $\cP( \bT^d) \times \bT^d$) $\frac{\delta \cV}{\delta m}: \cP(\bT^d) \times \bT^d \to \bR$ is said to be the \emph{linear functional derivative} of $\cV: \cP(\bT^d) \to \bR$, if  for any $m, m' \in \cP(\bT^d)$,
 \begin{align}\label{eq de first order deriv}
\cV(m')- \cV(m) = \int_0^1 \int_{\bT^d} \frac{\delta \cV}{\delta m}( (1-s)m + sm',y) \, (m'-m)(dy) \, ds. 
 \end{align}
 We then introduce higher-order derivatives through iteration:
for any $m, m' \in \cP(\bT^d)$ and $y \in (\bT^d)^{p-1}$,
\begin{align} \label{eq de higher order deriv}
 \frac{\delta^{p-1} \cV}{\delta m^{p-1}}(m',y) - \frac{\delta^{p-1} \cV}{\delta m^{p-1}}(m,y) = \int_0^1 \int_{\bT^{d}} \frac{\delta^p \cV}{\delta m^p}((1-s)m +sm',y,y') \, (m'-m)( d y') \,d s,
\end{align}
provided that the $(p-1)$-th order derivative is well defined. These derivatives are defined up to an additive constant via \eqref{eq de first order deriv} and \eqref{eq de higher order deriv}. They are normalised by the convention 
\begin{equation}
    \intrd   \frac{\delta^p \cV}{\delta m^p}(m,y_1, \ldots, y_p) \, m(dy_i) = 0, \quad  i \in \{1, \ldots, p \}.
    \label{eq: normalisation linear functional deriatives}
\end{equation} 
The main result of this paper is Theorem \ref{main result}. The definitions of the assumptions are found in Section \ref{section assumptions}. The definitions of the higher-order Kolmogorov equations $m^{(\beta)}$ and the multi-indices $\Lambda \in e(\Lambda_k)$ can be found in \eqref{PDE linearised k} and Definitions \ref{multi indices delta k}- \ref{ def  Lambda k} respectively. 
\begin{theorem*}  
\emph{Let  $k \in \bN$. Assume \emph{\hregb{k+2}{k}},  \emph{\hlipb{k+1}{k}}, \emph{\hlipphi[k]} and \emph{\hregphi{k+2}{k}}.  
Then $\frac{\delta^k \cU}{\delta m^k}$ exists and is given by}
\begin{eqnarray}
     &&    \frac{\delta^k \cU}{\delta m^k} ( t,\mu)(z_1, \ldots , z_k)  \nonumber \\
     & = &  \sum_{\Lambda = \big( \hat{n}, (\beta_j), (\alpha_{i,j}) \big) \in e(\Lambda_k)} \bigg[ \frac{ \delta^{{{\hat{n}}}} \Phi}{ \delta m^{{{\hat{n}}}}}( m(t, \mu)) \bigg( m^{(\beta_1)} \Big(t, \mu, \delta_{z_{\alpha_{1,1}}}, \ldots, \delta_{z_{\alpha_{1,\beta_1}}} \Big), \ldots, m^{(\beta_{{\hat{n}}})} \Big(t, \mu, \delta_{z_{\alpha_{{\hat{n}},1}}}, \ldots, \delta_{z_{\alpha_{{\hat{n}},\beta_{\hat{n}}}}} \Big) \bigg) \bigg].   \nonumber \\
     && \label{k order full formula} 
       \end{eqnarray}
            \emph{In particular, if we also assume \emph{\hintphi{k+1}{k}}, then}
$$\sup_{z_1, \ldots, z_k \in \bT^d}  \sup_{\mu \in \cP(\bT^d)} \sup_{t \in [0,T]} \bigg|  \frac{\delta^k \cU}{\delta m^k} ( t,\mu)(z_1, \ldots , z_k) \bigg| <+ \infty. $$
\end{theorem*}    
 \subsection{Links of the main result with the theory of quantitative propagation of chaos} 
 This result has intricate links with the theory of McKean-Vlasov stochastic differential equations (MVSDEs) and mean-field optimal control.   Let us consider a probability space $(\Omega, \cF, \bP)$ equipped with a $d$-dimensional Brownian motion $W$. Denoting the law of random variable $\eta$ by $\rvlaw[\eta]$, we consider a $d$-dimensional MVSDE given by 
    \begin{equation} \begin{cases}
      X^{0, \eta}_t  = \eta + \int_0^t b(X^{0,\eta}_s, \rvlaw[X^{0, \eta}_s]) \, ds + \sqrt{2} W_t, \\
              \rvlaw[\eta]= \mu. 
    \end{cases} \label{eq:MVSDE} \end{equation}
Lipschitz condition on $b$ ensures uniqueness of the solution to \eqref{eq:MVSDE} (\cite{sznitman1991topics}) and it can be easily checked that in this case
$$ m(s, \mu) = \rvlaw[X^{0, \eta}_s].$$ 
MVSDEs provide a probabilistic
representation to the solutions of a class of nonlinear PDEs. A particular example of such nonlinear PDEs was first
studied by McKean (\cite{mckean1966class}). These equations describe the limiting behaviour of an individual particle evolving within
a large system of particles undergoing diffusive motion and interacting in a `mean-field' sense, as the population
size grows to infinity. More precisely, we consider the following system of particles,
\begin{equation} 
\begin{cases}
       Y^{i,N}_t= \eta_i + \int_0^t  b \big( Y^{i,N}_s, \mu^{N}_s \big)  \,ds + \sqrt{2} W^i_t, \quad 1 \leq i \leq N, \quad t \in [0,T], \\
               \mu^{N}_s := \frac{1}{N} \sum_{i=1}^N \delta_{Y^{i,N}_s},    \end{cases} \label{eq particles}
\end{equation}
where $W^{i},$ $1 \leq i \leq N$,  are independent $d$-dimensional Brownian motions and $\eta_i,$ $1 \leq i \leq N$, are i.i.d. random variables with the same distribution as $\eta$. A particular characteristic of the limiting behaviour of the system, is that any finite subset
of particles becomes asymptotically independent of each other. This phenomenon is known as \emph{propagation of chaos}. We refer the reader to \cite{gartner1988mckean,meleard1996asymptotic,sznitman1991topics} for the classical results in this direction and to \cite{bossy2011conditional,fournier2016propagation,jourdain2007nonlinear,lacker2018strong,mischler2013kac} for an  account (non-exhaustive) of recent results. Nonetheless, most results are only qualitative and do not give us a rate of convergence. 

For deterministic $\eta = c \in \bR^d$, it is shown in \cite{chassagneux2019weak} that under sufficient regularity of $b$ and $\Phi$, the weak error between the particle system \eqref{eq particles} and its mean-field limit 
\eqref{eq:MVSDE} is  given by
\begin{equation}  \bE[ \Phi(\mu^{N}_T)] - \Phi(\rvlaw[X^{0,\eta}_T]) =  \frac{1}{N} \int_0^T \bE \bigg[ \intrd  \text{Tr} \bigg( \partial_{y_{1}} \partial_{y_{2}} \frac{ \delta^2 \cU}{\delta m^2} (T-s, \mu^{N}_s)(z,z) \bigg) \, \mu^{N}_s(dz)  \bigg]  \, ds. \label{consequence master equation}  \end{equation}
(A more complicated formula is also given in \cite{chassagneux2019weak} for non-deterministic initial conditions.)  To obtain a full expansion of the form 
\[
\bE[ \Phi(\mu^{N}_T)] - \Phi(\rvlaw[X^{0,\eta}_T])= \sum_{j=1}^{k-1}\frac{C_j}{N^j} + O(\frac{1}{N^k}),
\]
for some positive constants $C_1, \ldots, C_{k-1}$ that do not depend on $N$, one would even need to consider higher order linear derivatives $\frac{ \delta^k \cU}{\delta m^k}$ (see \cite{chassagneux2019weak}). 

Note that in most practical applications, the test function $\Phi$ being considered is linear, therefore its linear derivatives have simple closed-form formulae.  In this case, the advantage of \eqref{k order full formula} is that it expresses $\frac{ \delta^k \cU}{\delta m^k}$ completely in terms of higher order Kolmogorov equations $m^{(\beta)}$, which are intrinsically Cauchy problems. 

Despite being out of the scope of this paper, we remark that it is not difficult to compute the expression for
\begin{equation}  \partial_{z_1} \ldots \partial_{z_k} \frac{ \delta^k \cU}{\delta m^k}(t,\mu)(z_1, \ldots, z_k) \label{ der linear der} \end{equation} 
by perturbing each of the measures $\mu_1, \ldots, \mu_{\beta}$ in $m^{(\beta)}(t, \mu, \mu_1, \ldots, \mu_{\beta})$. This is much simpler than the linearisation procedure performed in this paper, where we perturb measure $\mu$, which is more cumbersome and technical. Through more sophisticated techniques of global Schauder estimates, it should even be possible to obtain a control of \eqref{ der linear der} that decays over time $t$, which allows us to obtain a uniform estimate of propagation of chaos in $T$, by \eqref{consequence master equation}. This is a closely related research direction. 
\subsection{Main method of proof in this paper} 
The main idea of proof comes from \cite{cardaliaguet2015master}, based on their idea of `linearising' a forward-backward mean-field game system by perturbating the measure component.  Our strategy follows a similar argument as Proposition 3.4.3 and Corollary 3.4.4 in \cite{cardaliaguet2015master}. 

To explore regularity of \eqref{eq: forward eqn } along the measure component, we perturb probability measure $\mu \in \cP(\bT^d)$ along direction $\hmu \in \cP(\bT^d)$. Take any smooth test function $\phi: [0,T] \times \bT^d \to \bR$. We have
\begin{eqnarray}
&&  \int_{\bT^d} \phi(t,x) \big( m(t, (1- \eps) \mu + \eps \hmu ) \big) (dx) - \int_{\bT^d} \phi(0,x) \big(  (1- \eps) \mu + \eps \hmu \big) (dx) \nonumber \\
&=& \int_0^t  \int_{\bT^d} \partial_s \phi(s,x) \, \big( m(s, (1- \eps) \mu + \eps \hmu) \big) (dx) \,ds \nonumber \\ 
&& + \int_0^t \int_{\bT^d} \Delta \phi(s,x) \, \big( m(s, (1- \eps) \mu + \eps \hmu) \big) (dx) \,ds \nonumber \\
&& + \int_0^t \int_{\bT^d} \Big[ b \big( x, m(s, (1- \eps) \mu + \eps \hmu) \big) \cdot \nabla \phi(s,x) \Big]  \, \big( m(s, (1- \eps) \mu + \eps \hmu) \big) (dx) \,ds. \label{eq: fokker planck expanded} \end{eqnarray} 
We define
$$ m^{(1)} (s,\mu,\hat{\mu}):= \frac{d}{d \eps} \bigg|_{\eps=0^{+}}  m(s, (1- \eps) \mu + \eps \hat{\mu}) $$
in the sense of distributions. 
Then one should expect that
\begin{equation} \frac{d}{d \eps} \bigg|_{\eps=0^{+}} \Phi \big( m(s, (1- \eps) \mu + \eps \hat{\mu}) \big) = \int_{\bT^d} \frac{\delta \Phi}{\delta m}(m(s, \mu))(y) \, m^{(1)}(s,\mu,\hat{\mu}) (dy). \label{heuristics der}  \end{equation} 
(In particular, for the linear case when $m(s, \mu)= \mu$, we have 
$$ \frac{d}{d \eps} \bigg|_{\eps=0^{+}} \Phi \big(  (1- \eps) \mu + \eps \hat{\mu} \big) = \int_{\bT^d} \frac{\delta \Phi}{\delta m}( \mu)(y) \, (\hat{\mu} - \mu) (dy),$$
which is a consequence of the definition of the linear functional derivative.) 
Applying \eqref{heuristics der} to \eqref{eq: fokker planck expanded}, by differentiating \eqref{eq: fokker planck expanded} w.r.t. $\eps $ at $0$, we have
\begin{eqnarray}
&&  \int_{\bT^d} \phi(t,y)  \, m^{(1)}(t,\mu,\hat{\mu}) (dy) - \int_{\bT^d} \phi(0,y)  \, m^{(1)}(0,\mu,\hat{\mu}) (dy)  \nonumber \\
&= & \int_0^t \int_{\bT^d} \partial_s \phi(s,y)  m^{(1)}(s, \mu,\hat{\mu}) (dy)  \,ds  + \int_0^t \int_{\bT^d} \Delta \phi(s,y)  \, m^{(1)}(s, \mu,\hat{\mu}) (dy)  \,ds \nonumber \\
&& + \int_0^t \int_{\bT^d} \Big[ b \big( y, m(s, \mu) \big) \cdot \nabla \phi(s,y) \Big]  \,m^{(1)}(s,\mu,\hat{\mu}) (dy)  \,ds \nonumber \\
&& + \int_0^t \int_{\bT^d}  \int_{\bT^d} \bigg[ \frac{\delta b}{\delta m} \big( x, m(s,\mu) \big)(y) \cdot \nabla \phi(s,x) \bigg] \,  m^{(1)}(s,\mu,\hat{\mu}) (dy)   \, \big( m(s, \mu) \big) (dx) \,ds. \label{eq: derivative expansion}   \end{eqnarray}
Note that, in the distribution sense, \eqref{eq: derivative expansion} can be rewritten as the linearised forward Kolmogorov equation 
\begin{equation} 
\left\{
\begin{array}{rrl}
      {} &  \partial_t m^{(1)}(t,\mu,\hat{\mu}) + \text{div} ( b(\cdot,m(t,\mu)) m^{(1)}(t,\mu,\hat{\mu})) & \\
      & + \text{div}  \big(m(t,\mu) \frac{\delta b}{\delta m} (\cdot,m(t,\mu))(m^{(1)}(t,\mu,\hat{\mu})) \big) - \Delta m^{(1)} (t,\mu,\hat{\mu}) &= 0, \\
      & &  \\
      {} & m^{(1)}(0,\mu,\hat{\mu})& = \hat{\mu}-\mu. \\
\end{array} 
\right. 
  \label{PDE linearised} \end{equation}
This is what we expect by differentiating \eqref{eq: forward eqn } formally in $m$. To show that this is indeed the case, we consider the difference $m(t, \hat{\mu}) - m(t, \mu) - m^{(1)}(t, \mu, \hat{\mu})$ to prove differentiability of $m$ with respect to the measure.

We adopt the approach of Schauder theory and most of the results follow from Theorem \ref{thm schauder}, which is a fundamental result of Schauder estimates on the viscous transport equation. Based on Schauder theory, it is shown in Theorem \ref{thm m m-one}  that there exists some constant $C>0$ such that
    $$ \sup_{t \in [0,T]} \| m(t, \hat{\mu}) - m(t, \mu) - m^{(1)}(t, \mu, \hat{\mu}) \|_{-(n,\infty)} \leq C W_1( \mu, \hat{\mu})^2,$$ 
under the assumptions   {\regb[n]}, {\lipb[0]}, {\hlipphi[1]} and {\hregphi{n}{1}}, where $n \geq 2$. Therefore, we can show that
$$ \frac{d}{d \eps} \bigg|_{\eps=0^{+}} \Phi \big( m(t, (1- \eps) \mu + \eps \hat{\mu}) \big) = \int_{\bT^d} \frac{\delta \Phi}{\delta m}(m(t, \mu))(y) \, m^{(1)}(t,\mu,\hat{\mu}) (dy). $$ 

Nonetheless, to show that $\cU$ indeed has a linear functional derivative, we need to express the integral on the right hand side in terms of the signed measure $\hat{\mu}- \mu$. Here is where probability theory comes into action. For every $t \in [0,T]$ and $x \in \bR^d$, we consider the decoupled process $\{ X^{0,x,\mu}_u \}_{u \in [0,t]}$ defined by
\begin{equation} 
      X^{0,x,\mu}_s  = x + \int_0^s b( X^{0,x,\mu}_r, m(r,\mu)) \, dr + \sqrt{2} W_s, \quad \quad 0  \leq s \leq t. \label{decoupled paul eric} \end{equation}
For every $\xi: \bT^d \to \bR$ and $t \in [0,T]$, we define a function $v(\cdot, \cdot, \cdot; \xi ,t): [0,t] \times \bT^d \times \cP(\bT^d) \to \bR$ such that
$$ v(s,x, \mu; \xi ,t):= \bE \big[ \xi(X^{0,x, \mu}_t) \big| X^{0,x, \mu}_s = x \big],$$ 
which satisfies the backward Kolmogorov equation
\begin{equation*}
\begin{cases}
     \partial_s v(s,x, \mu) + b(x, m(s, \mu)) \cdot \nabla v(s,x,\mu) + \Delta v(s,x,\mu) =0, \\
             v(t,x, \mu)= \xi(x).
    \end{cases} 
    \end{equation*}
Note that
$$ v(0,x,\mu; \xi ,t) = \bE \big[ \xi(X^{0,x, \mu}_t) \big] $$ 
and therefore \footnote{Note that if the law of $\eta_1$ is equal to the law of $\eta_2$, then the law of  $X^{0, \eta_1}_t$ is  also equal to the law of $X^{0, \eta_2}_t$. Therefore, if we are only interested in the law of the process $X^{0,\eta}_t$, where $\eta$ is distributed as $\mu$, then it is proper to adopt the notation $X^{0,\mu}_t$.} 
$$ \int_{\bT^d} \xi(x) \, m(t,\mu)(dx) = \bE[ \xi(X^{0, \mu}_t)] =  \int_{\bT^d} v(0,x, \mu; \xi ,t) \, \mu(dx) .$$
By linearisating with respect to $\mu$ in the same way as \eqref{eq: fokker planck expanded} and \eqref{heuristics der}, we obtain that 
\begin{equation*}
\int_{\bT^d} \xi(x) \, m^{(1)}(t, \mu, \hat{\mu})(dx)   =   \int_{\bT^d}  \bigg[ v(0,x, \mu; \xi ,t) + \int_{\bT^d}  \frac{\delta v}{\delta m}(0,z, \mu,x; \xi ,t)   \, \mu(dz) \bigg] (\hat{\mu}- \mu)(dx). 
\end{equation*}
Consequently, by replacing $\xi$ by $\ld[\Phi](m(s,\mu))(\cdot)$, we can deduce from \eqref{heuristics der} the existence of the first order linear derivative of $\cU$. We repeat the same procedure for higher order linear derivatives of $\cU$. It is precisely this combination of forward and backward equations that allows us to prove existence of the linear derivatives of $\cU$. 
\subsection{Comparison with other approaches in the literature} 
There are various alternative methods for establishing smoothness of functions of the form \eqref{eq: defofflow} in the literature, all of which are probabilistic.

The method of Malliavin calculus is adopted in \cite{crisan2017smoothing}. That paper proves smoothness of $\cU$, for $\Phi$ being in the form
$$ \Phi (\mu) = \int_{\bR^d} \zeta(y) \,  \mu (dy), $$ 
where $\zeta: \bR^d \to \bR$ is infinitely differentiable with bounded partial derivatives. 

The method of parametrix is considered in  \cite{de2015strong}. We represent $\cU$  in terms of the transition density $p(s,\mu;t',y';t,y)$ of $X^{s, x, \mu}_t$ (defined above in \eqref{decoupled paul eric}). This method is applied to the case in which $b$ and $\Phi$ are of the form
    $$ b(x, \mu) = \varphi_2 \bigg( x, \int_{\bR^d} \varphi_1 (y) \mu (dy) \bigg) , \quad \quad  \Phi (\mu) = \int_{\bR^d} \zeta(y) \,  \mu (dy),$$ 
    for some functions  $\varphi_1: \bR^d  \to \bR$, $\varphi_2: \bR^d \times \bR \to \bR^d$ and $\zeta: \bR^d \to \bR$.
Nonetheless, it is not clear whether this method can be applied to $b$ and $\Phi$ with more general forms.  

Finally, a `variational' approach is adopted in \cite{buckdahn2017mean}. The core idea is to prove smoothness of $\cU$ by viewing the lift of $\cU$ (i.e. the map $Y \mapsto \cU(\rvlaw[Y])$) as a composition of the map $\eta \mapsto X^{0, \eta}_{t}$ and the lift of $\Phi$ (i.e. the map $Y \mapsto \Phi(\rvlaw[Y])$). In \cite{buckdahn2017mean}, the smoothness of $\cU$ is proven up to the second order, under fairly general conditions on $b$ and $\Phi$.  
\subsection{Notations and main assumptions} 
\subsubsection{Notations}
The scalar product between two vectors $a,b \in \bR^d$ is denoted by $a \cdot b$. $\cP(\bT^d)$ denotes the space of integrable probability measures and $W_1$ denotes the 1-Wasserstein distance, defined by
\begin{equation*}
W_1(\mu,\nu) :=  \inf_{\pi \in \Pi(\mu,\nu)}  \int_{\bT^d\times \bT^d} |x-y| \, \pi(dx,dy),
\end{equation*}
where $\Pi(\mu,\nu)$ denotes the set of {\em couplings} between $\mu$ and $\nu$, i.e. all measures on $\mathscr{B}(\bT^d\times \bT^d)$ 
such that $\pi(B \times \bT^d) = \mu(B)$ and $\pi(\bT^d \times B) = \nu(B)$ 
for every $B \in \mathscr B(\bT^d)$. 

 ${}$ \\ To write the norms of a Sobolev space $W^{n,\infty}(\bT^d)$ and its dual, we use the notations
$$\| \cdot \|_{{n,\infty}} := \| \cdot \|_{W^{n,\infty}(\bT^d)}, \quad \quad \quad \quad \| \cdot \|_{-(n,\infty)} := \| \cdot \|_{(W^{n,\infty}(\bT^d))'}.$$
Moreover, for dual elements with their arguments, we use the notation
$$  \lev \xi, \rho \rev_{{n,\infty}} := \lev \xi, \rho   \rev_{W^{n,\infty}(\bT^d), (W^{n,\infty}(\bT^d))'}.$$ 

 ${}$ \\  Denoting $W^{0,\infty}(\bT^d):= L^\infty(\bT^d)$, for any $f \in W^{n-1,\infty}(\bT^d, \bR^d)$  and $ \eta \in L^{\infty} ([0,T],  (W^{n-1,\infty}(\bT^d))')$, we use the notation 
\begin{equation}  \lev \xi, -\text{div} \big( \eta(t) f \big) \rev_{{n,\infty}} := \lev f \cdot \nabla \xi,  \eta(t) \rev_{{n-1,\infty}} , \quad \quad t \in [0,T], \quad \quad n \in \bN.    \label{eq xi q t def 2}  \end{equation}  
 ${}$ \\  $W^{0,n,\infty}([0,T] \times {\mathbb T}^d)$ denotes, for $n \geq 1$, 
the space of measurable functions $f : [0,T] \times {\mathbb T}^d \rightarrow {\mathbb R}$
with spatial generalized derivatives up to order $n$ that all belong to $L^\infty([0,T] \times {\mathbb T}^d)$. We define
$$ \| f \|_{0,n, \infty} := \sup_{t \in [0,T]} \| f(t, \cdot) \|_{n, \infty}. $$ 

 ${}$ \\ For functions $f=(f_1, \ldots, f_d): [0,T] \times \bT^d \to \bR^d$ such that each component function $f_i$ belongs to $W^{0,n,\infty}([0,T] \times {\mathbb T}^d)$, we write $f \in W^{0,n,\infty}([0,T] \times {\mathbb T}^d, \bR^d)$ with 
$$ \| f\|_{0,n, \infty} := \bigg( \sum_{i=1}^d \| f_i\|^2_{0,n, \infty} \bigg)^{1/2}. $$

 ${}$ \\ For any  signed measures $\mu_1, \ldots, \mu_n$, we write $\frac{\delta^n \Phi}{\delta m^n} (\mu)(\mu_1, \ldots, \mu_n) $ to denote
$$ \frac{\delta^n \Phi}{\delta m^n} (\mu)(\mu_1, \ldots, \mu_n):= \int_{\bT^d} \ldots \int_{\bT^d}  \frac{\delta^n \Phi}{\delta m^n} (\mu, x_1, \ldots, x_n) \, \mu_1(dx_1) \, \ldots \, \mu_n (dx_n),$$ 
if this iterated integral is well-defined.

 ${}$ \\  Unless otherwise specified, $C$ is a constant that only depends on $n$,  $k$, $T$, $b$ and $\Phi$, whose value varies from line to line.   
\subsubsection{Main assumptions} \label{section assumptions}

Throughout this work, we work with the following assumptions on $b=(b_i)_{1 \leq i \leq d}$ and $\Phi$.  \hregb{n}{k} denotes the condition that, for each $i \in \{1, \ldots, d\}$, $\ell \in \{1, \ldots, k \}$, 
\begin{multline}
\frac{\delta^{\ell} b_i}{\delta m^{\ell}} \text{ exists and satisfies }   \sup_{m \in \cP(\bT^d)}   \| b_i(\cdot, m) \|_{n, \infty} + \sup_{x \in \bT^d} \sup_{m \in \cP(\bT^d)} \bigg\| \ld[b_i] (x, m)(\cdot) \bigg\|_{n, \infty} \\
+  \sup_{m \in \cP(\bT^d)} \sup_{\substack{ \|m_1\|_{-(n, \infty)}, \ldots, \| m_{\ell}\|_{-(n, \infty)} \leq 1}} \bigg\|  \frac{\delta^{\ell} b_i}{\delta m^{\ell}} (\cdot,m)(m_1, \ldots, m_{\ell}) \bigg\|_{n, \infty} < + \infty. 
\label{eq:Int-b-n-k} \tag{Int-$b$-($n,k$)}
\end{multline} 
\hlipb{n}{k} denotes the condition that, for each $i \in \{1, \ldots, d\}$, $\ell \in \{1, \ldots, k \}$, 
\[
\frac{\delta^{\ell} b_i}{\delta m^{\ell}} \text{ exists and satisfies } \quad \text{Lip}_{n}\bigg(\frac{\delta^{\ell} b_i}{\delta m^{\ell}} \bigg) <+ \infty,
\label{eq:Lip-b-n-k} \tag{Lip-$b$-($n,k$)}
\]
where
\begin{eqnarray}
&&  \text{Lip}_{n}\Big(\frac{\delta^{\ell} b_i}{\delta m^{\ell}} \Big)  \nonumber \\
& := &  \sup_{\substack{ \|m_1\|_{-(\na)}, \ldots, \| m_{\ell}\|_{-(\na)} \leq 1}}  \sup_{\substack{\mu_1 \neq \mu_2\\ \mu_1, \mu_2 \in \cP(\bT^d)}} \big(W_1(\mu_1,\mu_2) \big)^{-1} \bigg\| \frac{\delta^{\ell} b_i}{\delta m^{\ell}}(\cdot,\mu_1)(m_1, \ldots, m_{\ell}) - \frac{\delta^{\ell} b_i}{\delta m^{\ell}}(\cdot,\mu_2)(m_1, \ldots, m_{\ell}) \bigg\|_{n, \infty }. \nonumber 
\end{eqnarray}
For the test function $\Phi: \cP(\bT^d) \to \bR$, we shall impose the following assumptions. \hlipphi[k] denotes the condition that, for each $\ell \in \{1, \ldots, k\}$,
\[ \frac{\delta^{\ell} \Phi}{\delta m^{\ell}} \text{ exists and satisfies } \quad   \text{Lip}\bigg(\partial_{x_{\ell}} \frac{\delta^{\ell} \Phi}{\delta m^{\ell}} \bigg) < +\infty, \label{eq:TReg-Phi} \tag{TLip-$\Phi$-($k$)}
\]
where
\begin{eqnarray}\text{Lip}\bigg(\partial_{x_{\ell}} \frac{\delta^{\ell} \Phi}{\delta m^{\ell}} \bigg) & := & \sup_{x_1, \ldots, x_{\ell} \in \bT^d} \sup_{\substack{\mu_1 \neq \mu_2\\ \mu_1, \mu_2 \in \cP(\bT^d)}}  \big(W_1({\mu}_1,{\mu}_2) \big)^{-1} \bigg| \partial_{x_{\ell}} \frac{\delta^{\ell} \Phi}{\delta m^{\ell}}({\mu}_1)(x_1, \ldots, x_{\ell}) - \partial_{x_{\ell}} \frac{\delta^{\ell} \Phi}{\delta m^{\ell}}({\mu}_2)(x_1, \ldots, x_{\ell}) \bigg|. \nonumber 
\end{eqnarray}
\hregphi{n}{k} denotes the condition that, for each $\ell \in \{1, \ldots, k\}$ and $i \in \{1, \ldots, \ell\}$, 
\[\frac{\delta^{\ell} \Phi}{\delta m^{\ell}}  \text{ exists and }  \sup_{y_1, \ldots, y_{i-1}, y_{i+1}, \ldots, y_{ \ell} \in \bT^d} \sup_{m \in \cP(\bT^d)} \bigg\| \frac{\delta^{\ell} \Phi}{\delta m^{\ell}} (m, y_1, \ldots, y_{i-1}, \cdot, y_{i+1} , \ldots, y_{\ell}) \bigg\|_{n,\infty} < +\infty. \label{eq:hReg-Phi-n-k} \tag{TReg-$\Phi$-($n,k$)}
\]
Finally, \hintphi{n}{k} denotes the integrability condition that, for each $\ell \in \{1, \ldots, k\}$,
\[
\frac{\delta^{\ell} \Phi}{\delta m^{\ell}} \text{ exists and satisfies } \sup_{m \in \cP(\bT^d)} \sup_{\substack{ \|m_1\|_{-(\na)}, \ldots, \| m_{\ell}\|_{-(\na)} \leq 1}} \bigg|  \frac{\delta^{\ell} \Phi}{\delta m^{\ell}} (m)(m_1, \ldots, m_{\ell}) \bigg| < + \infty. \label{eq:TInt-Phi-n-k} \tag{TInt-$\Phi$-($n,k$)}
\]
\subsection{Practical examples of our model} \label{examples}
We now give a  result of a class of drift terms $b$ and test functions $\Phi$ that satisfies the above assumptions, followed by practical examples of our model. 
\begin{theorem} \label{thm example}
Let $n \in \bN$. Suppose that for each $i \in \{1, \ldots, d\}$, $F_i: \bT^d \times \bT^d \to \bR$ belongs to $W^{\na}(\bT^d \times \bT^d)$ and that $G: \bT^d \to \bR$ belongs to $W^{\na}(\bT^d)$. We then define functions $b_i: \bT^d \times \cP(\bT^d) \to \bR$ and $ \Phi: \cP(\bT^d) \to \bR$ by
$$ b_i(x, \mu):= \intrd F_i(x, y) \, \mu(dy),  \quad \quad \Phi(\mu):= \intrd G(y) \, \mu(dy).$$ 
Then $b$ satisfies \emph{\hregb{n}{k}} and \emph{\hlipb{n}{k}}. Moreover, $\Phi $ satisfies \emph{\hlipphi[k]}, \emph{\hregphi{n}{k}} and \emph{\hintphi{n}{k}}.
\end{theorem}
\begin{proof}
Let $k \in \bN$ be arbitrary. Let 
$$\Theta := \bigg| \bigg\{ \text{multi-index } \bm{\beta} \in \bN^d \,  \bigg| \,   |\bm{\beta}| \leq n \bigg\} \bigg| . $$ 
It can be shown easily by the definition of linear functional derivatives (along with the condition of normalisation) that
$$\frac{ \delta^{k} b_i}{\delta m^k} (x, \mu)(y_1, \ldots, y_k)= (-1)^k \bigg( \intrd F_i(x, y) \, \mu(dy) - F_i(x, y_k) \bigg).$$ 
It can be easily checked that
$$ \sup_{m \in \cP(\bT^d)}   \| b_i(\cdot, m) \|_{\na} \leq  \Theta \| F_i \|_{\na} \, ,$$
$$ \sup_{x \in \bT^d} \sup_{m \in \cP(\bT^d)} \bigg\| \ld[b_i] (x, m)(\cdot) \bigg\|_{\na} \leq 2 \Theta \| F_i \|_{\na}$$ 
and
\begin{eqnarray}
&& \sup_{m \in \cP(\bT^d)} \sup_{m_1, \ldots, m_{k} \in (W^{n,\infty})'} \bigg\|  \frac{\delta^{k} b_i}{\delta m^{k}} (\cdot,m)(m_1, \ldots, m_{k}) \bigg\|_{\na}  \nonumber \\
& \leq &  \Theta \Big(  \| F_i \|_{\na} +  \| F_i \|_{\na} \| m_{k} \|_{-(\na)} \Big) \prod_{\ell=1}^{k-1} \| m_{\ell} \|_{-(\na)}. \nonumber \end{eqnarray}
Moreover, by the Kantorovich Rubinstein duality (see Remark 6.5 in \cite{villani}),
\begin{eqnarray}
&& \sup_{m_1, \ldots, m_{k} \in  (W^{n,\infty})'}  \sup_{\substack{\mu_1 \neq \mu_2\\ \mu_1, \mu_2 \in \cP(\bT^d)}} \big(W_1(\mu_1,\mu_2) \big)^{-1} \bigg\| \frac{\delta^{k} b_i}{\delta m^{k}}(\cdot,\mu_1)(m_1, \ldots, m_{k}) - \frac{\delta^{k} b_i}{\delta m^{k}}(\cdot,\mu_2)(m_1, \ldots, m_{k}) \bigg\|_{\na} \nonumber \\
& \leq & \Theta \| F_i \|_{\na} \prod_{\ell=1}^{k-1} \| m_{\ell} \|_{-(\na)}. \nonumber 
\end{eqnarray}
Similarly,
$$\frac{ \delta^{k} \Phi}{\delta m^k} ( \mu)(y_1, \ldots, y_k)= (-1)^k \bigg( \intrd G( y) \, \mu(dy) - G( y_k) \bigg).$$ 
This allows us to show that
$$\text{Lip}\bigg(\partial_{x_{k}} \frac{\delta^{k} \Phi}{\delta m^{k}} \bigg) \leq \| G\|_{\na} ,$$ 
$$\sup_{y_1, \ldots, y_{i-1}, y_{i+1}, \ldots, y_{ k} \in \bT^d} \sup_{m \in \cP(\bT^d)} \bigg\| \frac{\delta^{k} \Phi}{\delta m^{k}} (m, y_1, \ldots, y_{i-1}, \cdot, y_{i+1} , \ldots, y_{k}) \bigg\|_{\na} \leq 2 \Theta \| G\|_{\na} , $$ 
 $$ \sup_{m \in \cP(\bT^d)} \sup_{m_1, \ldots, m_{k} \in  (W^{n,\infty})'} \bigg|  \frac{\delta^{k} \Phi}{\delta m^{k}} (m)(m_1, \ldots, m_{k}) \bigg| \leq 2 \| G\|_{\na} \prod_{\ell=1}^{k} \| m_{\ell} \|_{-(\na)}.$$ 
 These calculations show that $b$ and $\Phi$ satisfy the aforementioned regularity properties in the theorem. Note that $k$ is arbitrary in $\bN$, since the dependence on measure is linear for functions $b$ and $\sigma$. 
\end{proof}
\begin{example}[Kuramoto model]
The Kuramoto model is used to describe the behaviour of synchronization for a large set of coupled oscillators and is defined in dimension $d=1$ (see, e.g., \cite{bertini2010dynamical}):
$$ b(x,\mu) := - \frac{K}{\pi} \int_{\bT} \sin (2 \pi (x-y)) \, \mu(dy).$$  
\end{example}
\begin{example}[Aggregation models]
Aggregation models are commonly used in the analysis of mean-field models in biology, ecology, for space homogeneous
granular media (see \cite{bolley2010trend,bolley2007quantitative,cattiaux2008probabilistic,jabin2017mean,malrieu2001logarithmic}). In such models, the drift term $b$ typically takes the form
$$ b(x,\mu) := - \int_{\bT^d} \nabla W(x-y) \, \mu(dy) - \nabla V (x),$$ 
for some smooth functions $V,W: \bT^d \to \bR$. According to Theorem \ref{thm example}, our analysis would be applicable to functions $V, W \in W^{n,\infty}$, where $n \geq 2$. 
\end{example}
\section{Regularity of first order linear derivative in measure of \texorpdfstring{$\cU$}{U}}

\subsection{Analysis of the forward Kolmogorov equation}
The first step in the analysis of PDEs is the regularity of $m$. The following result concerns regularity of \eqref{eq: forward eqn } and is standard in the literature. 
\begin{lemma}
Suppose that $b$ is jointly Lipschitz continuous in the  space and measure variables w.r.t. the Euclidean and $W_1$ metrics. Then \eqref{eq: forward eqn } has a unique solution and satisfies 
\begin{equation} 
     \sup_{t \in [0,T]} W_1(m(t, \mu_1),m(t, \mu_2)) \leq C W_1(\mu_1, \mu_2), 
 \label{eq: bound buckdahn} \end{equation}
for some constant $C>0$.
\end{lemma}
\begin{proof}
The fact that \eqref{eq: forward eqn } has a unique solution follows from the strong uniqueness of \eqref{eq:MVSDE}, by Theorem 1.1 of \cite{sznitman1991topics}. The estimate follows from the proof of Lemma 3.1 in \cite{buckdahn2017mean}. 
\end{proof} 

The following result is a modified version of Proposition 3.4.3 in \cite{cardaliaguet2015master} from H\"{o}lder spaces to Sobolev spaces. 

\begin{proposition}
\label{thm schauder}
Let $n \geq 1$, $f \in  W^{0,n-1,\infty}([0,T] \times {\mathbb T}^d)$ and $g \in W^{0,n-1,\infty}([0,T] \times {\mathbb T}^d, \bR^d)$. 
Then, for any $z_{T} \in  W^{n, \infty}({\mathbb T}^d)$, the Cauchy problem
    \begin{equation}  \begin{cases}
    \partial_t z + \Delta z + g(t,x) \cdot \nabla z = f(t,x), \quad \quad (t,x) \in [0,T] \times \bT^d,\\
            z(T,x) = z_{T}(x),
    \end{cases}   \label{cauchy w:appendix:2} 
    \end{equation}
has a unique solution
in the following space:
\begin{equation*}
\begin{split}
{\mathcal V} &= \Bigl\{ z \in {\mathcal C}([0,T] \times {\mathbb T}^d,{\mathbb R}) \cap {\mathcal C}^{0,1}([0,T) \times {\mathbb T}^d ) \cap W^{1,2,d+1}_{\textrm{\rm loc}}([0,T) \times {\mathbb T}^d) \, \bigg| 
\\
&\qquad \exists \gamma >0: \quad \sup_{(t,x) \in [0,T) \times {\mathbb T}^d} \bigl( \vert z(t,x) \vert + 
(T-t)^{1/2-\gamma} \vert \nabla_{x} z(t,x) \vert \bigr) < \infty \Bigr\},
\end{split}
\end{equation*}
where 
${\mathcal C}^{0,1}([0,T) \times {\mathbb T}^d,{\mathbb R})$ is the space of
real-valued functions $z$ (on $[0,T) \times {\mathbb T}^d$) that are continuous in time and space, differentiable in 
space, and the derivative of which is continuous in time and space, 
and where 
$W^{1,2,d+1}_{\textrm{\rm loc}}([0,T) \times {\mathbb T}^d)$
is the space of  functions $z$ such that $\vert z\vert$, 
$\vert \nabla_{x} z \vert$, $\vert \nabla_{x}^2 z \vert$ 
and $\vert \partial_{t} z \vert$ belong to $L^{d+1}_{\textrm{\rm loc}}([0,T) \times {\mathbb T}^d)$. 
The unique solution satisfies
\begin{equation*}
\sup_{t \in [0,T]} \| z(t,\cdot) \|_{n,\infty} \leq C \Bigl( \| z_{T} \|_{n,\infty} + 
\sup_{t \in [0,T]} \|f(t,\cdot) \|_{n-1,\infty} \Bigr),
\end{equation*}
where $C$ only depends on $\| g \|_{n-1,\infty}$. 
\end{proposition}

\begin{proof}
\textit{First Step.}
We start with uniqueness. Uniqueness of a solution (in ${\mathcal V}$) is a trivial consequence of the solvability of the SDE:
\begin{equation}
\label{eq:proof:prop:6:2}
dX_{t} = g(t,X_{t}) dt + dB_{t},
\end{equation}
and, then, of It\^o-Krylov's formula (see for instance \cite{krylov2008controlled}), which guarantees that 
\begin{equation}
\label{eq:proof:prop:6:2:b}
z(t,x) = {\mathbb E} \biggl[ z_{T}(X_{T}) + \int_{t}^T f(s,X_{s}) ds \, \big\vert \, X_{t} = x
\biggr].
\end{equation}
Notice that It\^o's formula does not suffice since the solution may just have
first order $t$-derivative and 
second order $x$-derivatives in $L^{d+1}$. Obviously, It\^o-Krylov's formula here applies because of the non-degeneracy of the noise.  
\vskip 5pt

\textit{Second Step.}
Existence of a solution with generalized second order derivatives 
is a well known fact in the literature. The main reference is the monograph of 
Ladyzenskaja et al., \cite{lady_book}; a more precise application of the results of 
\cite{lady_book}
to our setting may be found in  
\cite{DelarueGuatteri}, see Theorem 2.1 therein. 
The latter says that existence of a solution hold in the space
${\mathcal V}$ defined in the statement. 
\vskip 5pt

\textit{Third Step.}
Now, the main point is to prove that the solution satisfies the required bounds. 
By mollifying the coefficients $f$ and $g$ in space (using a standard convolution argument), 
we may easily assume that the coefficients $f$ and $g$ are smooth in space, 
and that their derivatives up the order $n-1$ satisfy the same Lipschitz bounds as the 
original coefficients. If we can prove that the solution associated with the equation with mollified coefficients 
satisfies the inequality announced in the statement, with a constant $C$ therein that remains uniform along 
the mollification, then we are done: it suffices to observe that
the solution associated with the mollified equation
converges (as the mollification parameter tends to $0$) to the original $z$ by passing to the limit along the 
stochastic representation (based upon (\ref{eq:proof:prop:6:2})--(\ref{eq:proof:prop:6:2:b})). 

So, from now on, we assume that the coefficients $f$ and $g$ are smooth in space, 
and that their derivatives up the order $n-1$ satisfy the same Lipschitz bounds as the 
original coefficients. 
The key point is then to observe that we can differentiate with respect to $x$ in 
the representation formula 
\eqref{eq:proof:prop:6:2:b}, since 
the solution to 
\eqref{eq:proof:prop:6:2} generates a smooth flow (see for instance \cite{Kunita}). 
As a by-product, we deduce that, for any $k \geq 1$, 
\begin{equation*}
\sup_{t \in [0,T]} \| \nabla^k z(t,\cdot)\|_{\infty} < \infty.
\end{equation*}
Obviously, the bound of the above left-hand side depends on the (additional) smoothness
of $f$ and $g$. 
Now, by expanding $(z(T-s,x+B_{s-t}))_{t \leq s \leq T}$ by means of It\^o's formula, we get that 
\begin{equation}
\label{eq:representation:z}
\begin{split}
z(t,x) &= \int_{{\mathbb T}^d} z_{T}(y) p(T-t,y-x) dy + 
 \int_{t}^T \int_{{\mathbb T}^d} g(s,y) \cdot \nabla z(s,y) p(s-t,x-y) ds dy
\\
&\hspace{15pt}
+ \int_{t}^T \int_{{\mathbb T}^d} f(s,y) p(s-t,x-y) ds dy,  
\end{split}
\end{equation}
where $p$ is the standard heat kernel. {We know from}
Theorem 11 in \cite[Chapter 1]{friedman2008partial} that there exists a bounded density
 $g$
on the torus such that, for any $0 \leq s < s' $ with $s'-s \leq 1$, 
\begin{equation}
\label{eq:gradient:transition:density}
\bigl\vert \nabla_{x} p(s,s',x,x') \bigr\vert \leq C (s'-s)^{-(d+1)/2}
g \Bigl( \frac{x'-x}{s'-s} \Bigr),
\end{equation}
for a constant $C$ only depending on the bound of $V$. Taking the derivative with respect to $x$, 
\begin{equation*}
\begin{split}
\bigl\| \nabla z(t,\cdot) \bigr\|_{\infty}
\leq  \| \nabla z_{T} \|_{\infty} +C  \int_{t}^T (s-t)^{-1/2} \Bigl( \| f(s,\cdot)\|_{\infty} +  \bigl\| \nabla z(s,\cdot) \|_{\infty} \Bigr) ds,
\end{split}
\end{equation*}
where the constant $C$ depends on $\| g \|_{\infty}$ (and is allowed to vary from line to line). 
By a standard variant of Gronwall's lemma (see for instance 
\cite[Lemma 7.1.1 and Exercise 1]{Henry_book}), we get
\begin{equation*}
\bigl\| \nabla z(t,\cdot) \bigr\|_{\infty}
\leq  C \Bigl( \| \nabla z_{T} \|_{\infty} +  \| f\|_{\infty}   \Bigr),
\end{equation*}
which is  exactly the announced result when $n=1$. 
Differentiating twice 
\eqref{eq:representation:z}
(hence differentiating twice the heat kernel in the right-hand side),
performing an integration by parts in 
the resulting second and third terms in the right-hand side of 
\eqref{eq:representation:z}
and 
eventually 
plugging the above bound 
in the resulting 
formula, we then get 
\begin{equation*}
\begin{split}
\bigl\| \nabla^2 z(t,\cdot) \bigr\|_{\infty}
\leq  \| \nabla^2 z_{T} \|_{\infty} +C  \int_{t}^S (s-t)^{-1/2} \Bigl( \| f(s,\cdot)\|_{1,\infty} +  \bigl\| \nabla z(s,\cdot) \|_{\infty}
+ \bigl\| \nabla^2 z(s,\cdot) \|_{\infty} \Bigr) ds,
\end{split}
\end{equation*}
where $C$ now depends on $\| g \|_{1,\infty}$, 
and in turn
\begin{equation*}
\begin{split}
\bigl\| \nabla^2 z(t,\cdot) \bigr\|_{\infty}
\leq  C \| z_{T} \|_{2,\infty} +C  \int_{t}^S (s-t)^{-1/2} \Bigl( \| f(s,\cdot)\|_{1,\infty} 
+ \bigl\| \nabla^2 z(s,\cdot) \|_{\infty} \Bigr) ds,
\end{split}
\end{equation*}
which yields, by the same variant of Gronwall's lemma, 
\begin{equation*}
\begin{split}
\sup_{t \in [0,S]}\bigl\| \nabla^2 z(t,\cdot) \bigr\|_{\infty}
\leq  C \Bigl( \| z_{T} \|_{2,\infty} +  \| f \|_{1,\infty}   \Bigr).
\end{split}
\end{equation*}
This is exactly the desired result when $n=2$. 

Now, we can iterate 
by induction, assuming that the result holds true for a given $k \in \{2,\cdots,n-1\}$. It suffices to take $k+1$ derivatives (in $x$) in the left-hand side of
\eqref{eq:representation:z}
and 
then 
to use an integration by parts to pass 
 $k$ derivatives from the heat kernel onto 
 $f$ and $g \cdot \nabla z$ 
  in 
the resulting second and third terms in the right-hand side of 
\eqref{eq:representation:z}. We then get 
\begin{equation*}
\begin{split}
\bigl\| \nabla^{k+1} z(t,\cdot) \bigr\|_{\infty}
\leq  \| z_{T} \|_{k+1,\infty} +C  \int_{t}^T (s-t)^{-1/2} \Bigl( \bigl\| f(s,\cdot) \bigr\|_{k,\infty} + 
 \bigl\| \nabla^{k+1} z(s,\cdot) \bigr\|_{\infty}
+ \bigl\| z(s,\cdot) \bigr\|_{k,\infty} \Bigr) ds,
\end{split}
\end{equation*}
where $C$ now depends on $\|g \|_{k,\infty}$. 
Plugging the bound, we have, for 
$\| z(s,\cdot) \|_{k,\infty}$ (as given by the induction assumption), 
we get 
\begin{equation*}
\begin{split}
\bigl\| \nabla^{k+1} z(t,\cdot) \bigr\|_{\infty}
\leq  \| z_{T} \|_{k+1,\infty} +C  \int_{t}^T (s-t)^{-1/2} \Bigl( \bigl\| f(s,\cdot) \bigr\|_{k,\infty} + 
 \bigl\| \nabla^{k+1} z(s,\cdot) \bigr\|_{\infty} \Bigr) ds,
\end{split}
\end{equation*}
and, then, refined Gronwall's lemma applies as before. 
\end{proof}

The core analysis of forward Kolmogorov equations depends heavily on the following fact.  The main ideas of the proof follow from the proof of Lemma 3.3.1 in \cite{cardaliaguet2015master}.
\begin{theorem}[Bound for forward Kolmogorov equations] \label{lions paper pde result} Let $n \geq 1$ and $q_0 \in ( W^{n,\infty}(\mathbb{T}^d) )'$. Assume \emph{\regb[n]}.
Let $ r \in L^{\infty} \big([0,T], ( W^{n, \infty}(\mathbb{T}^d) )' \big)$.  Then the Cauchy problem defined by
\begin{equation}  \begin{cases}
    \partial_t q(t) - \Delta q(t) + \text{\emph{div}} \big( b(\cdot,m(t,\mu)) q(t) \big) + \text{\emph{div}} \Big(  m(t, \mu)\frac{\delta b}{\delta m}( \cdot, m(t,\mu) )(q(t))\Big) -  r(t)  = 0, \\
             q(0) = q_0,
    \end{cases}   \label{eq q} \end{equation}
   interpreted as 
    \begin{eqnarray}
&&  \int_{\bT^d} \phi(t,y)  \, q(t) (dy) - \int_{\bT^d} \phi(0,y)  \,  \, q(0) (dy)  \nonumber \\
& = & \int_0^t \int_{\bT^d} \partial_s \phi(s,y)  \, q(s)  (dy)  \,ds + \int_0^t \int_{\bT^d} \Delta \phi(s,y)  \, q(s) (dy)  \,ds \nonumber \\
&& + \int_0^t \int_{\bT^d} \Big[ b \big( y, m(s, \mu) \big) \cdot \nabla \phi(s,y) \Big]  \,q(s) (dy)  \,ds \nonumber \\
&& +  \int_0^t \int_{\bT^d}  \int_{\bT^d} \bigg[  \frac{\delta b}{\delta m} \big( x, m(s,\mu) \big)(y) \cdot \nabla \phi(s,x) \bigg] \,  q(s) (dy)   \,  m(s, \mu)  (dx) \,ds \nonumber  \\
&& + \int_0^t \lev \phi(s,\cdot), r(s)   \rev_{n, \infty} \, ds, \nonumber 
\end{eqnarray}
for each $\phi \in C^{\infty}([0,T] \times \bT^d)$,
 has a unique solution in $L^{\infty} \big([0,T], (W^{n, \infty}(\mathbb{T}^d))' \big)$ such that
$$  \sup_{t \in [0,T]}  \| q(t) \|_{-(n,\infty)} \leq C \bigg( \| q_0 \|_{-(n,\infty)}  + \sup_{t \in [0,T]} \| r (t) \|_{-(n,\infty)} \bigg) , $$ 
for some constant $C>0$. \end{theorem} 
\begin{proof}
We consider the space $X:= C^{\beta}([0,T], (W^{\na}(\bT^d))')$, where $\beta \in (0, \frac{1}{2})$. We recall that the norm of $X$ is given by
$$ \| q\|_X := \sup_{t \in [0,T]} \| q(t) \|_{-(\na)} + \sup_{t \neq t'} \frac{\|q(t) - q(t') \|_{-(\na)} }{|t-t'|^{\beta}} .$$ 
For $q \in X$, we consider the Cauchy problem
\begin{equation}  \begin{cases}
    \partial_t \tilde{q}(t) - \Delta \tilde{q}(t) + \text{div} \big( b(\cdot,m(t,\mu)) \tilde{q}(t) \big) + \text{div}  \Big(  m(t, \mu)\frac{\delta b}{\delta m}( \cdot, m(t,\mu) )(q(t)) \Big) -   r(t)  = 0, \\
             \tilde{q}(0) = q_0.
    \end{cases}   \label{q prime} \end{equation}
By Schauder estimates, setting $T(q):= \tilde{q}$ defines a continuous and compact map $T: X \to X$. (See Step 1 in the proof of Lemma 3.3.1 in \cite{cardaliaguet2015master}). We show the existence of solution to \eqref{eq q} by applying the Leray-Schauder theorem, i.e. by showing that the set
$$ X_0:= \big\{ q \in X \, \, \big| \, \, q = \sigma Tq \text{ for some } \sigma \in [0,1] \big\}$$ 
is bounded. To this end, we pick an arbitrary $q \in X_0$, which satisfies the Cauchy problem
\begin{equation}  \begin{cases}
    \partial_t {q}(t) - \Delta {q}(t) + \text{div} \big( b(\cdot,m(t,\mu)) {q}(t) \big) - \sigma \Big( \text{div}  \Big(  m(t, \mu)\frac{\delta b}{\delta m}( \cdot, m(t,\mu) )(q(t)) \Big) -   r(t) \Big) = 0, \\
            {q}(0) = \sigma q_0.
    \end{cases}   \label{q X0} \end{equation}
The estimates rely on the classical argument of duality pairing. Fix $t \in [0,T]$ and $\xi \in W^{\na} (\bT^d)$. Let $w$ be the solution to the Cauchy problem
\begin{equation}  \begin{cases}
    \partial_s w + \Delta w + b(x, m(s,\mu)) \cdot \nabla w = 0, \quad \quad (s,x) \in [0,t] \times \bT^d, \\
            w(t,x) = \xi(x).
    \end{cases}   \label{cauchy w} \end{equation}
    By Theorem \ref{thm schauder},  $w$ satisfies
    \begin{equation}
        \sup_{s \in [0,t]} \| w(s,\cdot) \|_{\na}    \leq C \| \xi \|_{ \na}. \label{eq: bound w schauder} 
    \end{equation}
    By the definition of \eqref{q X0}, we have
\begin{eqnarray}
&&  \int_{\bT^d} w(t,y)  \, q(t) (dy) - \sigma \int_{\bT^d} w(0,y)  \,  \, q(0) (dy)  \nonumber \\
& = & \int_0^t \int_{\bT^d} \partial_s w(s,y)  \, q(s)  (dy)  \,ds + \int_0^t \int_{\bT^d} \Delta w(s,y)  \, q(s) (dy)  \,ds \nonumber \\
&& + \int_0^t \int_{\bT^d} \Big[ b \big( y, m(s, \mu) \big) \cdot \nabla w(s,y) \Big]  \,q(s) (dy)  \,ds \nonumber \\
&& + \sigma \int_0^t \int_{\bT^d}  \int_{\bT^d} \bigg[  \frac{\delta b}{\delta m} \big( x, m(s,\mu) \big)(y) \cdot \nabla w(s,x) \bigg] \,  q(s) (dy)   \,  m(s, \mu)  (dx) \,ds \nonumber  \\
&& + \sigma \int_0^t \lev w(s,\cdot), r(s)   \rev_{n,\infty} \, ds. \nonumber 
\end{eqnarray}
Therefore, by \eqref{cauchy w}, 
\begin{eqnarray}
&&  \int_{\bT^d} \xi(y) \, q(t) (dy)  \nonumber \\
& = & \sigma \int_{\bT^d} w(0,y)  \,  \, q(0) (dy)  + \sigma
 \int_0^t \int_{\bT^d}  \int_{\bT^d} \bigg[  \frac{\delta b}{\delta m} \big( x, m(s,\mu) \big)(y) \cdot \nabla w(s,x) \bigg] \,  q(s) (dy)   \,  m(s, \mu)  (dx) \,ds \nonumber  \\
&& + \sigma \int_0^t \lev w(s,\cdot), r(s)   \rev_{n,\infty} \, ds. \label{eq: xi q equality} 
\end{eqnarray}
We now estimate each of the three terms on the right hand side by \eqref{eq: bound w schauder}.  Firstly,
\begin{equation} \bigg|  \sigma \int_{\bT^d} w(0,y)  \,  \, q(0) (dy)  \bigg|  \leq \| w(0, \cdot) \|_{n,\infty} \| q(0)\|_{-(n,\infty)} \leq C \| \xi \|_{n,\infty} \| q(0) \|_{-(n,\infty)}. \label{eq: bound q est 1} \end{equation} 
By \eqref{eq: bound w schauder} and \regb[n], we  obtain the estimate
\begin{eqnarray}
&& \bigg| \sigma \int_0^t \int_{\bT^d}  \int_{\bT^d} \bigg[  \frac{\delta b}{\delta m} \big( x, m(s,\mu) \big)(y) \cdot \nabla w(s,x) \bigg] \,  q(s) (dy)   \,  m(s, \mu)  (dx) \,ds \bigg| \nonumber \\
& \leq & \int_0^t \sup_{x \in \bT^d} \bigg| \intrd \bigg[ \frac{\delta b}{\delta m} \big( x, m(s,\mu) \big)(y)  \cdot \nabla w(s,x) \bigg] \,  q(s) (dy)   \bigg| \,ds \nonumber \\  
& \leq  & C  \sum_{i=1}^d \int_0^t \| q(s) \|_{-(\na)}  \| \partial_{x_i} w(s, \cdot) \|_{\infty} \,ds    \nonumber \\
& \leq & C \bigg( \sup_{s \in [0,T]} \| w(s, \cdot ) \|_{n,\infty} \bigg) \int_0^t  \| q(s)\|_{-(n,\infty)}  \,ds     \nonumber \\
& \leq & C   \| \xi \|_{n,\infty} \int_0^t  \| q(s)\|_{-(n,\infty)}  \,ds.   \label{eq: bound q est 2}  
\end{eqnarray}
Finally, by \eqref{eq: bound w schauder}, 
\begin{equation}
  \bigg| \sigma \int_0^t \lev w(s,\cdot), r(s)   \rev_{n,\infty}  \, ds \bigg| \leq C \| \xi \|_{n,\infty} \sup_{u \in [0,T]} \| r(u) \|_{-(n,\infty)}. \label{eq: bound q est 3}
    \end{equation}
By \eqref{eq: xi q equality}, along with estimates \eqref{eq: bound q est 1}, \eqref{eq: bound q est 2} and \eqref{eq: bound q est 3}, we have
$$ \|q(t)\|_{-(\na)} \leq C \bigg( \|q_0\|_{-(\na)} + \sup_{u \in [0,T]} \|r(u)\|_{-(\na)} + \int_0^t \|q(s)\|_{-(\na)}  \,ds \bigg),$$ 
which concludes by Gronwall's inequality that
\begin{equation}  \sup_{t \in [0,T]}  \| q(t) \|_{-(n,\infty)} \leq C \bigg( \| q_0 \|_{-(n,\infty)}  + \sup_{t \in [0,T]} \| r (t) \|_{-(n,\infty)} \bigg) .\label{gron est schauder} \end{equation} 
Now we pick $t, t' \in [0,T]$. Then \eqref{eq: xi q equality} becomes \begin{eqnarray}
&&  \int_{\bT^d} \xi(y) \, (q(t) - q(t')) (dy)  \nonumber \\
& = &\sigma
 \int_{t'}^t \int_{\bT^d}  \int_{\bT^d} \bigg[  \frac{\delta b}{\delta m} \big( x, m(s,\mu) \big)(y) \cdot \nabla w(s,x) \bigg] \,  q(s) (dy)   \,  m(s, \mu)  (dx) \,ds \nonumber  \\
&& + \sigma \int_{t'}^t \lev w(s, \cdot), r(s)   \rev_{n,\infty}  \, ds. \label{eq: xi q equality holder } 
\end{eqnarray} 
By combining \eqref{gron est schauder} with the argument of \eqref{eq: bound q est 2},  
\begin{eqnarray} 
&& \bigg| \sigma \int_0^t \int_{\bT^d}  \int_{\bT^d} \bigg[  \frac{\delta b}{\delta m} \big( x, m(s,\mu) \big)(y) \cdot \nabla w(s,x) \bigg] \,  q(s) (dy)   \,  m(s, \mu)  (dx) \,ds \bigg| \nonumber \\
& \leq &  C   \| \xi \|_{n,\infty} \int_{t'}^t  \| q(s)\|_{-(n,\infty)}  \,ds \nonumber \\
& \leq & C  \| \xi \|_{n,\infty} |t-t'|  \bigg( \| q_0 \|_{-(n,\infty)}  + \sup_{s \in [0,T]} \| r (s) \|_{-(n,\infty)} \bigg). \label{eq: xi q equality holder 2} 
\end{eqnarray}
Similarly,
\begin{equation}
  \bigg| \sigma \int_{t'}^t \lev w(s, \cdot), r(s)   \rev_{n,\infty}  \, ds \bigg| \leq  C |t-t'| \| \xi \|_{n,\infty} \sup_{s \in [0,T]} \| r(s) \|_{-(n,\infty)}. \label{eq: bound q est 3 holder}
    \end{equation}
Therefore, by combining \eqref{eq: xi q equality holder }, \eqref{eq: xi q equality holder 2}  and \eqref{eq: bound q est 3 holder}, we have
$$ \sup_{t \neq t'} \frac{\|q(t) - q(t') \|_{-(n,\infty)} }{|t-t'|^{\beta}}  \leq  C|t-t'|^{1- \beta} \bigg( \| q_0 \|_{-(n,\infty)}  + \sup_{s \in [0,T]} \| r (s) \|_{-(n,\infty)} \bigg). $$ 
Combining with \eqref{gron est schauder} gives
$$ \| q \|_X \leq C  \bigg( \| q_0 \|_{-(n,\infty)}  + \sup_{t \in [0,T]} \| r (t) \|_{-(n,\infty)} \bigg). $$ 
Consequently, by the Leray-Schauder theorem, the map $T$ admits a fixed point. This shows the existence of solution to \eqref{eq q}. For uniqueness, one simply has to apply a Gronwall argument to \eqref{eq: xi q equality}. Finally, the estimate for the solution follows by repeating the proof up to  \eqref{gron est schauder}, but with $\sigma =1$. 
\end{proof}

\begin{lemma}
Assume \emph{\regb[n]}, where $n \geq 1$. Then the Cauchy problem $m^{(1)}$ defined in \eqref{PDE linearised}  has a unique solution in $L^{\infty} \big([0,T], ( W^{n,\infty}(\mathbb{T}^d) )' \big)$. \label{ m1 regularity} \end{lemma}
\begin{proof}
This is immediate from Theorem \ref{lions paper pde result}.
\end{proof}
For every $t \in [0,T]$, $\mu, \hmu \in \cP(\bT^d)$, let 
$$ \rho(t, \mu, \hat{\mu}) := {m}(t, \hat{\mu}) - m(t, {\mu}) - m^{(1)}(t, \mu,\hat{\mu}).$$ 
Let $\phi \in C^{\infty} ([0,T] \times \bT^d)$. By \eqref{eq: forward eqn } and \eqref{PDE linearised}, we have
\begin{eqnarray}
&&  \int_{\bT^d} \phi(t,y)  \, \rho(t, \mu,\hat{\mu}) (dy) - \int_{\bT^d} \phi(0,y)  \, \rho(0, \mu,\hat{\mu}) (dy)  \nonumber \\
& = & \int_0^t \int_{\bT^d} \partial_s \phi(s,y)  \, \rho(s, \mu, \hat{\mu}) (dy)  \,ds + \int_0^t \int_{\bT^d} \Delta \phi(s,y)  \, \rho(s, \mu, \hat{\mu}) (dy)  \,ds \nonumber \\
&& + \int_0^t \int_{\bT^d} \Big[ b \big( y, m(s, \mu) \big) \cdot \nabla \phi(s,y) \Big]  \,\rho(s, \mu, \hat{\mu}) (dy)  \,ds \nonumber \\
&& + \int_0^t \int_{\bT^d}  \Big[ \big( b \big( y, {m}(s, \hat{\mu}) \big) - b \big( y, {m}(s, \mu) \big) \big) \cdot \nabla \phi(s,y) \Big]  \,{m}(s, \hat{\mu}) (dy)  \,ds \nonumber \\
&& - \int_0^t \int_{\bT^d}  \int_{\bT^d} \bigg[ \frac{\delta b}{\delta m} \big( x, m(s,\mu) \big)(y) \cdot \nabla \phi(s,x) \bigg] \,  m^{(1)}(s, \mu, \hat{\mu}) (dy)   \,  m(s, \mu)  (dx) \,ds. \nonumber  \\
\end{eqnarray}
We rewrite the final two terms as
\begin{eqnarray}
&&  \int_0^t \int_{\bT^d}  \Big[ \big( b \big( y, {m}(s, \hat{\mu}) \big) - b \big( y, {m}(s, \mu) \big) \big) \cdot \nabla \phi(s,y) \Big]  \,{m}(s, \hat{\mu}) (dy)  \,ds \nonumber \\
&& - \int_0^t \int_{\bT^d}  \int_{\bT^d} \bigg[ \frac{\delta b}{\delta m} \big( x, m(s,\mu) \big)(y) \cdot \nabla \phi(s,x) \bigg] \,  m^{(1)}(s, \mu, \hat{\mu})  (dy)   \,  m(s, \mu)  (dx) \,ds \nonumber \\
&=&  \int_0^t \int_{\bT^d}  \Big[ \big( b \big( y, {m}(s, \hat{\mu})  \big) - b \big( y, {m}(s, \mu) \big) \big) \cdot \nabla \phi(s,y) \Big]  \,\big( {m}(s, \hat{\mu}) - {m}(s, \mu) \big) (dy) \,ds  \nonumber \\
&& - \int_0^t  \int_{\bT^d}  \int_{\bT^d} \bigg[ \bigg( \frac{\delta b}{\delta m}\big( x, m(s,\mu) \big)(y) - \Big( b \big( x, {m}(s, \hat{\mu}) \big) - b \big( x, {m}(s, \mu) \big) \Big)   \bigg)  \nonumber \\
&& \quad \quad \cdot \nabla \phi(s,x) \bigg] \,  m^{(1)}(s, \mu, \hat{\mu}) (dy)   \,  m(s, \mu)  (dx)  \,ds  \nonumber \\
& = &  \int_0^t \int_{\bT^d}  \Big[ \big( b \big( y, {m}(s, \hat{\mu})  \big) - b \big( y, {m}(s, \mu) \big) \big) \cdot \nabla \phi(s,y) \Big]  \,\big( {m}(s, \hat{\mu}) - {m}(s, \mu) \big) (dy)  \,ds \nonumber \\
&& - \int_0^t \int_{\bT^d}  \int_{\bT^d} \bigg[ \bigg( \frac{\delta b}{\delta m}\big( x, m(s,\mu) \big)(y) - \int_0^1 \int_{\bT^d} \frac{\delta b}{\delta m}\big( x, r{m}(s, \hat{\mu}) + (1-r) {m}(s,\mu) \big)(y') \nonumber \\
&& \big( {m}(s, \hat{\mu}) - {m}(s, \mu) \big)(dy') \, dr \bigg) \cdot \nabla \phi(s,x) \bigg] \,  \,  m^{(1)}(s, \mu, \hat{\mu}) (dy)   \,  m(s, \mu)  (dx) \,ds \nonumber \\
& = &  \int_0^t \int_{\bT^d}  \Big[ \big( b \big( y, {m}(s, \hat{\mu})  \big) - b \big( y, {m}(s, \mu) \big) \big) \cdot \nabla \phi(s,y) \Big]  \,\big( {m}(s, \hat{\mu}) - {m}(s, \mu) \big) (dy)  \,ds \nonumber \\
&& -\int_0^t \int_{\bT^d}  \int_{\bT^d} \int_0^1 \bigg[ \bigg( \frac{\delta b}{\delta m}\big( x, m(s,\mu) \big)(y) -   \frac{\delta b}{\delta m}\big( x, r{m}(s, \hat{\mu}) + (1-r) {m}(s,\mu) \big)(y) \bigg) \cdot \nabla \phi(s,x) \bigg] \nonumber \\
&& \quad \quad \, dr \, \big( {m}(s, \hat{\mu}) -  {m}(s, \mu) \big)(dy)    \,  m(s, \mu)  (dx)  \,ds \nonumber \\
&& + \int_0^t \int_{\bT^d}  \int_{\bT^d} \bigg[ \frac{\delta b}{\delta m}\big( x, m(s,\mu) \big)(y)  \cdot \nabla \phi(s,x) \bigg] \, \, \big( {m}(s, \hat{\mu}) -  {m}(s, \mu) -  m^{(1)}(s, \mu, \hat{\mu}) \big)(dy)  \,  m(s, \mu)  (dx) \,ds . \nonumber 
\end{eqnarray}
Therefore, we obtain that
\begin{eqnarray}
&&  \int_{\bT^d} \phi(t,y)  \, \rho(t, \mu,\hat{\mu}) (dy) - \int_{\bT^d} \phi(0,y)  \, \rho(0, \mu,\hat{\mu}) (dy)  \nonumber \\
& = & \int_0^t \int_{\bT^d} \partial_s \phi(s,y)  \, \rho(s, \mu, \hat{\mu}) (dy)  \,ds + \int_0^t \int_{\bT^d} \Delta \phi(s,y)  \, \rho(s, \mu, \hat{\mu}) (dy)  \,ds  \nonumber \\
&& + \int_0^t \int_{\bT^d} \Big[ b \big( y, m(s, \mu) \big) \cdot \nabla \phi(s,y) \Big]  \,\rho(s, \mu, \hat{\mu}) (dy)  \,ds \nonumber \\
&& +  \int_0^t \int_{\bT^d}   \bigg[ \frac{\delta b}{\delta m}\big( x, m(s,\mu) \big)(\rho(s, \mu,\hat{\mu}))  \cdot \nabla \phi(s,x) \bigg]  \,  m(s, \mu)  (dx) \,ds  \nonumber \\
&&  + \int_0^t \int_{\bT^d}  \Big[ \big( b \big( y, {m}(s, \hat{\mu})  \big) - b \big( y, {m}(s, \mu) \big) \big) \cdot \nabla \phi(s,y) \Big]  \,\big( {m}(s, \hat{\mu}) - {m}(s, \mu) \big) (dy)  \,ds \nonumber \\
&& -\int_0^t \int_{\bT^d}  \int_0^1 \bigg[ \bigg( \frac{\delta b}{\delta m}\big( x, m(s,\mu) \big) \big({m}(s, \hat{\mu}) -  {m}(s, \mu) \big) \nonumber \\
&& -   \frac{\delta b}{\delta m}\big( x, r{m}(s, \hat{\mu}) + (1-r) {m}(s,\mu) \big) \big({m}(s, \hat{\mu}) -  {m}(s, \mu) \big) \bigg) \cdot \nabla \phi(s,x) \bigg] \, dr \,   m(s, \mu)  (dx)  \,ds. \label{rho long form} 
\end{eqnarray}
In distributional sense, we write
\begin{equation}
\left\{
\begin{array}{rrl}
      {} & \partial_t \rho(t, \mu,\hat{\mu}) - \Delta \rho(t, \mu,\hat{\mu}) + \text{div} \big( b(\cdot,m(t,\mu)) \rho(t, \mu,\hat{\mu}) \big) \\
      &  + \text{div}  \big( m(t, \mu)\frac{\delta b}{\delta m}( x, m(t,\mu) )(\rho(t, \mu,\hat{\mu})\big)  -   c(t,\mu,\hat{\mu})  & =0, \\
      & &  \\
      {} &  \rho(0,\mu, \hat{\mu}) &= 0,\\
\end{array} 
\right.  \label{eq: difference fokker planck first order}  \end{equation} 
where
\begin{eqnarray} c(t,\mu,\hat{\mu}) & := & -\text{div} \bigg[ \big( m(t,\hat{\mu}) - m(t,{\mu}) \big) \Big( b \big( \cdot, {m}(t,\hat{\mu})\big)  - b \big( \cdot, {m} (t,{\mu})\big) \Big)   \nonumber \\
&& - m(t,{\mu})  \int_0^1 \bigg[  \frac{\delta b}{\delta m}\big( \cdot, m(t,\mu) \big) \big({m}(t, \hat{\mu}) -  {m}(t, \mu) \big) \nonumber \\
&&  -   \frac{\delta b}{\delta m}\big( \cdot, r{m}(t, \hat{\mu}) + (1-r) {m}(t,\mu) \big) \big({m}(t, \hat{\mu}) -  {m}(t, \mu) \big)  \bigg] \, dr \bigg].  \label{eq: signed measure}  
\end{eqnarray}

 We first establish the regularity of $c$.
\begin{lemma} \label{verify regularity of c} 
Assume \emph{\regb[n-1]} and \emph{\lipb[0]}, where $n \geq 2$. Then $c (\cdot, \mu, \hat{\mu}) \in L^{\infty} \big([0,T], (W^{n,\infty}(\mathbb{T}^d) )' \big)$. 
\end{lemma}
\begin{proof}
For any $\xi \in W^{n,\infty}(\bT^d),$
\begin{eqnarray}
&& \lev \xi, c(t, \mu, \hat{\mu}) \rev_{n,\infty} \nonumber \\
& = &  \int_{\bT^d}  \Big[ \big( b \big( y, {m}(t, \hat{\mu})  \big) - b \big( y, {m}(t, \mu) \big) \big) \cdot \nabla \xi(y) \Big]  \,\big( {m}(t, \hat{\mu}) - {m}(t, \mu) \big) (dy)  \nonumber \\
&& - \int_{\bT^d}  \int_{\bT^d} \int_0^1 \bigg[ \bigg( \frac{\delta b}{\delta m}\big( x, m(t,\mu) \big)(y) -   \frac{\delta b}{\delta m}\big( x, r{m}(t, \hat{\mu}) + (1-r) {m}(t,\mu) \big)(y) \bigg) \cdot \nabla \xi(x) \bigg] \nonumber \\
&& \quad \quad \, dr \, \big( {m}(t, \hat{\mu}) -  {m}(t, \mu) \big)(dy)    \,  m(t, \mu)  (dx).  \nonumber 
\end{eqnarray}
Next, we estimate each of the two terms. By \eqref{eq: bound buckdahn} and {\regb[n-1]}, since $n \geq 2$,
\begin{eqnarray}
&& \bigg| \int_{\bT^d}  \Big[ \big( b \big( y, {m}(t, \hat{\mu})  \big) - b \big( y, {m}(t, \mu) \big) \big) \cdot \nabla \xi(y) \Big]  \,\big( {m}(t, \hat{\mu}) - {m}(t, \mu) \big) (dy) \bigg| \nonumber \\
& \leq & \sum_{i=1}^d \big\| \big( b_i \big( \cdot, {m}(t, \hat{\mu})  \big) - b_i \big( \cdot, {m}(t, \mu) \big)  \partial_{x_i} \xi \|_{n-1,\infty} \| {m}(t, \hat{\mu}) -{m}(t, \mu)) \|_{-(n-1,\infty)}\nonumber \\
& \leq & C \sum_{i=1}^d \bigg\| \int_0^1 \intrd  \frac{ \partial b_i}{\partial m}(\cdot, r m(t, \hat{\mu}) + (1-r) m(t, \mu),y') \,  \big( m(t, \hat{\mu})- m(t, {\mu}) \big)(dy') \, dr \bigg\|_{n-1,\infty} \nonumber \\
&& \times  \| \partial_{x_i} \xi \|_{n-1,\infty} W_1( {m}(t, \hat{\mu}) ,{m}(t, \mu)) \nonumber \\
& \leq & C \| \xi \|_{n,\infty} W_1( {m}(t, \hat{\mu}) ,{m}(t, \mu))^2 \nonumber \\
& \leq & C W_1(  \hat{\mu} , \mu)^2  \| \xi \|_{n,\infty}  .\label{ct estimate first part}  \end{eqnarray}
Similarly, by \lipb[0],
\begin{eqnarray}
&& \bigg| \int_{\bT^d}  \int_{\bT^d} \int_0^1 \bigg[ \bigg( \frac{\delta b}{\delta m}\big( x, m(t,\mu) \big)(y) -   \frac{\delta b}{\delta m}\big( x, r{m}(t, \hat{\mu}) + (1-r) {m}(t,\mu) \big)(y) \bigg) \cdot \nabla \xi(x) \bigg] \nonumber \\
&& \quad \quad \, dr \, \big( {m}(t, \hat{\mu}) -  {m}(t, \mu) \big)(dy)    \,  m(t, \mu)  (dx)  \bigg| \nonumber \\
& \leq & \sum_{i=1}^d \int_0^1 \bigg\| \bigg[ \frac{\delta b_i}{\delta m}\big( \cdot, m(t,\mu) \big)({m}(t, \hat{\mu}) -  {m}(t, \mu)) \nonumber \\
&& -   \frac{\delta b_i}{\delta m}\big( \cdot, r{m}(t, \hat{\mu}) + (1-r) {m}(t,\mu) \big)({m}(t, \hat{\mu}) -  {m}(t, \mu)) \bigg] \partial_{x_i} \xi  \bigg\|_{\infty} \,dr \nonumber \\
& \leq &  \sum_{i=1}^d \bigg[ \text{Lip}_{0}\bigg(\frac{\delta b_i}{\delta m} \bigg) W_1(m(t, \mu), m(t, \hat{\mu}))^2 \| \partial_{x_i} \xi \|_{\infty}  \bigg] \nonumber \\
& \leq & W_1(  \hat{\mu} , \mu)^2 \bigg(  \| \xi \|_{n,\infty} \sum_{i=1}^d \text{Lip}_{0}\bigg(\frac{\delta b_i}{\delta m} \bigg) \bigg) . \label{ct estimate second part} 
\end{eqnarray}
Combining \eqref{ct estimate first part} and \eqref{ct estimate second part}, we have
\begin{eqnarray}
&& \lev \xi, c(t, \mu, \hat{\mu}) \rev_{n,\infty} \nonumber \\
& \leq & C  W_1(  \hat{\mu} , \mu)^2  \| \xi \|_{n,\infty} , \nonumber 
\end{eqnarray}
which implies that $c(t, \mu, \hat{\mu})$ is a bounded operator with its operator norm given by
\begin{eqnarray}
&& \sup_{t \in [0,T]} \|c(t, \mu, \hat{\mu})\|_{-(n,\infty)} \nonumber \\
& \leq & C W_1(  \hat{\mu} , \mu)^2  . \label{ct main estimate}
\end{eqnarray}
\end{proof}

The following theorem is a straightforward consequence of the above results.
\begin{theorem} \label{thm m m-one} 
Assume \emph{\regb[n]}, \emph{\lipb[0]}, \emph{\hlipphi[1]} and \emph{\hregphi{n}{1}}, where $n \geq 2$.  Then the following statements hold.
\begin{enumerate}[(i)]
        \item There exists some constant $C>0$ such that
    $$ \sup_{t \in [0,T]} \| m(t, \hat{\mu}) - m(t, \mu) - m^{(1)}(t, \mu, \hat{\mu}) \|_{-(n,\infty)} \leq C W_1( \mu, \hat{\mu})^2.$$ 
       \item For $\cU$ defined by \eqref{eq: defofflow}, 
   \begin{equation}  \sup_{t \in [0,T]} \bigg| \cU(t, \hat{\mu}) - \cU( t,{\mu}) - \int_{\bT^d} \frac{\delta \Phi}{\delta m} (m(t,\mu))(x) \, m^{(1)}(t, \mu, \hat{\mu})(dx)  \bigg| \leq   C W_1({\mu} ,\hat{\mu})^2 .
   \label{ U phi W1 W2 bound} \end{equation}
          \item \begin{equation} \frac{d}{d \eps} \bigg|_{\eps=0^{+}} \Phi \big( m(t, (1- \eps) \mu + \eps \hat{\mu}) \big) = \int_{\bT^d} \frac{\delta \Phi}{\delta m}(m(t, \mu))(y) \, m^{(1)}(t,\mu,\hat{\mu}) (dy).
        \label{heuristics verified}  \end{equation}
            \end{enumerate}
\end{theorem} ${}$ \\
\begin{proof} ${}$ \\
\begin{enumerate}[(i)] 
\item This follows from \eqref{eq: difference fokker planck first order}, estimate \eqref{ct main estimate} and Theorem \ref{lions paper pde result}.
\item Let $\pi$ be the optimal transport plan from $m(t, \mu)$ to $m(t, \hmu)$. The computation from the proof of Proposition 5.44 from \cite{carmona2017probabilistic} shows that 
\begin{eqnarray}
&& \Phi \big( m(t, \hat{\mu}) \big) - \Phi \big( m(t, {\mu}) \big) - \int_{\bT^d} \frac{\delta \Phi}{\delta m} \big( m(t, {\mu}) \big) (y) \, \big( m(t, \hat{\mu}) - m(t, {\mu})  \big)(dy) \nonumber \\
& = & \int_0^1 \int_0^1 \int_{\bT^d \times \bT^d} \bigg[ \bigg( \partial_x \ld[\Phi] \Big( t m(t, \hat{\mu}) + (1-t) m(t, {\mu}) \Big) \big(\lambda y + (1- \lambda )x \big) \nonumber \\
&& - \partial_x \ld[\Phi] (m(t, {\mu})) \big( \lambda y +  (1-\lambda)x \big) \bigg) \cdot (y-x) \bigg] \, \pi(dx, dy) \, d \lambda \, dt. \label{eq: optimal transport delarue} 
\end{eqnarray}
By \hlipphi[1], \eqref{eq: bound buckdahn} and the fact that
\begin{equation}
                 W_1( (1- \eps) \mu + \eps \hat{\mu}, \mu)  \leq \eps W_1(\mu, \hat{\mu}),
      \label{W1 W2 estimates linear interpolation}  \end{equation}
there exists some constant $C>0$ such that
\begin{eqnarray} 
&& \bigg| \Phi \big( m(t, \hat{\mu}) \big) - \Phi \big( m(t, {\mu}) \big) - \int_{\bT^d} \frac{\delta \Phi}{\delta m} \big( m(t, {\mu}) \big) (y) \, \big( m(t, \hat{\mu}) - m(t, {\mu})  \big)(dy) \bigg|  \nonumber \\
& \leq &  C W_1(m(t, {\mu}) , m(t, \hat{\mu}) )^2 \leq  C W_1({\mu} ,\hat{\mu})^2. \nonumber
\end{eqnarray}
By assumption {\hregphi{n}{1}} and part (i), there exists some constant $C'>0$ such that
\begin{eqnarray}
&& \bigg| \cU(t,\hat{\mu})  - \cU(t,{\mu}) - \int_{\bT^d} \frac{\delta \Phi}{\delta m} \big( m(t, {\mu}) \big) (y) \, m^{(1)}(t, \mu, \hat{\mu})  (dy) \bigg| \nonumber \\
& \leq & \bigg| \int_{\bT^d} \frac{\delta \Phi}{\delta m} \big( m(t, {\mu}) \big) (y) \, \big(  m(t, \hat{\mu}) - m(t, {\mu})  - m^{(1)}(t,  \mu, \hat{\mu}) \big)(dy)\bigg| + C W_1({\mu} ,\hat{\mu})^2  \nonumber \\
& \leq & C'W_1(\mu ,  \hat{\mu} )^2, \nonumber
\end{eqnarray}
which completes the proof.
\item \begin{eqnarray}
&& \frac{d}{d \eps} \bigg|_{\eps=0^{+}} \Phi \big( m(t, (1- \eps) \mu + \eps \hat{\mu}) \big) \nonumber \\
& = & \lim_{\eps \to 0^{+}} \frac{1}{\eps} \bigg[  \Phi \big( m(t, (1- \eps) \mu + \eps \hat{\mu}) \big) - \Phi \big( m(t,  \mu ) \big) \bigg] \nonumber \\
& = & \lim_{\eps \to 0^{+}} \frac{1}{\eps} \bigg[ \int_{\bT^d} \frac{ \delta \Phi}{\delta m} \big( m(t, \mu) \big)(x) \, m^{(1)}(t, \mu ,  (1- \eps) \mu + \eps \hat{\mu})  (dx) + O \Big(W_1( \mu, (1- \eps) \mu + \eps \hat{\mu} )^2 \Big) \bigg] \nonumber \\
& = & \int_{\bT^d} \frac{ \delta \Phi}{\delta m} \big( m(t, \mu) \big)(x) \, m^{(1)}(t, \mu ,  \hat{\mu}) (dx), \nonumber 
\end{eqnarray}
by \eqref{W1 W2 estimates linear interpolation} and the fact that $m^{(1)}(t, \mu ,  (1- \eps) \mu + \eps \hat{\mu}) = \eps m^{(1)}(t, \mu , \hat{\mu}) $. 
\end{enumerate}
\end{proof}
\subsection{Analysis of the backward Kolmogorov equation} \label{ backward first order } 
We observe that, in \eqref{ U phi W1 W2 bound}, the integral is with respect to the signed measure $m^{(1)}(t, \mu, \hat{\mu} )$. To show that $\cU$ indeed has a linear functional derivative, we need to express the integral in terms of the signed measure $\hat{\mu}- \mu$. To this end, we fix $t \in [0,T]$ and $x \in \bR^d$ and introduce the decoupled process $\{ X^{0,x,\mu}_u \}_{u \in [0,t]}$ by
\begin{equation} 
      X^{0,x,\mu}_s  = x + \int_0^s b( X^{0,x,\mu}_r, m(r,\mu)) \, dr + \sqrt{2} W_s, \quad \quad 0  \leq s \leq t. \label{eq:decoupled} \end{equation}
For every $\xi: \bT^d \to \bR$ and $t \in [0,T]$, we define a function $v(\cdot, \cdot, \cdot; \xi ,t): [0,t] \times \bT^d \times \cP(\bT^d) \to \bR$ such that
$$ v(s,x, \mu; \xi ,t):= \bE \big[ \xi(X^{0,x, \mu}_t) \big| X^{0,x, \mu}_s = x \big].$$ 
Note that
$$ v(0,x,\mu; \xi ,t) = \bE \big[ \xi(X^{0,x, \mu}_t) \big]. $$ 
It is well-known that (see, for example, equation (3.4) in \cite{buckdahn2017mean})
$$ \int_{\bT^d} \xi(x) \, m(t,\mu)(dx) = \bE[ \xi(X^{0, \mu}_t)] =  \int_{\bT^d} v(0,x, \mu; \xi ,t) \, \mu(dx) .$$
Therefore,
\begin{equation}  \int_{\bT^d} \xi(x) \, m(t,(1- \eps) \mu + \eps \hat{\mu} )(dx) = \int_{\bT^d} v(0,x, (1- \eps) \mu + \eps \hat{\mu} ; \xi ,t) \, \big( (1- \eps) \mu + \eps \hat{\mu} \big)(dx) , \label{v phi connection 1} \end{equation}
for any $\mu, \hmu \in \cP(\bT^d)$. If $\frac{\delta v}{\delta m}$ exists, taking derivative w.r.t. $\eps$ at $0$ gives
\begin{eqnarray}
\int_{\bT^d} \xi(x) \, m^{(1)}(t, \mu, \hat{\mu})(dx)  & = & \int_{\bT^d} v(0,x, \mu; \xi ,t) \, (\hat{\mu}- \mu)(dx) \nonumber \\
&& + \int_{\bT^d} \int_{\bT^d}  \frac{\delta v}{\delta m}(0,x, \mu,z; \xi ,t) (\hat{\mu}- \mu)(dz) \, \mu(dx) \nonumber \\
& = &  \int_{\bT^d}  \bigg[ v(0,x, \mu; \xi ,t) + \int_{\bT^d}  \frac{\delta v}{\delta m}(0,z, \mu,x; \xi ,t)   \, \mu(dz) \bigg] (\hat{\mu}- \mu)(dx). \nonumber \\
&& \label{v phi connection 2}
\end{eqnarray}
Hence, it suffices to study the regularity of $v$. In most of the analysis for $v$, we suppress the parameters $\xi$ and $t$, for simplicity of notations. By the standard Feynman-Kac equation (Kolmogorov backward equation), $v$ satisfies the PDE
\begin{equation}
\begin{cases}
     \partial_s v(s,x, \mu) + b(x, m(s, \mu)) \cdot \nabla v(s,x,\mu) + \Delta v(s,x,\mu) =0, \\
             v(t,x, \mu)= \xi(x).
    \end{cases} \label{eq: backward v}
    \end{equation}
    \begin{lemma} \label{reg v lemma} 
        Assume \emph{\regb[n]}, where $n \geq 2$. Suppose that $\xi \in W^{n+ 1,\infty}$. Then the  Cauchy problem $v$ defined in \eqref{eq: backward v}  has a unique solution in $L^{\infty} \big([0,t], W^{n+1,\infty}(\mathbb{T}^d)  \big)$. Moreover, there exists a constant $C>0$ (depending on $\xi$) such that for any $\mu,\hmu \in \cP(\bT^d)$,
        $$ \sup_{s \in [0,t]} \| v(s, \cdot, \mu) - v(s, \cdot, \hmu) \|_{n+1,\infty} \leq C W_1( \mu, \hmu).$$ 
    \end{lemma}
    \begin{proof}
    The fact that $v \in L^{\infty} \big([0,t], W^{n+1,\infty}(\mathbb{T}^d)  \big)$ follows from Proposition \ref{thm schauder}. For the second part, take any $\mu,\hmu \in \cP(\bT^d)$. Let
$$ z(s,x):= v(s,x ,\mu)- v(s,x, \hmu).$$ 
Then $z$ satisfies the Cauchy problem
\begin{equation}
\begin{cases}
     \partial_s z(s,x) + \Delta z(s,x) + b(x, m(s, \hmu)) \cdot \nabla z(s,x)  =  \big( b(x, m(s, \hmu)) - b(x, m(s, \mu)) \big) \cdot \nabla v(s,x, \mu), \\
             z(t,x)= 0.
    \end{cases} \label{eq: backward r}
    \end{equation}
  Using the same argument as \eqref{ct estimate first part}, by \eqref{eq: bound buckdahn}, {\regb[n]} and Proposition \ref{thm schauder}, there exists a constant $C>0$ such that
  \begin{eqnarray} \sup_{s \in [0,t]} \| z(s, \cdot)\|_{n+1,\infty} & \leq &  C \sup_{s \in [0,t]} \Big\| \big( b(\cdot, m(s, \hmu)) - b(\cdot, m(s, \mu)) \big) \cdot \nabla v(s,\cdot, \mu) \Big\|_{n,\infty} \nonumber \\
  & \leq & C \sum_{i=1}^d \sup_{s \in [0,t]} \bigg[ \Big\|  b_i(\cdot, m(s, \hmu)) - b_i(\cdot, m(s, \mu)) \Big\|_{n,\infty} \| v(s, \cdot, \mu) \|_{n+1,\infty} \bigg] \nonumber \\
  & \leq & C W_1( \mu, \hmu). \nonumber 
  \end{eqnarray}
    \end{proof}
    The core analysis of backward Kolmogorov equations depends on the following fact. 
\begin{theorem}[Bound for backward Kolmogorov equations] \label{lions paper pde result second}
Assume \emph{\regb[n]}, where $n \geq 2$. Suppose that $\xi \in W^{n+1,\infty}$. Let $q \in L^{\infty} \big([0,t], ( W^{n,\infty}(\mathbb{T}^d) )' \big)$ and  $\gamma \in L^{\infty} \big([0,t],  W^{n,\infty}(\mathbb{T}^d) \big)$.  Then the Cauchy problem $h$   
 \begin{equation}
\left\{
\begin{array}{rrl}
      {} & \partial_s h (s,x) +\Delta h (s,x) +  b(x, m(s, \mu)) \cdot \nabla h(s,x)  + \ld[b](x, m(s,\mu))(q(s)) \cdot \nabla v(s,x,\mu) + \gamma(s,x) & =0, \\
      & &  \\
      {} &  h(t,x) &= 0, \\
\end{array} 
\right. \label{general backward} \end{equation} 
has a unique solution in $L^{\infty} \big([0,t],  W^{n+1,\infty}(\mathbb{T}^d) \big)$ such that
$$  \sup_{s \in [0,t]} \| h (s, \cdot) \|_{n+1,\infty} \leq C \bigg( \sup_{s \in [0,t]} \| q(s)  \|_{-(n,\infty)} + \sup_{s \in [0,t]} \| \gamma(s, \cdot)  \|_{n,\infty} \bigg) , $$ 
for some  constant $C>0$ depending on $\xi$. \end{theorem}
    \begin{proof}
By \regb[n] and Proposition \ref{thm schauder},
\begin{eqnarray}
\sup_{s \in [0,t]} \| h (s, \cdot) \|_{n+1,\infty} & \leq & C \bigg( \sup_{s \in [0,t]} \bigg\| \ld[b](\cdot, m(s,\mu))(q(s)) \cdot \nabla v(s,\cdot,\mu) \bigg\|_{\na}  + \sup_{s \in [0,t]}\| \gamma(s,\cdot) \|_{\na} \bigg)  \nonumber \\
& \leq & C \sum_{i=1}^d \bigg( \sup_{s \in [0,t]} \bigg[ \bigg\| \ld[b_i](\cdot, m(s,\mu))(q(s)) \bigg\|_{\na}  \|  v(s,\cdot,\mu) \|_{n+ 1 ,\infty}  \bigg] + \sup_{s \in [0,t]}\| \gamma(s,\cdot) \|_{\na} \bigg) \nonumber \\
& \leq & C \bigg( \sup_{s \in [0,t]} \| q(s)  \|_{-(n,\infty)} + \sup_{s \in [0,t]} \| \gamma(s, \cdot)  \|_{n,\infty} \bigg). \nonumber 
\end{eqnarray}
\end{proof}
  Formal differentiation of \eqref{eq: backward v} w.r.t. the measure component gives
 
 \begin{equation}
\left\{
\begin{array}{rrl}
      {} & \partial_s \vone(s,x, \mu, \hat{\mu}) +\Delta \vone(s,x, \mu, \hat{\mu}) +  b(x, m(s, \mu)) \cdot \nabla \vone(s,x, \mu, \hat{\mu})   &
     \\
      {} & + \ld[b](x, m(s,\mu))(m^{(1)}(s,\mu, \hmu)) \cdot \nabla v(s,x, \mu) & =0, \\
      & &  \\
      {} &  \vone(t,x, \mu, \hat{\mu}) &= 0. \\
\end{array} 
\right. \label{vone linearised} \end{equation} 
We now study the regularity of $\vone$. 
\begin{lemma} \label{representation of v one} 
    Assume \emph{\regb[n]}, where $n \geq 2$. Suppose that $\xi \in W^{n+1,\infty}$.  Then the Cauchy problem $v^{(1)}$ defined in \eqref{vone linearised}  has a unique solution in $L^{\infty} \big([0,t], W^{n+1,\infty}(\mathbb{T}^d)  \big)$. Moreover, $\vone$ satisfies the relation
    $$ \vone(0,x, \mu, \hmu) = \intrd \vone(0,x, \mu, \delta_{z}) \, ( \hmu- \mu)(dz).$$ 
\end{lemma}
\begin{proof}
{The first part of the lemma follows directly from Theorem \ref{lions paper pde result second}.} For the second part, we note that $\vone(0,x, \mu, \delta_{z})$ satisfies 
 \begin{equation*}
\begin{array}{rrl}
      {} & \partial_s \vone(0,x, \mu, \delta_{z}) +\Delta \vone(0,x, \mu, \delta_{z}) +  b(x, m(0, \mu)) \cdot \nabla \vone(0,x, \mu, \delta_{z})   &
     \\
      {} & + \ld[b](x, m(0,\mu))(z) \cdot \nabla v(0,x, \mu) & =0, \\
     \end{array} 
\end{equation*} 
where the final term uses the normalisation condition of $\ld[b]$. Integrating both sides w.r.t. $z$ with measure $\hmu- \mu$, we have
\begin{equation}
\begin{array}{rl}
       \partial_s \Big[ \intrd \vone(0,x, \mu, \delta_{z}) \, (\hmu-\mu)(dz) \Big]  +\Delta_x \Big[ \intrd \vone(0,x, \mu, \delta_{z}) \, (\hmu-\mu)(dz) \Big] & \\
       +  b(x, m(0, \mu)) \cdot \nabla_x \Big[ \intrd \vone(0,x, \mu, \delta_{z}) \, (\hmu-\mu)(dz) \Big] + \ld[b](x, m(0,\mu))(m^{(1)}(0,\mu, \hmu)) \cdot \nabla v(0,x, \mu)  & =0. \\
     \end{array} \label{eq: v1 integral system}
\end{equation} 
By comparing \eqref{vone linearised} and \eqref{eq: v1 integral system}, the result follows by an argument of stability similar to Corollary 3.4.2 of \cite{cardaliaguet2015master}.
\end{proof}
As before, we consider the difference
$$ \Gamma (s,x, \mu, \hmu):= v(s,x, \hmu) - v(s,x, \mu) - \vone(s,x, \mu, \hmu). $$ 
Then $\Gamma$ satisfies the Cauchy problem 
 \begin{equation}
\left\{
\begin{array}{rrl}
      {} & \partial_s \Gamma(s,x, \mu, \hat{\mu}) +\Delta \Gamma(s,x, \mu, \hat{\mu}) +  b(x, m(s, \mu)) \cdot \nabla \Gamma(s,x, \mu, \hat{\mu})  \\
      & +\ld[b](x, m(s,\mu))(\rho(s,\mu, \hmu)) \cdot \nabla v(s,x, \mu) + F(s,x, \mu, \hmu) & =0, \\
      & &  \\
      {} &  \Gamma(t,x, \mu, \hat{\mu}) &= 0, \\
\end{array} 
\right. \label{gamma linearised} \end{equation} 
where
\begin{eqnarray}
F(s,x, \mu, \hmu) & := &  \Big( b(x, m(s, \hmu)) - b(x, m(s, \mu)) \Big) \cdot \nabla v(s,x, \hat{\mu})  \nonumber \\
&& -\bigg[\ld[b](x, m(s,\mu))(\rho(s,\mu, \hmu)) +\ld[b](x, m(s,\mu))(m^{(1)}(s,\mu, \hmu)) \bigg]  \cdot \nabla v(s,x, \mu) \nonumber \\
& = & \Big( b(x, m(s, \hmu)) - b(x, m(s, \mu)) \Big) \cdot \Big( \nabla v(s,x, \hat{\mu})  - \nabla v(s,x, {\mu}) \Big) \nonumber \\
&& + \bigg[ \Big( b(x, m(s, \hmu)) - b(x, m(s, \mu)) \Big) - \ld[b](x, m(s,\mu)) \big( m(s, \hmu) - m(s, \mu) \big)  \bigg] \cdot \nabla v(s,x, \mu) \nonumber \\
& = & \Big( b(x, m(s, \hmu)) - b(x, m(s, \mu)) \Big) \cdot \Big( \nabla v(s,x, \hat{\mu})  - \nabla v(s,x, {\mu}) \Big) \nonumber \\
&& + \int_0^1 \intrd \bigg[ \ld[b] \big( x, u m(s, \hmu) + (1-u) m(s,\mu) \big)(y) \nonumber \\
&& \quad \quad  \quad \quad \quad  - \ld[b] (x,m(s, \mu))(y) \bigg] \, \big( m(s, \hmu)- m(s, \mu) \big)(dy) \, du. \label{eq: F explicit form} 
\end{eqnarray}
The following result is immediate.
\begin{theorem} \label{thm derivative of v explicit form} 
 Assume \emph{\regb[n]}  and \emph{\lipb[0]}, where $n \geq 2$. Suppose that $\xi \in W^{n+1,\infty}$. Then $ \ld[v](0,x, \mu,y)$ exists and is given by
 $$\ld[v](0,x, \mu,y) = \vone (0,x , \mu, \delta_{y}). $$ 
\end{theorem}
\begin{proof}
We proceed in the same way as in the proof of Lemma \ref{verify regularity of c}. By {\regb[n]}, \lipb[0], \eqref{eq: bound buckdahn}, \eqref{W1 W2 estimates linear interpolation} and  Lemma \ref{reg v lemma}, we deduce from \eqref{eq: F explicit form}  that
$$ \sup_{s \in [0,t]} \| F(s, \cdot, \mu, \hat{\mu})  \|_{n,\infty} \leq C W_1(\mu, \hmu)^2, $$
for some constant $C>0$ depending on $\xi$. Therefore, by Theorem \ref{lions paper pde result second}, 
$$ \sup_{s \in [0,t]} \| \Gamma(s, \cdot, \mu, \hat{\mu})  \|_{n+1,\infty} \leq C W_1(\mu, \hmu)^2. $$
Therefore, by Lemma \ref{representation of v one},
$$ \bigg\| v(0,\cdot, \hmu) - v(0,\cdot,\mu) - \intrd v^{1}(0,\cdot, \mu,  \delta_z) \, (\hmu-\mu)(dz) \bigg\|_{n+1,\infty} \leq C W_1(\mu, \hmu)^2.$$ 
We conclude the result by the characterisation of linear functional derivatives in Remark 5.47 of \cite{carmona2017probabilistic}.
\end{proof}
\begin{corollary}[Existence of the first order linear derivative] \label{Corr existence first order} 
Assume \emph{\regb[n]}, \emph{\lipb[0]}, \emph{\hlipphi[1]} and \emph{\hregphi{n+1}{1}}, where $n \geq 2$. Then $\ld[\cU]$ exists and is given by
$$ \ld[\cU](t,\mu)(x) = v \Big(0,x, \mu; \ld[\Phi] (m(t, \mu))(\cdot) ,t \Big) + \int_{\bT^d}  \frac{\delta v}{\delta m} \Big(0,z, \mu,x; \ld[\Phi] (m(t, \mu))(\cdot) ,t \Big)   \, \mu(dz), $$ 
for every $\mu \in \cP(\bT^d)$. 
\end{corollary}
\begin{proof}
Fix $\mu_0 \in \cP(\bT^d)$. Firstly, we recall from \eqref{v phi connection 1} that
\begin{equation}  \int_{\bT^d} \ld[\Phi] (m(t, \mu_0))(x) \, m(t,(1- \eps) \mu + \eps \hat{\mu} )(dx) = \int_{\bT^d} v \Big(0,x, (1- \eps) \mu + \eps \hat{\mu} ; \ld[\Phi] (m(t, \mu_0))(\cdot) ,t \Big) \, \big( (1- \eps) \mu + \eps \hat{\mu} \big)(dx) . \label{ Phi v corollary relation} 
\end{equation}
Since $\Phi$ satisfies {\hlipphi[1]} and {\hregphi{n+1}{1}}, the function
$$ \tilde{\Phi}( {\mu}) := \intrd \ld[\Phi] (m(t, \mu_0))(x) \, {\mu}(dx)$$ 
satisfies {\hliptphi} and {\regtphi[n+1]}. Moreover, 
$$ \xi(x):= \ld[\Phi] (m(t, \mu_0))(x)$$ 
lies in $W^{n+1,\infty}$. Therefore, by part (iii) of Theorem \ref{thm m m-one}  and Theorem \ref{thm derivative of v explicit form}, we differentiate \eqref{ Phi v corollary relation} w.r.t. $\eps$ at $0$, which gives (by \eqref{v phi connection 2})
\begin{eqnarray}
&& \int_{\bT^d} \ld[\Phi] (m(t, \mu_0))(x) \, m^{(1)}(t, \mu, \hat{\mu})(dx) \nonumber \\
& = &  \int_{\bT^d}  \bigg[ v \Big(0,x, \mu; \ld[\Phi] (m(t, \mu_0))(\cdot) ,t \Big) + \int_{\bT^d}  \frac{\delta v}{\delta m} \Big(0,z, \mu,x; \ld[\Phi] (m(t, \mu_0))(\cdot) ,t \Big)   \, \mu(dz) \bigg] (\hat{\mu}- \mu)(dx). \nonumber 
\end{eqnarray}
Putting $\mu_0= \mu$, we have
\begin{eqnarray}
&& \int_{\bT^d} \ld[\Phi] (m(t, \mu))(x) \, m^{(1)}(t, \mu, \hat{\mu})(dx) \nonumber \\
& = &  \int_{\bT^d}  \bigg[ v \Big(0,x, \mu; \ld[\Phi] (m(t, \mu))(\cdot) ,t \Big) + \int_{\bT^d}  \frac{\delta v}{\delta m} \Big(0,z, \mu,x; \ld[\Phi] (m(t, \mu))(\cdot) ,t \Big)   \, \mu(dz) \bigg] (\hat{\mu}- \mu)(dx). \nonumber 
\end{eqnarray}
Finally, by part (iii) of Theorem \ref{thm m m-one}, we conclude that $\ld[\cU]$ exists and is given by
$$ \ld[\cU](t,\mu)(x) = v \Big(0,x, \mu; \ld[\Phi] (m(t, \mu))(\cdot) ,t \Big) + \int_{\bT^d}  \frac{\delta v}{\delta m} \Big(0,z, \mu,x; \ld[\Phi] (m(t, \mu))(\cdot) ,t \Big)   \, \mu(dz). $$ 
\end{proof}
\section{Higher order forward and backward Kolmogorov equations}
In this section, we repeat the same procedure in the previous section to establish regularity of higher order Kolmogorov equations. In order to proceed with an iteration argument, we first introduce the following class of multi-indices in the class $\tau_k$.
\subsection{Definitions and notations for iteration in multi-indices in the class \texorpdfstring{$\tau_k$}{tau-k}}
\begin{definition}[Class \texorpdfstring{$\tau_k$}{tau-k} of multi-indices]
For any $k \in \bN$, the class $\tau_k$ contains all multi-indices of the form
\begin{equation} {\lambda}:=  \bigg({\hat{n}}, ( \beta_j )_{j=1}^{{\hat{n}}}, ( {\alpha_{i,j}} )_{\substack{1 \leq i \leq {\hat{n}} \\ 1 \leq j \leq \beta_{i}}}, \hat{\beta}, ( \hat{\alpha}_{\ell} )_{1 \leq \ell \leq \hat{\beta}} \bigg), \label{lambda def} \end{equation}
where ${\hat{n}}$, $\beta_j$ and $\hat{\beta}$ are non-negative integers and $ \alpha_{i,j}$, $ \hat{\alpha}_{\ell} $,  $1 \leq i \leq {\hat{n}}$,  $ 1 \leq j \leq \beta_{i}$, $1 \leq \ell \leq \hat{\beta} $, are positive integers satisfying
\begin{enumerate}[(i)]
\item $ {\hat{n}} \leq k, \quad \quad 1 \leq \alpha_{i,1}< \ldots<  \alpha_{i,\beta_{i}} \leq k, \quad \quad 1 \leq \hat{\alpha}_1< \ldots<  \hat{\alpha}_{\hat{\beta}} \leq k, $
\item $ \beta_1, \ldots, \beta_{\hat{n}}, \hat{\beta} <k,$
\item exactly one of $\alpha_{i,j}$ and $\hat{\alpha}_{\ell}$ is equal to $k$,
\item \begin{equation} \sum_{i=1}^{\hat{n}} \beta_i + \hat{\beta} = k, \label{eq: sum condition}  \end{equation}
\item for any $i,i' \in \{1, \ldots, \hat{n} \},$ \begin{equation} \Big\{ \alpha_{i,1}, \ldots, \alpha_{i, \beta_i} \Big\} \cap  \Big\{ \alpha_{i',1}, \ldots, \alpha_{i', \beta_{i'}} \Big\}  = \emptyset,  \quad \quad   \Big\{ \alpha_{i,1}, \ldots, \alpha_{i, \beta_i} \Big\} \cap \Big\{ \hat{\alpha}_1, \ldots, \hat{\alpha}_{\hat{\beta}}\Big\} = \emptyset. \label{eq: emptyset condition}  \end{equation} 
\end{enumerate}
In particular, $o(\lambda)$ is called the \emph{order} of $\lambda$ defined by
$$ o(\lambda):={\hat{n}}.$$ 
Moreover, for any $(\lambda^{(1)}, \ldots, \lambda^{(q)}) \in (\tau_k)^q$, we define the \emph{magnitude} of $(\lambda^{(1)}, \ldots, \lambda^{(q)})$ by
$$ m\big( (\lambda^{(1)}, \ldots, \lambda^{(q)}) \big) := q.$$ If $\lambda= \lambda^{(i)}$, for some $i \in \{1, \ldots, q \}$, we write
$$ \lambda \in e \big( (\lambda^{(1)}, \ldots, \lambda^{(q)}) \big):= \{ \lambda^{(1)}, \ldots, \lambda^{(q)} \}.$$ 
\end{definition}
\begin{remark}
This definition is modified accordingly when one of ${\hat{n}}$, $\beta_j$ and $\hat{\beta}$ is zero. When ${\hat{n}}=0$, we set $\lambda:= \big(0, \hat{\beta},  (\hat{\alpha}_{\ell} )_{1 \leq \ell \leq \hat{\beta}} \big) $. On the other hand, when $\hat{\beta}=0$, we set $\lambda:= \Big({\hat{n}}, ( \beta_j )_{j=1}^{{\hat{n}}}, ( {\alpha_{i,j}} )_{\substack{1 \leq i \leq {\hat{n}} \\ 1 \leq j \leq \beta_{i}}}, 0 \Big). $ Finally, when $\beta_{j_0}=0$, for some $j_0 \in \{1, \ldots, {\hat{n}}\}$, the column entry of $j_0$ disappears in the array $( {\alpha_{i,j}} )$.
\end{remark}
Next, we introduce the recurrence map $T_k$ for multi-indices, followed by the sequence of multi-dimensional vectors $\lambda_k $ of elements in $\tau_k$.
\begin{definition}[Recurrence map $T_k$] \label{ def Tk} 
Let $\lambda \in \tau_k$ be given by the form \eqref{lambda def}. We define a recurrence map $T_k$ by 
\begin{eqnarray}
(\tau_{k+1})^{o(\lambda)+2} \ni T_k(\lambda) &:=& \bigg( \Big( {\hat{n}}+1, ( \beta_1, \ldots, \beta_{\hat{n}}, 1 ), ({\alpha}_{1,1}, \ldots, {\alpha}_{{\hat{n}}, \beta_{\hat{n}}},k+1), \hat{\beta}, (\hat{\alpha}_{\ell})_{1 \leq \ell \leq \hat{\beta}} \Big), \nonumber \\
&& \Big( {\hat{n}}, ( \beta_1, \ldots, \beta_{p-1}, \beta_p +1, \beta_{p+1}, \ldots, \beta_{\hat{n}}), \nonumber \\
&& ( \alpha_{1,1}, \ldots, \alpha_{p-1, \beta_{p-1}}, \alpha_{p,1}, \ldots, \alpha_{p, \beta_p}, k+1, \alpha_{p+1,1}, \ldots, \alpha_{{\hat{n}}, \beta_{\hat{n}}} ), \hat{\beta}, (\hat{\alpha}_{\ell})_{1 \leq \ell \leq \hat{\beta}} \Big)_{1 \leq p \leq {\hat{n}}}, \nonumber \\
&& \Big({\hat{n}}, ( \beta_j )_{j=1}^{{\hat{n}}}, ( {\alpha_{i,j}} )_{\substack{1 \leq i \leq {\hat{n}} \\ 1 \leq j \leq \beta_{i}}}, \hat{\beta}+1, (\hat{\alpha}_1, \ldots,\hat{\alpha}_{\hat{\beta}}, k+1) \Big) \bigg). \nonumber 
\end{eqnarray}
\end{definition}
\begin{definition}[Multi-dimensional vectors $\lambda_k $ of elements in $\tau_k$] \label{ def  lambda k} 
We first define
\begin{eqnarray}
\lambda_{2} & : = & \Big( \big(1, (1), (2), 1,(1) \big), \nonumber \\
&& \, \, \, \big(2, (1,1), (1, 2), 0 \big), \nonumber \\
&& \, \, \, \big(1, (1), (1), 1, (2) \big) \Big) \in \big( \tau_2)^3. \nonumber
\end{eqnarray}
For every $k \geq 2$,  we define a multi-dimensional vector $\lambda_{k+1}$ of elements in $\tau_{k+1}$ by the recurrence relation
\begin{eqnarray}
\lambda_{k+1} & : = & \Big( \big(1, (1), ( k+1), k,(1, \ldots, k) \big), \nonumber \\
&& \, \, \, \big(2, (k,1), (1, \ldots, k,k+1), 0 \big), \nonumber \\
&& \, \, \, \big(1, (k), (1, \ldots, k), 1, (k+1) \big), \nonumber \\
&& \, \, \,  T_k(\lambda^{(1)}_k), \ldots, T_k(\lambda^{(m(\lambda_k))}_k ) \Big), \nonumber
\end{eqnarray}
for $\lambda_{k} = \big(\lambda^{(1)}_k, \ldots,  \lambda^{(m(\lambda_k))}_k \big)$.
\end{definition}
\subsection{Analysis of higher order forward Kolmogorov equations}
In this subsection, we consider the following Cauchy problem (defined recursively by \eqref{ F lambda one dim}, \eqref{ F lambda multi dim}, Definition \ref{ def Tk}  and Definition \ref{ def  lambda k}):
\begin{equation}
\left\{
\begin{array}{rrl}
      {} & \partial_t m^{(k)} (t, \mu, \mu_1, \ldots, \mu_k)- \Delta m^{(k)} (t, \mu, \mu_1, \ldots, \mu_k)  + \text{div} \big( b(\cdot,m(t, \mu)) m^{(k)}(t, \mu, \mu_1, \ldots, \mu_k) \big)  \\
      & + \text{div} \Big(m(t, \mu) \frac{\delta b}{\delta m} (\cdot,m(t, \mu)) \big(m^{(k)}(t, \mu, \mu_1, \ldots, \mu_k) \big) \Big) - F_{\lambda_k}(t,\mu, \mu_1, \ldots, \mu_k)   & =0, \\
      & &  \\
      {} &  m^{(k)} (0, \mu, \mu_1, \ldots, \mu_k) &= \mu_k- \mu,\\
\end{array} 
\right. \label{PDE linearised k}  \end{equation} 
where, for $k=1$, $F_{\lambda_1}(t,\mu,\mu_1):=0. $
For $\lambda \in \tau_k$ given by \eqref{lambda def}, we define   
\begin{eqnarray}
F_{\lambda} (t, \mu, \mu_1, \ldots, \mu_k) & := &   -\text{div} \bigg[ m^{({\hat{\beta}})}  \Big(t, \mu, \mu_{{\hat{\alpha}}_1}, \ldots, \mu_{{\hat{\alpha}}_{\hat{\beta}}} \Big)  \,   \frac{ \delta^{{{\hat{n}}}} b}{ \delta m^{{{\hat{n}}}}}(\cdot, m(t, \mu)) \bigg( m^{(\beta_1)} \Big(t, \mu, \mu_{\alpha_{1,1}}, \ldots, \mu_{\alpha_{1,\beta_1}} \Big), \ldots, \nonumber \\
& & m^{(\beta_{{\hat{n}}})} \Big(t, \mu, \mu_{\alpha_{{\hat{n}},1}}, \ldots, \mu_{\alpha_{{\hat{n}},\beta_{\hat{n}}}} \Big) \bigg) \bigg]   \,. \label{ F lambda one dim} 
\end{eqnarray}
Note that  $F_{\lambda} (t, \mu, \mu_1, \ldots, \mu_k)$ can be interpreted as an element in the dual space $(W^{n+k-1, \infty}(\bT^d) )'$ (under the assumption \hregb{n+k-1}{k}):
\begin{eqnarray}
&& \lev \xi, F_{\lambda} (t,\mu, \mu_1, \ldots, \mu_k) \rev_{n+k-1, \infty}  \nonumber \\
& := &  \int_{\bT^d} \bigg[ \frac{ \delta^{{{\hat{n}}}} b}{ \delta m^{{{\hat{n}}}}}(x, m(t, \mu)) \bigg( m^{(\beta_1)} \Big(t, \mu, \mu_{\alpha_{1,1}}, \ldots, \mu_{\alpha_{1,\beta_1}} \Big), \ldots, \nonumber \\
& & m^{(\beta_{{\hat{n}}})} \Big(t, \mu, \mu_{\alpha_{{\hat{n}},1}}, \ldots, \mu_{\alpha_{{\hat{n}},\beta_{\hat{n}}}} \Big) \bigg) \cdot \nabla \xi (x) \bigg] \quad \, m^{({\hat{\beta}})}  \Big(t, \mu, \mu_{{\hat{\alpha}}_1}, \ldots, \mu_{{\hat{\alpha}}_{\hat{\beta}}} \Big) (dx)  . \label{F lambda dual}
\end{eqnarray}
For any $(\lambda^{(1)}, \ldots, \lambda^{(q)}) \in (\tau_k)^q$, we define
\begin{equation} F_{ (\lambda^{(1)}, \ldots, \lambda^{(q)})} (t, \mu, \mu_1, \ldots, \mu_k):= \sum_{\ell=1}^q F_{ \lambda^{(\ell)}} (t, \mu, \mu_1, \ldots, \mu_k). \label{ F lambda multi dim} \end{equation}
\begin{theorem} \label{ thm mk lipschitz} 
Let $k \in \bN$. Assume \emph{\hregb{n+k-1}{k}}, where $n \geq 2$. Then \eqref{F lambda dual} is well-defined and the Cauchy problem defined by \eqref{PDE linearised k} has a unique solution in $L^{\infty} \big([0,T], ( W^{n+k-1, \infty}(\mathbb{T}^d) )' \big)$ and satisfies 
\begin{equation}\sup_{\mu,\mu_1, \ldots, \mu_k \in \cP(\bT^d)} \sup_{t \in [0,T]} \Big\| m^{(k)}(t, \mu, \mu_1, \ldots, \mu_k) \Big\|_{-(n+k-1, \infty)} < \infty. \label{ m k bound} \end{equation}
Also, if we assume \emph{\hregb{n+k}{k+1}}, then
\begin{equation} \sup_{t \in [0,T]} \Big\| m^{(k)}(t, \mu_{k+1}, \mu_1, \ldots, \mu_k) - m^{(k)}(t, \mu, \mu_1, \ldots, \mu_k) \Big\|_{-(n+k, \infty)}\leq  C   W_1(\mu,\mu_{k+1}),  \label{eq: W1 mk}  \end{equation}  for any $\mu, \mu_1, \ldots, \mu_{k+1} \in \cP(\bT^d)$, for some constant $C>0$.  
\end{theorem}
\begin{proof}
We proceed by strong induction for \eqref{ m k bound}. The base step  follows clearly from \eqref{PDE linearised} and Theorem \ref{lions paper pde result}, since
$$ \| \mu_1 - \mu \|_{\na}  = \sup_{\| \xi \|_{\na} \leq 1} \bigg| \int_{\bT^d} \xi(x) \, (\mu_1-\mu)(dx) \bigg| \leq 2.  $$ 
Suppose that \eqref{ m k bound} holds for $\{1, \ldots, k-1\}$. Take any $\xi \in W^{{{n}}+k-1,\infty}(\bT^d)$ and $k \geq 2$.  We first show that \eqref{F lambda dual} is well-defined, i.e. $F_{\lambda} (t,\mu, \mu_1, \ldots, \mu_k)$ is indeed in $(W^{{{n}}+k-1,\infty}(\bT^d) )'$, for any $\lambda \in \tau_k$. Note that $\beta_1, \ldots, \beta_{\hat{n}},\hat{\beta} \leq k-1$, which implies by \eqref{eq: sum condition} and \eqref{eq: emptyset condition} that
\begin{eqnarray}
& & \bigg| \int_{\bT^d} \bigg[ \frac{ \delta^{{{\hat{n}}}} b}{ \delta m^{{{\hat{n}}}}}(x, m(t, \mu)) \bigg( m^{(\beta_1)} \Big(t, \mu, \mu_{\alpha_{1,1}}, \ldots, \mu_{\alpha_{1,\beta_1}} \Big), \ldots, \nonumber \\
& & m^{(\beta_{{\hat{n}}})} \Big(t, \mu, \mu_{\alpha_{{\hat{n}},1}}, \ldots, \mu_{\alpha_{{\hat{n}},\beta_{\hat{n}}}} \Big) \bigg) \cdot \nabla \xi(x) \bigg] \quad \, m^{({\hat{\beta}})}  \Big(t, \mu, \mu_{{\hat{\alpha}}_1}, \ldots, \mu_{{\hat{\alpha}}_{\hat{\beta}}} \Big) (dx) \bigg| \nonumber  \\
& \leq & C \bigg\|  \frac{ \delta^{{{\hat{n}}}} b}{ \delta m^{{{\hat{n}}}}}(\cdot, m(t, \mu)) \bigg( m^{(\beta_1)} \Big(t, \mu, \mu_{\alpha_{1,1}}, \ldots, \mu_{\alpha_{1,\beta_1}} \Big), \ldots, \nonumber \\
& & m^{(\beta_{{\hat{n}}})} \Big(t, \mu, \mu_{\alpha_{{\hat{n}},1}}, \ldots, \mu_{\alpha_{{\hat{n}},\beta_{\hat{n}}}} \Big) \bigg) \cdot \nabla \xi(\cdot) \bigg\|_{n+(k-1)-1,\infty} \quad \bigg\|   m^{({\hat{\beta}})}  \Big(t, \mu, \mu_{{\hat{\alpha}}_1}, \ldots, \mu_{{\hat{\alpha}}_{\hat{\beta}}} \Big) \bigg\|_{-(n+\hat{\beta}-1, \infty)} \nonumber \\
& \leq & C \| \xi \|_{n+k-1,\infty}, \label{ bound forward well defined} \end{eqnarray}
where the final step follows from {\hregb{n+k-1}{k}}. Therefore, the first statement that the Cauchy problem has a unique solution in $L^{\infty} \big([0,T], ( W^{n+k-1, \infty}(\mathbb{T}^d) )' \big)$ and \eqref{ m k bound} both follow directly from Theorem \ref{lions paper pde result}, by the assumption of \hregb{n+k-1}{k} and the fact that $F_{\lambda} (\cdot,\mu, \mu_1, \ldots, \mu_k)$ is in $L^{\infty}([0,T], (W^{{{n}}+k-1,\infty}(\bT^d) )')$.

It remains to prove \eqref{eq: W1 mk} under the stronger assumption  {\hregb{n+k}{k+1}}. Let ${\xi} \in W^{n+k,\infty}(\bT^d) $. Again, we proceed by strong induction. The base step is omitted as it is a special case of the procedure of the induction step. Suppose that \eqref{eq: W1 mk} holds for $\{1, \ldots, k\}$. Replacing $\mu$ by $\mu_{k+1}$ in \eqref{PDE linearised k}, we have 
\begin{eqnarray}
&&  \int_{\bT^d} \xi(y)  \, m^{(k)}(t,\mu_{k+1},\mu_1, \ldots, \mu_k ) (dy) - \int_{\bT^d} \xi(y)  \, m^{(k)}(0,\mu_{k+1},\mu_1, \ldots, \mu_k ) (dy)  \nonumber \\
& = & \int_0^t \int_{\bT^d} \Delta \xi(y)  \, m^{(k)}(s,\mu_{k+1},\mu_1, \ldots, \mu_k )(dy)  \,ds  \nonumber \\
&& +\int_0^t \int_{\bT^d} \Big[ b \big( y, m(s, \mu_{k+1}) \big) \cdot \nabla \xi(y) \Big]  \,m^{(k)}(s,\mu_{k+1},\mu_1, \ldots, \mu_k ) (dy)  \,ds \nonumber \\
&& + \int_0^t \int_{\bT^d}  \int_{\bT^d} \bigg[ \frac{\delta b}{\delta m} \big( x, m(s,\mu_{k+1}) \big)(y) \cdot \nabla \xi(x) \bigg] \, \big( m(s, \mu_{k+1}) \big) (dx) \,  m^{(k)}(s,\mu_{k+1},\mu_1, \ldots, \mu_k ) (dy)    \,ds \nonumber \\
&& + \int_0^t {\inp[\Big]{\xi} {F_{\lambda_k} (s,\mu_{k+1}, \mu_1, \ldots, \mu_k) }}_{n+k-1, \infty} \,ds \nonumber \\
& = & {\int_0^t \int_{\bT^d} \Delta \xi(y)  \, m^{(k)}(s,\mu_{k+1},\mu_1, \ldots, \mu_k ) (dy)  \,ds  } \nonumber \\
&& {+ \int_0^t \int_{\bT^d} \Big[ b \big( y, m(s, \mu) \big) \cdot \nabla \xi(y) \Big]  \,m^{(k)}(s,\mu_{k+1},\mu_1, \ldots, \mu_k ) (dy)  \,ds  }\nonumber \\
&& + \int_0^t \int_{\bT^d}  \int_{\bT^d} \bigg[ \frac{\delta b}{\delta m} \big( x, m(s,\mu) \big)(y) \cdot \nabla \xi(x) \bigg] \, \big( m(s, \mu) \big) (dx)  \,  m^{(k)}(s,\mu_{k+1},\mu_1, \ldots, \mu_k ) (dy)   \,ds \nonumber \\
&& + \int_0^t \int_{\bT^d} \Big[ \Big( b \big( y, m(s, \mu_{k+1}) \big) - b \big( y, m(s, \mu) \big) \Big) \cdot \nabla \xi(y) \Big]  \,m^{(k)}(s,\mu_{k+1},\mu_1, \ldots, \mu_k ) (dy)  \,ds \nonumber \\
&& + \int_0^t \int_{\bT^d}  \int_{\bT^d} \bigg[ \bigg(  \frac{\delta b}{\delta m} \big( x, m(s,\mu_{k+1}) \big)(y) \nonumber \\
&& - \frac{\delta b}{\delta m} \big( x, m(s,\mu) \big)(y) \bigg) \cdot \nabla \xi(x) \bigg] \, \big( m(s, \mu_{k+1}) \big) (dx) \,  m^{(k)}(s,\mu_{k+1},\mu_1, \ldots, \mu_k ) (dy)    \,ds \nonumber \\
&& + \int_0^t \int_{\bT^d}  \int_{\bT^d} \bigg[ \frac{\delta b}{\delta m} \big( x, m(s,\mu) \big)(y) \cdot \nabla \xi(x) \bigg] \, \big( m(s, \mu_{k+1}) - m(s, \mu) \big) (dx)  \,  m^{(k)}(s,\mu_{k+1},\mu_1, \ldots, \mu_k ) (dy)   \,ds \nonumber  \\
&& + \int_0^t {\inp[\Big]{\xi} {F_{\lambda_k} (s,\mu_{k+1}, \mu_1, \ldots, \mu_k) }}_{n+k-1, \infty} \,ds. \label{eq: mk mu k+1}
\end{eqnarray}
On the other hand, we have
\begin{eqnarray}
&&  \int_{\bT^d} \xi(y)  \, m^{(k)}(t,\mu,\mu_1, \ldots, \mu_k ) (dy) - \int_{\bT^d} \xi(y)  \, m^{(k)}(0,\mu,\mu_1, \ldots, \mu_k ) (dy)  \nonumber \\
& = & \int_0^t \int_{\bT^d} \Delta \xi(y)  \, m^{(1)}(s, \mu,\mu_1) (dy)  \,ds  \nonumber \\
&& +\int_0^t \int_{\bT^d} \Big[ b \big( y, m(s, \mu) \big) \cdot \nabla \xi(y) \Big]  \,m^{(k)}(s,\mu,\mu_1, \ldots, \mu_k ) (dy)  \,ds \nonumber \\
&& + \int_0^t \int_{\bT^d}  \int_{\bT^d} \bigg[ \frac{\delta b}{\delta m} \big( x, m(s,\mu) \big)(y) \cdot \nabla \xi(x) \bigg] \, \big( m(s, \mu) \big) (dx) \,  m^{(k)}(s,\mu,\mu_1, \ldots, \mu_k ) (dy)    \,ds \nonumber \\
&& + \int_0^t {\inp[\Big]{\xi} {F_{\lambda_k} (s,\mu, \mu_1, \ldots, \mu_k) }}_{n+k-1, \infty} \, ds. \label{eq: mk mu}
\end{eqnarray}
Next, we compute that
\begin{eqnarray}
&&  \, {\inp[\Big]{\xi} {F_{\lambda_k} (s,\mu_{k+1}, \mu_1, \ldots, \mu_k) }}_{n+k-1, \infty} -  \, {\inp[\Big]{\xi} {F_{\lambda_k} (s,\mu, \mu_1, \ldots, \mu_k) }}_{n+k-1, \infty} \nonumber \\
&=& \sum_{\lambda \in e(\lambda_k)} \Bigg[  \int_{\bT^d} \bigg[ \frac{ \delta^{{{\hat{n}}}} b}{ \delta m^{{{\hat{n}}}}}(x, m(s, \mu_{k+1})) \bigg( m^{(\beta_1)} \Big(s, \mu_{k+1}, \mu_{\alpha_{1,1}}, \ldots, \mu_{\alpha_{1,\beta_1}} \Big), \ldots, \nonumber \\
& & m^{(\beta_{{\hat{n}}})} \Big(s, \mu_{k+1}, \mu_{\alpha_{{\hat{n}},1}}, \ldots, \mu_{\alpha_{{\hat{n}},\beta_{\hat{n}}}} \Big) \bigg) \cdot \nabla \xi(x) \bigg] \quad \, m^{({\hat{\beta}})}  \Big(s, \mu_{k+1}, \mu_{{\hat{\alpha}}_1}, \ldots, \mu_{{\hat{\alpha}}_{\hat{\beta}}} \Big) (dx)  \nonumber \\
&& -\int_{\bT^d} \bigg[ \frac{ \delta^{{{\hat{n}}}} b}{ \delta m^{{{\hat{n}}}}}(x, m(s, \mu)) \bigg( m^{(\beta_1)} \Big(s, \mu, \mu_{\alpha_{1,1}}, \ldots, \mu_{\alpha_{1,\beta_1}} \Big), \ldots, \nonumber \\
& & m^{(\beta_{{\hat{n}}})} \Big(s, \mu, \mu_{\alpha_{{\hat{n}},1}}, \ldots, \mu_{\alpha_{{\hat{n}},\beta_{\hat{n}}}} \Big) \bigg) \cdot \nabla \xi(x) \bigg] \quad \, m^{({\hat{\beta}})}  \Big(s, \mu, \mu_{{\hat{\alpha}}_1}, \ldots, \mu_{{\hat{\alpha}}_{\hat{\beta}}} \Big) (dx) \Bigg] \nonumber  \\
& = & \sum_{\lambda \in e(\lambda_k)} \Bigg[ \int_{\bT^d} \bigg[ \bigg( \frac{ \delta^{{{\hat{n}}}} b}{ \delta m^{{{\hat{n}}}}}(x, m(s, \mu_{k+1})) - \frac{ \delta^{{{\hat{n}}}} b}{ \delta m^{{{\hat{n}}}}}(x, m(s, \mu)) \bigg) \bigg( m^{(\beta_1)} \Big(s, \mu, \mu_{\alpha_{1,1}}, \ldots, \mu_{\alpha_{1,\beta_1}} \Big), \ldots, \nonumber \\
& & m^{(\beta_{{\hat{n}}})} \Big(s, \mu, \mu_{\alpha_{{\hat{n}},1}}, \ldots, \mu_{\alpha_{{\hat{n}},\beta_{\hat{n}}}} \Big) \bigg) \cdot \nabla \xi(x) \bigg] \quad \, m^{({\hat{\beta}})}  \Big(s, \mu, \mu_{{\hat{\alpha}}_1}, \ldots, \mu_{{\hat{\alpha}}_{\hat{\beta}}} \Big) (dx) \nonumber \\
&& + \sum_{\ell=1}^{\hat{n}}  \int_{\bT^d} \bigg[ \frac{ \delta^{{{\hat{n}}}} b}{ \delta m^{{{\hat{n}}}}}(x, m(s, \mu_{k+1})) \bigg( m^{(\beta_1)} \Big(s, \mu_{k+1}, \mu_{\alpha_{1,1}}, \ldots, \mu_{\alpha_{1,\beta_1}} \Big), \ldots, \nonumber \\
&& m^{(\beta_{\ell-1})} \Big(s, \mu_{k+1}, \mu_{\alpha_{\ell-1,1}}, \ldots, \mu_{\alpha_{\ell-1,\beta_{\ell-1}}} \Big), \nonumber \\
&& m^{(\beta_{\ell})} \Big(s, \mu_{k+1}, \mu_{\alpha_{\ell,1}}, \ldots, \mu_{\alpha_{\ell,\beta_{\ell}}}\Big)  - m^{(\beta_{\ell})} \Big(s, \mu, \mu_{\alpha_{\ell,1}}, \ldots, \mu_{\alpha_{\ell,\beta_{\ell}}} \Big), m^{(\beta_{\ell+1})} \Big(s, \mu, \mu_{\alpha_{\ell+1,1}}, \ldots, \mu_{\alpha_{\ell+1,\beta_{\ell+1}}} \Big), \nonumber \\
&& \ldots, m^{(\beta_{{\hat{n}}})} \Big(s, \mu, \mu_{\alpha_{{\hat{n}},1}}, \ldots, \mu_{\alpha_{{\hat{n}},\beta_{{\hat{n}}}}} \Big) \bigg) \cdot \nabla \xi(x) \bigg] \quad \, m^{({\hat{\beta}})}  \Big(s, \mu, \mu_{{\hat{\alpha}}_1}, \ldots, \mu_{{\hat{\alpha}}_{\hat{\beta}}} \Big) (dx)  \nonumber \\
& & +  \int_{\bT^d} \bigg[ \frac{ \delta^{{{\hat{n}}}} b}{ \delta m^{{{\hat{n}}}}}(x, m(s, \mu_{k+1})) \bigg( m^{(\beta_1)} \Big(s, \mu_{k+1}, \mu_{\alpha_{1,1}}, \ldots, \mu_{\alpha_{1,\beta_1}} \Big), \ldots, m^{(\beta_{{\hat{n}}})} \Big(s, \mu_{k+1}, \mu_{\alpha_{{\hat{n}},1}}, \ldots, \mu_{\alpha_{{\hat{n}},\beta_{\hat{n}}}} \Big) \bigg) \nonumber \\
&& \cdot \nabla \xi(x) \bigg] \quad \, \bigg( m^{({\hat{\beta}})}  \Big(s, \mu_{k+1}, \mu_{{\hat{\alpha}}_1}, \ldots, \mu_{{\hat{\alpha}}_{\hat{\beta}}} \Big) - m^{({\hat{\beta}})}  \Big(s, \mu, \mu_{{\hat{\alpha}}_1}, \ldots, \mu_{{\hat{\alpha}}_{\hat{\beta}}} \Big) \bigg) (dx) \Bigg]. \label{F lambda k diff}
\end{eqnarray}
Note that the first term in \eqref{F lambda k diff} can be rewritten as 
\begin{eqnarray}
&& \int_{\bT^d} \bigg[ \bigg( \frac{ \delta^{{{\hat{n}}}} b}{ \delta m^{{{\hat{n}}}}}(x, m(s, \mu_{k+1})) - \frac{ \delta^{{{\hat{n}}}} b}{ \delta m^{{{\hat{n}}}}}(x, m(s, \mu)) \bigg) \bigg( m^{(\beta_1)} \Big(s, \mu, \mu_{\alpha_{1,1}}, \ldots, \mu_{\alpha_{1,\beta_1}} \Big), \ldots, \nonumber \\
& & m^{(\beta_{{\hat{n}}})} \Big(s, \mu, \mu_{\alpha_{{\hat{n}},1}}, \ldots, \mu_{\alpha_{{\hat{n}},\beta_{\hat{n}}}} \Big) \bigg) \cdot \nabla \xi(x) \bigg] \quad \, m^{({\hat{\beta}})}  \Big(s, \mu, \mu_{{\hat{\alpha}}_1}, \ldots, \mu_{{\hat{\alpha}}_{\hat{\beta}}} \Big) (dx) \nonumber \\
& = & \int_{\bT^d} \int_0^1 \bigg[ \bigg( \frac{ \delta^{{{\hat{n}+1}}} b}{ \delta m^{{{\hat{n}+1}}}}(x, u m(s, \mu_{k+1}))  + (1-u) m(s, \mu)) \bigg) \bigg( m^{(\beta_1)} \Big(s, \mu, \mu_{\alpha_{1,1}}, \ldots, \mu_{\alpha_{1,\beta_1}} \Big), \ldots, \nonumber \\
& & m^{(\beta_{{\hat{n}}})} \Big(s, \mu, \mu_{\alpha_{{\hat{n}},1}}, \ldots, \mu_{\alpha_{{\hat{n}},\beta_{\hat{n}}}} \Big),  m(s, \mu_{k+1})  - m(s, \mu)  \bigg) \cdot \nabla \xi(x) \bigg] \,  \, du \, \,  m^{({\hat{\beta}})}  \Big(s, \mu, \mu_{{\hat{\alpha}}_1}, \ldots, \mu_{{\hat{\alpha}}_{\hat{\beta}}} \Big) (dx), \nonumber \\
&& \label{eq: linear derivative higher order trick}
\end{eqnarray}
by which we can estimate by the assumption {\hregb{n+k}{k+1}}. For every $\lambda \in \tau_k$, we know that $\beta_1, \ldots, \beta_n, \hat{\beta} < k$ by definition. For $i \in \{1, \ldots, \hat{n} \}$ and $\hat{\mu} \in \{ \mu, \mu_{k+1} \},$
\begin{equation} 
\Big\| m^{(\beta_i)} \Big(s, \hat{\mu}, \mu_{\alpha_{i,1}}, \ldots, \mu_{\alpha_{i,\beta_i}} \Big)  \Big\|_{-(n+k-1,\infty)}
\leq  C \Big\| m^{(\beta_i)} \Big(s, \hat{\mu}, \mu_{\alpha_{i,1}}, \ldots, \mu_{\alpha_{i,\beta_i}} \Big)  \Big\|_{-(n+\beta_i-1,\infty)} \leq C. \label{ m beta i bound} 
\end{equation}
By the induction hypothesis,  for every $\beta_{\ell} < k$,
\begin{eqnarray}
&& \Big\| m^{(\beta_{\ell})} \Big(s, \mu_{k+1}, \mu_{\alpha_{\ell,1}}, \ldots, \mu_{\alpha_{\ell,\beta_{\ell}}}\Big)  - m^{(\beta_{\ell})} \Big(s, \mu, \mu_{\alpha_{\ell,1}}, \ldots, \mu_{\alpha_{\ell,\beta_{\ell}}} \Big) \Big\|_{-(n+k-1,\infty)} \nonumber \\
& \leq & C \Big\| m^{(\beta_{\ell})} \Big(s, \mu_{k+1}, \mu_{\alpha_{\ell,1}}, \ldots, \mu_{\alpha_{\ell,\beta_{\ell}}}\Big)  - m^{(\beta_{\ell})} \Big(s, \mu, \mu_{\alpha_{\ell,1}}, \ldots, \mu_{\alpha_{\ell,\beta_{\ell}}} \Big) \Big\|_{-(n+\beta_{\ell},\infty)} \nonumber \\
& \leq & C W_1(\mu,\mu_{k+1}). \label{ind hypo diff 1} 
\end{eqnarray}
Similarly, by the induction hypothesis, for $\hat{\beta} < k$, 
\begin{eqnarray}
&& \Big\| m^{({\hat{\beta}})}  \Big(s, \mu_{k+1}, \mu_{{\hat{\alpha}}_1}, \ldots, \mu_{{\hat{\alpha}}_{\hat{\beta}}} \Big) - m^{({\hat{\beta}})}  \Big(s, \mu, \mu_{{\hat{\alpha}}_1}, \ldots, \mu_{{\hat{\alpha}}_{\hat{\beta}}} \Big) \Big\|_{-(n+k-1,\infty)} \nonumber \\
& \leq & C  W_1(\mu,\mu_{k+1}). \label{ind hypo diff 2} 
\end{eqnarray}
Hence, by \eqref{F lambda k diff}, \eqref{eq: linear derivative higher order trick}, \eqref{ m beta i bound}, \eqref{ind hypo diff 1},   \eqref{ind hypo diff 2} and the assumption of {\hregb{n+k}{k+1}}, we obtain that 
\begin{equation} \sup_{s \in [0,T]} \bigg|  {\inp[\Big]{\xi} {F_{\lambda_k} (s,\mu_{k+1}, \mu_1, \ldots, \mu_k) }}_{n+k-1,\infty} -  \, {\inp[\Big]{\xi} {F_{\lambda_k} (s,\mu, \mu_1, \ldots, \mu_k) }}_{n+k-1,\infty} \bigg|  \leq C \| \xi \|_{n+ k-1,\infty} W_1(\mu, \mu_{k+1}). \label{ind hypo W1 bound lower order terms} \end{equation}
Let \begin{equation} d^{(k+1)} (t,\mu,\mu_1, \ldots, \mu_k, \mu_{k+1} ):= m^{(k)} (t,\mu_{k+1},\mu_1, \ldots, \mu_k)- m^{(k)} (t,\mu,\mu_1, \ldots, \mu_k ) . \label{d k+1 def} \end{equation}
Subtracting \eqref{eq: mk mu k+1} by \eqref{eq: mk mu} gives
\begin{eqnarray}
&&  \int_{\bT^d} \xi(y)  \, d^{(k+1)} (t,\mu,\mu_1, \ldots, \mu_k, \mu_{k+1} ) (dy) - \int_{\bT^d} \xi(y)  \, d^{(k+1)} (0,\mu,\mu_1, \ldots, \mu_k, \mu_{k+1} )(dy)  \nonumber \\
& = & {\int_0^t \int_{\bT^d} \Delta \xi(y)  \, d^{(k+1)} (s,\mu,\mu_1, \ldots, \mu_k, \mu_{k+1} ) (dy)  \,ds  } \nonumber \\
&& {+ \int_0^t \int_{\bT^d} \Big[ b \big( y, m(s, \mu) \big) \cdot \nabla \xi(y) \Big]  \,d^{(k+1)} (s,\mu,\mu_1, \ldots, \mu_k, \mu_{k+1} )  (dy)  \,ds  }\nonumber \\
&& + \int_0^t \int_{\bT^d}  \int_{\bT^d} \bigg[ \frac{\delta b}{\delta m} \big( x, m(s,\mu) \big)(y) \cdot \nabla \xi(x) \bigg] \, \big( m(s, \mu) \big) (dx)  \,  d^{(k+1)} (s,\mu,\mu_1, \ldots, \mu_k, \mu_{k+1} )  (dy)   \,ds \nonumber \\
&& + \int_0^t \int_{\bT^d} \Big[ \Big( b \big( y, m(s, \mu_{k+1}) \big) - b \big( y, m(s, \mu) \big) \Big) \cdot \nabla \xi(y) \Big]  \,m^{(k)}(s,\mu_{k+1},\mu_1, \ldots, \mu_k ) (dy)  \,ds \nonumber \\
&& + \int_0^t \int_{\bT^d}  \int_{\bT^d} \bigg[ \bigg(  \frac{\delta b}{\delta m} \big( x, m(s,\mu_{k+1}) \big)(y) \nonumber \\
&& - \frac{\delta b}{\delta m} \big( x, m(s,\mu) \big)(y) \bigg) \cdot \nabla \xi(x) \bigg] \, \big( m(s, \mu_{k+1}) \big) (dx) \,  m^{(k)}(s,\mu_{k+1},\mu_1, \ldots, \mu_k ) (dy)    \,ds \nonumber \\
&& + \int_0^t \int_{\bT^d}  \int_{\bT^d} \bigg[ \frac{\delta b}{\delta m} \big( x, m(s,\mu) \big)(y) \cdot \nabla \xi(x) \bigg] \, \big( m(s, \mu_{k+1}) - m(s, \mu) \big) (dx)  \,  m^{(k)}(s,\mu_{k+1},\mu_1, \ldots, \mu_k ) (dy)   \,ds \nonumber  \\
&& + \int_0^t \bigg[ {\inp[\Big]{\xi} {F_{\lambda_k} (s,\mu_{k+1}, \mu_1, \ldots, \mu_k) }}_{n+k-1,\infty} \nonumber \\
&& \, \quad \quad \quad - {\inp[\Big]{\xi} {F_{\lambda_k} (s,\mu, \mu_1, \ldots, \mu_k) }}_{n+k-1,\infty} \bigg] \, ds. \nonumber \\
&& \label{dk+1 formula}
\end{eqnarray}
 Let $\eta^{(k+1)} (t,\mu,\mu_1, \ldots, \mu_k, \mu_{k+1} ) $ be an element in the dual space $(W^{n+k,\infty}(\bT^d) )'$ defined by
\begin{eqnarray}
&& \, {\inp[\Big]{{\xi}}{ \eta^{(k+1)} (t,\mu,\mu_1, \ldots, \mu_k, \mu_{k+1} ) }}_{n+k+\infty}  \nonumber \\
&:=&   \int_{\bT^d} \Big[ \Big( b \big( y, m(t, \mu_{k+1}) \big) - b \big( y, m(t, \mu) \big) \Big) \cdot \nabla {\xi}(y) \Big]  \,m^{(k)}(t,\mu_{k+1},\mu_1, \ldots, \mu_k ) (dy)   \nonumber \\
&& +  \int_{\bT^d}  \int_{\bT^d} \bigg[ \bigg(  \frac{\delta b}{\delta m} \big( x, m(t,\mu_{k+1}) \big)(y) \nonumber \\
&& - \frac{\delta b}{\delta m} \big( x, m(t,\mu) \big)(y) \bigg) \cdot \nabla {\xi}(x) \bigg] \, \big( m(t, \mu_{k+1}) \big) (dx) \,  m^{(k)}(t,\mu_{k+1},\mu_1, \ldots, \mu_k ) (dy)     \nonumber \\
&& +  \int_{\bT^d}  \int_{\bT^d} \bigg[ \frac{\delta b}{\delta m} \big( x, m(t,\mu) \big)(y) \cdot \nabla {\xi}(x) \bigg] \, \big( m(t, \mu_{k+1}) - m(t, \mu) \big) (dx)  \,  m^{(k)}(t,\mu_{k+1},\mu_1, \ldots, \mu_k ) (dy)    \nonumber  \\
&& + \, {\inp[\Big]{{\xi}} {F_{\lambda_k} (t,\mu_{k+1}, \mu_1, \ldots, \mu_k) }}_{n+k-1,\infty} - {\inp[\Big]{{\xi}} {F_{\lambda_k} (t,\mu, \mu_1, \ldots, \mu_k) }}_{n+k-1,\infty}. \nonumber
\end{eqnarray}
Clearly, by  {\hregb{n+k}{k+1}} and \eqref{ind hypo W1 bound lower order terms}, it follows from the same argument as Lemma \ref{verify regularity of c} to deduce that 
\begin{equation}
 \sup_{t \in [0,T]}  \| \eta^{(k+1)} (t,\mu,\mu_1, \ldots, \mu_k, \mu_{k+1} ) \|_{-(n+k ,\infty)} \leq  C  W_1(\mu, \mu_{k+1}). \label{ eta k+1 bound}
\end{equation}
By \eqref{dk+1 formula} (and replacing $\xi$ by arbitrary test functions $\phi \in  C^{\infty}([0,T] \times \bT^d)$) we note that $d^{(k+1)}$ satisfies the Cauchy problem 
\begin{equation}
\left\{
\begin{array}{rrl}
      {} & \partial_t d^{(k+1)} (t,\mu,\mu_1, \ldots, \mu_k, \mu_{k+1} )- \Delta d^{(k+1)} (t,\mu,\mu_1, \ldots, \mu_k, \mu_{k+1} ) \\
      & + \text{div} \big( b(\cdot,m(t, \mu)) d^{(k+1)} (t,\mu,\mu_1, \ldots, \mu_k, \mu_{k+1} )\big)  \\
      & +\text{div}\Big(m(t, \mu) \frac{\delta b}{\delta m} (\cdot,m(t, \mu)) \big(d^{(k+1)} (t,\mu,\mu_1, \ldots, \mu_k, \mu_{k+1} ) \big) \Big)  \\
      &- \eta^{(k+1)} (t,\mu,\mu_1, \ldots, \mu_k, \mu_{k+1} )   & =0, \\
      & &  \\
      {} &  d^{(k+1)} (0,\mu,\mu_1, \ldots, \mu_k, \mu_{k+1} ) &= \mu_{k+1}- \mu.\\
\end{array} 
\right.   \end{equation}
Therefore, by Theorem \ref{lions paper pde result} and \regb[n+k],
\begin{eqnarray} 
&& \sup_{t \in [0,T]} \| d^{(k+1)} (t,\mu,\mu_1, \ldots, \mu_k, \mu_{k+1} )  \|_{-(n+ k, \infty)} \nonumber \\
& \leq & C  \Big(  \| \mu_{k+1}- \mu  \|_{-(n+k, \infty)} +  \sup_{t \in [0,T]} \| \eta^{(k+1)} (t,\mu,\mu_1, \ldots, \mu_k, \mu_{k+1} )  \|_{-(n+ k,\infty)} \Big). \nonumber 
\end{eqnarray}
This completes the proof by \eqref{ eta k+1 bound}.
\end{proof}
\begin{theorem} \label{ mk mk+1 pertub} 
Let  $k \in \bN \cup \{0 \}$. Assume  \emph{\hregb{n+k+1}{k+1}} and \emph{\hlipb{n+k}{k+1}}, where $n \geq 2$. Then 
\begin{eqnarray}  
&& \sup_{t \in [0,T]} \Big\| m^{(k)}(t, \mu_{k+1}, \mu_1, \ldots, \mu_k) - m^{(k)}(t, \mu, \mu_1, \ldots, \mu_k) - m^{(k+1)}(t, \mu, \mu_1, \ldots, \mu_k, \mu_{k+1})\Big\|_{-(n+k+1, \infty)} \nonumber \\
& \leq & C  W_1(\mu, \mu_{k+1})^2,  \label{eq: linearisation formula mk}  \end{eqnarray} for any $\mu, \mu_1, \ldots, \mu_{k+1} \in \cP(\bT^d)$, for some constant $C>0$. 
\end{theorem}
\begin{proof}
We proceed by strong induction.  The base case is done in Theorem \ref{thm m m-one}.  Assume that the theorem holds for $\{1, \ldots, k-1 \}$, where $k \geq 2$. Then
\begin{eqnarray}  
&& \sup_{t \in [0,T]} \Big\| m^{(\ell)}(t, \mu_{\ell+1}, \mu_1, \ldots, \mu_{\ell}) - m^{(\ell)}(t, \mu, \mu_1, \ldots, \mu_{\ell}) - m^{({\ell}+1)}(t, \mu, \mu_1, \ldots, \mu_{\ell}, \mu_{{\ell}+1})\Big\|_{-(n+\ell + 1,\infty)} \nonumber \\
& \leq & C  W_1(\mu, \mu_{{\ell}+1})^2,   \quad \forall \ell \in \{1, \ldots, k-1 \}. \label{eq: linearisation formula m ell induction hypothesis}  \end{eqnarray}
Take $\xi \in W^{n+k+1, \infty}(\bT^d)$. We first recall from the definition of $\lambda_{k+1}$ (given in Definition \ref{ def  lambda k}) that  the PDE for $m^{(k+1)}$ is given by
\begin{eqnarray}
&&  \int_{\bT^d} \xi(y)  \, m^{(k+1)}(t,\mu,\mu_1,\ldots, \mu_{k+1}) (dy) - \int_{\bT^d} \xi(y)  \, m^{(k+1)}(0,\mu,\mu_1,\ldots, \mu_{k+1}) (dy)  \nonumber \\
& = & {\int_0^t \int_{\bT^d} \Delta \xi(y)  \, m^{(k+1)}(s,\mu,\mu_1,\ldots, \mu_{k+1}) (dy)  \,ds} \nonumber \\
&& {+\int_0^t \int_{\bT^d} \Big[ b \big( y, m(s, \mu) \big) \cdot \nabla \xi(y) \Big]  \,m^{(k+1)}(s,\mu,\mu_1,\ldots, \mu_{k+1})(dy)  \,ds} \nonumber \\
&&  {+ \int_0^t \int_{\bT^d}  \int_{\bT^d} \bigg[ \frac{\delta b}{\delta m} \big( x, m(s,\mu) \big)(y) \cdot \nabla \xi(x) \bigg] \, \big( m(s, \mu) \big) (dx) \,  m^{(k+1)}(s,\mu,\mu_1,\ldots, \mu_{k+1})(dy)    \,ds} \nonumber \\
&& +  \int_0^t \int_{\bT^d} \intrd \bigg[\ld[b] \big( y, m(s, \mu) \big)(z)  \cdot \nabla \xi(y) \bigg] \, m^{(1)}(s,\mu,\mu_{k+1})(dz) \, \,m^{(k)}(s,\mu,\mu_1,\ldots, \mu_{k}) (dy)  \,ds \nonumber \\
&& + \int_0^t \int_{\bT^d}  \int_{\bT^d} \int_{\bT^d} \bigg[ \frac{\delta^2 b}{\delta m^2} \big( x, m(s,\mu) \big)(y,z) \cdot \nabla \xi(x) \bigg] \nonumber \\
&& \, \,m^{(1)}(s,\mu,\mu_{k+1}) (dz) \,  \,m^{(k)}(s,\mu,\mu_1, \ldots, \mu_k) (dy)   \, \big( m(s, \mu) \big) (dx) \,ds \nonumber   \\
&& +  \int_0^t \int_{\bT^d}  \int_{\bT^d} \bigg[ \frac{\delta b}{\delta m} \big( x, m(s,\mu) \big)(y) \cdot \nabla \xi(x) \bigg] \,  m^{(k)}(s,\mu,\mu_1,\ldots, \mu_{k}) (dy)   \,  m^{(1)}(s,\mu,\mu_{k+1})  (dx) \,ds. \nonumber \\
& & + \sum_{\lambda \in e(\lambda_k)} \Bigg[ \int_0^t \int_{\bT^d} \bigg[  \frac{ \delta^{{{\hat{n}}+1}} b}{ \delta m^{{{\hat{n}}+1}}}(x, m(s, \mu)) \bigg( m^{(\beta_1)} \Big(s, \mu, \mu_{\alpha_{1,1}}, \ldots, \mu_{\alpha_{1,\beta_1}} \Big), \ldots, \nonumber \\
& & m^{(\beta_{{\hat{n}}})} \Big(s, \mu, \mu_{\alpha_{{\hat{n}},1}}, \ldots, \mu_{\alpha_{{\hat{n}},\beta_{\hat{n}}}} \Big), m^{(1)}(s, \mu, \mu_{k+1}) \bigg) \cdot \nabla \xi(x) \bigg] \quad \, m^{({\hat{\beta}})}  \Big(s, \mu, \mu_{{\hat{\alpha}}_1}, \ldots, \mu_{{\hat{\alpha}}_{\hat{\beta}}} \Big) (dx) \,ds \nonumber \\
&& +  \sum_{\ell=1}^{\hat{n}} \int_0^t \int_{\bT^d} \bigg[ \frac{ \delta^{{{\hat{n}}}} b}{ \delta m^{{{\hat{n}}}}}(x, m(s, \mu)) \bigg( m^{(\beta_1)} \Big(s, \mu, \mu_{\alpha_{1,1}}, \ldots, \mu_{\alpha_{1,\beta_1}} \Big), \ldots, m^{(\beta_{\ell-1})} \Big(s, \mu, \mu_{\alpha_{\ell-1,1}}, \ldots, \mu_{\alpha_{\ell-1,\beta_{\ell-1}}} \Big), \nonumber \\
&&  m^{(\beta_{\ell}+1)} \Big(s, \mu, \mu_{\alpha_{\ell,1}}, \ldots, \mu_{\alpha_{\ell,\beta_{\ell}}} , \mu_{k+1} \Big), m^{(\beta_{\ell+1})} \Big(s, \mu, \mu_{\alpha_{\ell+1,1}}, \ldots, \mu_{\alpha_{\ell+1,\beta_{\ell+1}}} \Big), \nonumber \\
&& \ldots, m^{(\beta_{{\hat{n}}})} \Big(s, \mu, \mu_{\alpha_{{\hat{n}},1}}, \ldots, \mu_{\alpha_{{\hat{n}},\beta_{{\hat{n}}}}} \Big) \bigg) \cdot \nabla \xi(x) \bigg] \quad \, m^{({\hat{\beta}})}  \Big(s, \mu, \mu_{{\hat{\alpha}}_1}, \ldots, \mu_{{\hat{\alpha}}_{\hat{\beta}}} \Big) (dx) \,ds \nonumber \\
& & +  \int_0^t \int_{\bT^d} \bigg[ \frac{ \delta^{{{\hat{n}}}} b}{ \delta m^{{{\hat{n}}}}}(x, m(s, \mu)) \bigg( m^{(\beta_1)} \Big(s, \mu, \mu_{\alpha_{1,1}}, \ldots, \mu_{\alpha_{1,\beta_1}} \Big), \ldots, m^{(\beta_{{\hat{n}}})} \Big(s, \mu, \mu_{\alpha_{{\hat{n}},1}}, \ldots, \mu_{\alpha_{{\hat{n}},\beta_{\hat{n}}}} \Big) \bigg) \nonumber \\
&& \cdot \nabla \xi(x) \bigg] \quad \, \bigg(  m^{({\hat{\beta}}+1)}  \Big(s, \mu, \mu_{{\hat{\alpha}}_1}, \ldots, \mu_{{\hat{\alpha}}_{\hat{\beta}}},\mu_{k+1} \Big) \bigg) (dx) \,ds \Bigg]. \label{m k+1 expression} \end{eqnarray}
Recalling the definition of $d^{(k+1)}$ in \eqref{d k+1 def}, we define
\begin{equation*} \rho^{(k+1)} (t,\mu,\mu_1, \ldots, \mu_k, \mu_{k+1} ):= d^{(k+1)} (t,\mu,\mu_1, \ldots, \mu_k, \mu_{k+1} )  - m^{(k+1)} (t,\mu,\mu_1, \ldots, \mu_k, \mu_{k+1} ). \end{equation*}
Subtracting \eqref{dk+1 formula} by \eqref{m k+1 expression} (and replacing $\xi$ by arbitrary test functions $\phi \in  C^{\infty}([0,T] \times \bT^d)$), we observe that ${\rho}^{(k+1)}$ satisfies the Cauchy problem 
\begin{equation}
\left\{
\begin{array}{rrl}
      {} & \partial_t {\rho}^{(k+1)} (t,\mu,\mu_1, \ldots, \mu_k, \mu_{k+1} )- \Delta {\rho}^{(k+1)} (t,\mu,\mu_1, \ldots, \mu_k, \mu_{k+1} ) \\
      & + \text{div} \big( b(\cdot,m(t, \mu)) {\rho}^{(k+1)} (t,\mu,\mu_1, \ldots, \mu_k, \mu_{k+1} )\big)  \\
      & + \text{div} \Big(m(t, \mu) \frac{\delta b}{\delta m} (\cdot,m(t, \mu)) \big({\rho}^{(k+1)} (t,\mu,\mu_1, \ldots, \mu_k, \mu_{k+1} ) \big) \Big) \\
      & - c^{(k+1)} (t,\mu,\mu_1, \ldots, \mu_k, \mu_{k+1} )   & =0, \\
      & &  \\
      {} & {\rho}^{(k+1)} (0,\mu,\mu_1, \ldots, \mu_k, \mu_{k+1} ) &= 0,\\
\end{array} 
\right.  \label{rho k+1  problem}  \end{equation}
where $$c^{(k+1)} (t,\mu,\mu_1, \ldots, \mu_k, \mu_{k+1} ):= \sum_{i=1}^4 c_i^{(k+1)} (t,\mu,\mu_1, \ldots, \mu_k, \mu_{k+1} ),$$
and $c_i^{(k+1)} (t,\mu,\mu_1, \ldots, \mu_k, \mu_{k+1} ), i \in \{1, \ldots, 4\}$, are elements in the dual space $(W^{n+k+1,\infty}(\bT^d) )'$ defined by
\begin{eqnarray}
&& \, {\inp[\Big]{\xi}{ c_1^{(k+1)} (t,\mu,\mu_1, \ldots, \mu_k, \mu_{k+1} ) }}_{n+k+1,\infty}  \nonumber \\
&:=&   \int_{\bT^d} \Big[ \Big( b \big( y, m(t, \mu_{k+1}) \big) - b \big( y, m(t, \mu) \big) \Big) \cdot \nabla \xi(y) \Big]  \,m^{(k)}(t,\mu_{k+1},\mu_1, \ldots, \mu_k ) (dy)   \nonumber \\
&& -\int_{\bT^d} \intrd \bigg[\ld[b] \big( y, m(t, \mu) \big)(z) \cdot \nabla \xi(y) \bigg] \, m^{(1)}(t,\mu,\mu_{k+1})(dz) \, \,m^{(k)}(t,\mu,\mu_1,\ldots, \mu_{k}) (dy)\, , \nonumber  \\
&& \nonumber \\
&& \nonumber \\
&& \, {\inp[\Big]{\xi}{ c_2^{(k+1)} (t,\mu,\mu_1, \ldots, \mu_k, \mu_{k+1} ) }}_{n+k+1,\infty}  \nonumber \\
& : = &   \int_{\bT^d}  \int_{\bT^d} \bigg[ \bigg(  \frac{\delta b}{\delta m} \big( x, m(t,\mu_{k+1}) \big)(y) \nonumber \\
&& - \frac{\delta b}{\delta m} \big( x, m(t,\mu) \big)(y) \bigg) \cdot \nabla \xi(x) \bigg] \, \big( m(t, \mu_{k+1}) \big) (dx) \,  m^{(k)}(t,\mu_{k+1},\mu_1, \ldots, \mu_k ) (dy)     \nonumber \\
&& - \int_{\bT^d}  \int_{\bT^d} \int_{\bT^d} \bigg[ \frac{\delta^2 b}{\delta m^2} \big( x, m(t,\mu) \big)(y,z) \cdot \nabla \xi(x) \bigg] \nonumber \\
&& \, \,m^{(1)}(t,\mu,\mu_{k+1}) (dz) \,  \,m^{(k)}(t,\mu,\mu_1, \ldots, \mu_k) (dy)   \, \big( m(t, \mu) \big) (dx) \, , \nonumber \\
&& \nonumber \\
&& \nonumber \\
&& \, {\inp[\Big]{\xi}{ c_3^{(k+1)} (t,\mu,\mu_1, \ldots, \mu_k, \mu_{k+1} ) }}_{n+k+1,\infty}  \nonumber \\
& : = &   \int_{\bT^d}  \int_{\bT^d} \bigg[ \frac{\delta b}{\delta m} \big( x, m(t,\mu) \big)(y) \cdot \nabla \xi(x) \bigg] \, \big( m(t, \mu_{k+1}) - m(t, \mu) \big) (dx)  \,  m^{(k)}(t,\mu_{k+1},\mu_1, \ldots, \mu_k ) (dy)    \nonumber \\
&& - \int_{\bT^d}  \int_{\bT^d} \bigg[ \frac{\delta b}{\delta m} \big( x, m(t,\mu) \big)(y) \cdot \nabla \xi(x) \bigg] \,  m^{(k)}(t,\mu,\mu_1,\ldots, \mu_{k}) (dy)   \,  m^{(1)}(t,\mu,\mu_{k+1})  (dx)\, , \nonumber  \end{eqnarray}
and, by \eqref{F lambda k diff}, 
\begin{eqnarray}
&& \, {\inp[\Big]{\xi}{ c_4^{(k+1)} (t,\mu,\mu_1, \ldots, \mu_k, \mu_{k+1} ) }}_{n+k+1,\infty}  \nonumber \\
& : = &  \sum_{\lambda \in e(\lambda_k)} \Bigg[ \int_{\bT^d} \bigg[ \bigg( \frac{ \delta^{{{\hat{n}}}} b}{ \delta m^{{{\hat{n}}}}}(x, m(t, \mu_{k+1})) - \frac{ \delta^{{{\hat{n}}}} b}{ \delta m^{{{\hat{n}}}}}(x, m(t, \mu)) \bigg) \bigg( m^{(\beta_1)} \Big(t, \mu, \mu_{\alpha_{1,1}}, \ldots, \mu_{\alpha_{1,\beta_1}} \Big), \ldots, \nonumber \\
& & m^{(\beta_{{\hat{n}}})} \Big(t, \mu, \mu_{\alpha_{{\hat{n}},1}}, \ldots, \mu_{\alpha_{{\hat{n}},\beta_{\hat{n}}}} \Big) \bigg) \cdot \nabla \xi(x) \bigg] \quad \, m^{({\hat{\beta}})}  \Big(t, \mu, \mu_{{\hat{\alpha}}_1}, \ldots, \mu_{{\hat{\alpha}}_{\hat{\beta}}} \Big) (dx) \nonumber \\
&& + \sum_{\ell=1}^{\hat{n}}  \int_{\bT^d} \bigg[ \frac{ \delta^{{{\hat{n}}}} b}{ \delta m^{{{\hat{n}}}}}(x, m(t, \mu_{k+1})) \bigg( m^{(\beta_1)} \Big(t, \mu_{k+1}, \mu_{\alpha_{1,1}}, \ldots, \mu_{\alpha_{1,\beta_1}} \Big), \ldots, \nonumber \\
&& m^{(\beta_{\ell-1})} \Big(t, \mu_{k+1}, \mu_{\alpha_{\ell-1,1}}, \ldots, \mu_{\alpha_{\ell-1,\beta_{\ell-1}}} \Big), \nonumber \\
&& m^{(\beta_{\ell})} \Big(t, \mu_{k+1}, \mu_{\alpha_{\ell,1}}, \ldots, \mu_{\alpha_{\ell,\beta_{\ell}}}\Big)  - m^{(\beta_{\ell})} \Big(t, \mu, \mu_{\alpha_{\ell,1}}, \ldots, \mu_{\alpha_{\ell,\beta_{\ell}}} \Big), m^{(\beta_{\ell+1})} \Big(t, \mu, \mu_{\alpha_{\ell+1,1}}, \ldots, \mu_{\alpha_{\ell+1,\beta_{\ell+1}}} \Big), \nonumber \\
&& \ldots, m^{(\beta_{{\hat{n}}})} \Big(t, \mu, \mu_{\alpha_{{\hat{n}},1}}, \ldots, \mu_{\alpha_{{\hat{n}},\beta_{{\hat{n}}}}} \Big) \bigg) \cdot \nabla \xi(x) \bigg] \quad \, m^{({\hat{\beta}})}  \Big(t, \mu, \mu_{{\hat{\alpha}}_1}, \ldots, \mu_{{\hat{\alpha}}_{\hat{\beta}}} \Big) (dx)  \nonumber \\
& & +  \int_{\bT^d} \bigg[ \frac{ \delta^{{{\hat{n}}}} b}{ \delta m^{{{\hat{n}}}}}(x, m(t, \mu_{k+1})) \bigg( m^{(\beta_1)} \Big(t, \mu_{k+1}, \mu_{\alpha_{1,1}}, \ldots, \mu_{\alpha_{1,\beta_1}} \Big), \ldots, m^{(\beta_{{\hat{n}}})} \Big(t, \mu_{k+1}, \mu_{\alpha_{{\hat{n}},1}}, \ldots, \mu_{\alpha_{{\hat{n}},\beta_{\hat{n}}}} \Big) \bigg) \nonumber \\
&& \cdot \nabla \xi(x) \bigg] \quad \, \bigg( m^{({\hat{\beta}})}  \Big(t, \mu_{k+1}, \mu_{{\hat{\alpha}}_1}, \ldots, \mu_{{\hat{\alpha}}_{\hat{\beta}}} \Big) - m^{({\hat{\beta}})}  \Big(t, \mu, \mu_{{\hat{\alpha}}_1}, \ldots, \mu_{{\hat{\alpha}}_{\hat{\beta}}} \Big) \bigg) (dx) \Bigg] \nonumber \\
&& - \sum_{\lambda \in e(\lambda_k)} \Bigg[ \int_{\bT^d} \bigg[  \frac{ \delta^{{{\hat{n}}+1}} b}{ \delta m^{{{\hat{n}}+1}}}(x, m(t, \mu)) \bigg( m^{(\beta_1)} \Big(t, \mu, \mu_{\alpha_{1,1}}, \ldots, \mu_{\alpha_{1,\beta_1}} \Big), \ldots, \nonumber \\
& & m^{(\beta_{{\hat{n}}})} \Big(t, \mu, \mu_{\alpha_{{\hat{n}},1}}, \ldots, \mu_{\alpha_{{\hat{n}},\beta_{\hat{n}}}} \Big), m^{(1)}(t, \mu, \mu_{k+1}) \bigg) \cdot \nabla \xi(x) \bigg] \quad \, m^{({\hat{\beta}})}  \Big(t, \mu, \mu_{{\hat{\alpha}}_1}, \ldots, \mu_{{\hat{\alpha}}_{\hat{\beta}}} \Big) (dx)  \nonumber \\
&& +  \sum_{\ell=1}^{\hat{n}}  \int_{\bT^d} \bigg[ \frac{ \delta^{{{\hat{n}}}} b}{ \delta m^{{{\hat{n}}}}}(x, m(t, \mu)) \bigg( m^{(\beta_1)} \Big(t, \mu, \mu_{\alpha_{1,1}}, \ldots, \mu_{\alpha_{1,\beta_1}} \Big), \ldots, m^{(\beta_{\ell-1})} \Big(t, \mu, \mu_{\alpha_{\ell-1,1}}, \ldots, \mu_{\alpha_{\ell-1,\beta_{\ell-1}}} \Big), \nonumber \\
&&  m^{(\beta_{\ell}+1)} \Big(t, \mu, \mu_{\alpha_{\ell,1}}, \ldots, \mu_{\alpha_{\ell,\beta_{\ell}}} , \mu_{k+1} \Big), m^{(\beta_{\ell+1})} \Big(t, \mu, \mu_{\alpha_{\ell+1,1}}, \ldots, \mu_{\alpha_{\ell+1,\beta_{\ell+1}}} \Big), \nonumber \\
&& \ldots, m^{(\beta_{{\hat{n}}})} \Big(t, \mu, \mu_{\alpha_{{\hat{n}},1}}, \ldots, \mu_{\alpha_{{\hat{n}},\beta_{{\hat{n}}}}} \Big) \bigg) \cdot \nabla \xi(x) \bigg] \quad \, m^{({\hat{\beta}})}  \Big(t, \mu, \mu_{{\hat{\alpha}}_1}, \ldots, \mu_{{\hat{\alpha}}_{\hat{\beta}}} \Big) (dx)  \nonumber \\
& & + \int_{\bT^d} \bigg[ \frac{ \delta^{{{\hat{n}}}} b}{ \delta m^{{{\hat{n}}}}}(x, m(t, \mu)) \bigg( m^{(\beta_1)} \Big(t, \mu, \mu_{\alpha_{1,1}}, \ldots, \mu_{\alpha_{1,\beta_1}} \Big), \ldots, m^{(\beta_{{\hat{n}}})} \Big(t, \mu, \mu_{\alpha_{{\hat{n}},1}}, \ldots, \mu_{\alpha_{{\hat{n}},\beta_{\hat{n}}}} \Big) \bigg) \nonumber \\
&& \cdot \nabla \xi(x) \bigg] \quad \, \bigg(  m^{({\hat{\beta}}+1)}  \Big(t, \mu, \mu_{{\hat{\alpha}}_1}, \ldots, \mu_{{\hat{\alpha}}_{\hat{\beta}}},\mu_{k+1} \Big) \bigg) (dx)  \Bigg]. \nonumber
\end{eqnarray}
Note that the term ${\inp[\Big]{\xi}{ c_1^{(k+1)} (t,\mu,\mu_1, \ldots, \mu_k, \mu_{k+1} ) }}_{n+k+1,\infty} $ can be rewritten as 
\begin{eqnarray}
&& \, {\inp[\Big]{\xi}{ c_1^{(k+1)} (t,\mu,\mu_1, \ldots, \mu_k, \mu_{k+1} ) }}_{n+k+1,\infty}  \nonumber \\
&=&   \int_{\bT^d} \Big[ \Big( b \big( y, m(t, \mu_{k+1}) \big) - b \big( y, m(t, \mu) \big) \Big) \cdot \nabla \xi(y) \Big]  \,\big( m^{(k)}(t,\mu_{k+1},\mu_1, \ldots, \mu_k ) -  m^{(k)}(t,\mu,\mu_1, \ldots, \mu_k ) \big)(dy)   \nonumber \\
&& + \int_{\bT^d} \Big[ \Big( b \big( y, m(t, \mu_{k+1}) \big) - b \big( y, m(t, \mu) \big) \Big) \cdot \nabla \xi(y) \Big]  \, m^{(k)}(t,\mu,\mu_1, \ldots, \mu_k ) (dy)  \nonumber \\
&& -\int_{\bT^d} \intrd \bigg[\ld[b] \big( y, m(t, \mu) \big)(z)  \cdot \nabla \xi(y) \bigg] \, m^{(1)}(t,\mu,\mu_{k+1})(dz) \, \,m^{(k)}(t,\mu,\mu_1,\ldots, \mu_{k}) (dy). \label{c1 k+1} 
\end{eqnarray}
By Theorem \ref{ thm mk lipschitz}, the first term of \eqref{c1 k+1} is controlled by
\begin{eqnarray} &&  \bigg| \int_{\bT^d} \Big[ \Big( b \big( y, m(t, \mu_{k+1}) \big) - b \big( y, m(t, \mu) \big) \Big) \cdot \nabla \xi(y) \Big] \nonumber \\
&& \quad \quad \,\big( m^{(k)}(t,\mu_{k+1},\mu_1, \ldots, \mu_k ) -  m^{(k)}(t,\mu,\mu_1, \ldots, \mu_k ) \big)(dy) \bigg| \nonumber \\
& \leq & C \| \xi \|_{n+k+1,\infty} W_1(\mu, \mu_{k+1})^2. \nonumber 
\end{eqnarray} 
By the same argument as \eqref{eq: linear derivative higher order trick}  and \eqref{ m beta i bound}, the second and third terms of \eqref{c1 k+1} are controlled by 
\begin{eqnarray}
&&  \bigg| \int_{\bT^d} \Big[ \Big( b \big( y, m(t, \mu_{k+1}) \big) - b \big( y, m(t, \mu) \big) \Big) \cdot \nabla \xi(y) \Big]  \, m^{(k)}(t,\mu,\mu_1, \ldots, \mu_k ) (dy)  \nonumber \\
&& -\int_{\bT^d} \intrd \bigg[\ld[b] \big( y, m(t, \mu) \big)(z)  \cdot \nabla \xi(y) \bigg] \, m^{(1)}(t,\mu,\mu_{k+1})(dz) \, \,m^{(k)}(t,\mu,\mu_1,\ldots, \mu_{k}) (dy) \bigg| \nonumber \\
& \leq & \bigg| \int_0^1  \int_{\bT^d} \bigg[ \ld[b] (y, rm(t, \mu_{k+1}) + (1-r) m(t, \mu))(m(t, \mu_{k+1}) - m(t, \mu))  \nonumber \\
&& - \ld[b] (y, m(t, \mu))(m(t, \mu_{k+1}) - m(t, \mu)) \bigg] \cdot \nabla \xi(y) \,  m^{(k)}(t,\mu,\mu_1, \ldots, \mu_k ) (dy) \,dr \bigg| \nonumber \\
&& + \bigg| \int_{\bT^d}  \bigg[\ld[b] \big( y, m(t, \mu) \big)(m(t, \mu_{k+1}) - m(t, \mu)  - m^{(1)}(t,\mu,\mu_{k+1}))  \cdot \nabla \xi(y) \bigg] \,  m^{(k)}(t,\mu,\mu_1, \ldots, \mu_k ) (dy) \bigg| \nonumber \\ 
& \leq & C \| \xi \|_{n+k+1,\infty} W_1(\mu, \mu_{k+1})^2, \nonumber 
\end{eqnarray} 
where the estimate for the first term follows from \hlipb{n+k}{k+1} with the same argument as \eqref{ct estimate second part}.
This shows that
\[ \big\| c_1^{(k+1)} (t,\mu,\mu_1, \ldots, \mu_k, \mu_{k+1} ) \big\|_{-(n+k+1,\infty)}\leq C  W_1(\mu, \mu_{k+1})^2. \nonumber 
\]
Similarly, by  \hregb{n+k}{k+1},   \hlipb{n+k}{k+1} and Theorem \ref{ thm mk lipschitz}, along with a similar argument applied to the induction hypothesis  \eqref{eq: linearisation formula m ell induction hypothesis} (as in estimates \eqref{ m beta i bound}, \eqref{ind hypo diff 1} and \eqref{ind hypo diff 2}),  we can show that, for $i \in \{2,3,4\}$, \[  \big\| c_i^{(k+1)} (t,\mu,\mu_1, \ldots, \mu_k, \mu_{k+1} ) \big\|_{-(n+k+1,\infty)}
 \leq  C  W_1(\mu, \mu_{k+1})^2 .
\]
Therefore,
\[ \big\| c^{(k+1)} (t,\mu,\mu_1, \ldots, \mu_k, \mu_{k+1} ) \big\|_{-(n+k+1,\infty)} 
 \leq  C  W_1(\mu, \mu_{k+1})^2. 
\] 
Finally, by \regb[n+k+1], \eqref{rho k+1  problem} and Theorem \ref{lions paper pde result}, we conclude that
\begin{eqnarray} && \big\| {\rho}^{(k+1)} (t,\mu,\mu_1, \ldots, \mu_k, \mu_{k+1} ) \big\|_{-(n+k+1,\infty)}\nonumber \\
& \leq &  C \big\| c^{(k+1)} (t,\mu,\mu_1, \ldots, \mu_k, \mu_{k+1} ) \big\|_{-(n+k+1,\infty)} \nonumber \\
& \leq & C  W_1(\mu, \mu_{k+1})^2. \nonumber 
\end{eqnarray}
\end{proof}
\subsection{Analysis of higher order backward Kolmogorov equations}
In this subsection, we fix $t \in [0,T]$ and consider the following Cauchy problem (defined recursively by \eqref{ G lambda}, Definition \ref{ def Tk}  and Definition \ref{ def  lambda k}):
 \begin{equation}
\left\{
\begin{array}{rrl}
      {} & \partial_s v^{(k)} (s,x, \mu, \mu_1, \ldots, \mu_k) +\Delta v^{(k)} (s,x, \mu, \mu_1, \ldots, \mu_k) \\
      & + \,  b(x, m(s, \mu)) \cdot \nabla v^{(k)} (s,x, \mu, \mu_1, \ldots, \mu_k)   &
     \\
      {} & + \, \ld[b](x, m(s,\mu))(m^{(k)}(s,\mu,\mu_1, \ldots, \mu_k)) \cdot \nabla v(s,x, \mu) + G_{\lambda_k} (s,x,\mu,\mu_1, \ldots, \mu_k) & =0, \\
      & &  \\
      {} &  v^{(k)} (t,x, \mu, \mu_1, \ldots, \mu_k) &= 0, \\
\end{array} 
\right. \label{vk linearised} \end{equation} 
where
\begin{eqnarray}
G_{\lambda_k} (s,x,\mu,\mu_1, \ldots, \mu_k) &:= &  \sum_{\lambda \in e(\lambda_k)} \bigg[  \frac{ \delta^{{{\hat{n}}}} b}{ \delta m^{{{\hat{n}}}}}(x, m(s, \mu)) \bigg( m^{(\beta_1)} \Big(s, \mu, \mu_{\alpha_{1,1}}, \ldots, \mu_{\alpha_{1,\beta_1}} \Big), \ldots, \nonumber \\
& & m^{(\beta_{{\hat{n}}})} \Big(s, \mu, \mu_{\alpha_{{\hat{n}},1}}, \ldots, \mu_{\alpha_{{\hat{n}},\beta_{\hat{n}}}} \Big) \bigg) \cdot \nabla  v^{({\hat{\beta}})}  \Big(s, x, \mu, \mu_{{\hat{\alpha}}_1}, \ldots, \mu_{{\hat{\alpha}}_{\hat{\beta}}} \Big)   \bigg]. \label{ G lambda}
\end{eqnarray}
The following theorem gives the regularity of $v^{(k)}$ by Schauder estimates. 
\begin{theorem} \label{ thm vk reg} 
Let $k \in \bN$. Assume  \emph{\hregb{ n+k-1}{k}}, where $n \geq 2$. Suppose that $\xi \in W^{n+1,\infty}$. Then the Cauchy problem $v^{(k)}$ defined by \eqref{vk linearised} has a unique solution in  $L^{\infty} \big([0,t],  W^{n+1, \infty}(\mathbb{T}^d)  \big)$. 
\end{theorem} 
\begin{proof}
We proceed by strong induction. The base step is proven in Lemma \ref{representation of v one}. For the induction step, we assume that the statement is true for $1, \ldots, k-1$, where $k \geq 2$. For each $\lambda \in e(\lambda_k)$, by \hregb{ n+k-1}{k},
\begin{eqnarray}
&& \sup_{s \in [0,t]}  \bigg\| \frac{ \delta^{{{\hat{n}}}} b}{ \delta m^{{{\hat{n}}}}}(\cdot, m(s, \mu)) \bigg( m^{(\beta_1)} \Big(s, \mu, \mu_{\alpha_{1,1}}, \ldots, \mu_{\alpha_{1,\beta_1}} \Big), \ldots, \nonumber \\
& & m^{(\beta_{{\hat{n}}})} \Big(s, \mu, \mu_{\alpha_{{\hat{n}},1}}, \ldots, \mu_{\alpha_{{\hat{n}},\beta_{\hat{n}}}} \Big) \bigg) \cdot \nabla  v^{({\hat{\beta}})}  \Big(s, \cdot, \mu, \mu_{{\hat{\alpha}}_1}, \ldots, \mu_{{\hat{\alpha}}_{\hat{\beta}}} \Big) \bigg\|_{n,\infty} \nonumber \\
& \leq & C \sup_{s \in [0,t]} \bigg\| \frac{ \delta^{{{\hat{n}}}} b}{ \delta m^{{{\hat{n}}}}}(\cdot, m(s, \mu)) \bigg( m^{(\beta_1)} \Big(s, \mu, \mu_{\alpha_{1,1}}, \ldots, \mu_{\alpha_{1,\beta_1}} \Big), \ldots, m^{(\beta_{{\hat{n}}})} \Big(s, \mu, \mu_{\alpha_{{\hat{n}},1}}, \ldots, \mu_{\alpha_{{\hat{n}},\beta_{\hat{n}}}} \Big) \bigg) \bigg\|_{\na} \nonumber \\
&& \times \, \sup_{s \in [0,t]} \Big\| v^{({\hat{\beta}})}  \Big(s, \cdot, \mu, \mu_{{\hat{\alpha}}_1}, \ldots, \mu_{{\hat{\alpha}}_{\hat{\beta}}} \Big) \Big\|_{n + 1,\infty} \nonumber \\
& \leq &  C \sup_{s \in [0,t]} \bigg\| \frac{ \delta^{{{\hat{n}}}} b}{ \delta m^{{{\hat{n}}}}}(\cdot, m(s, \mu)) \bigg( m^{(\beta_1)} \Big(s, \mu, \mu_{\alpha_{1,1}}, \ldots, \mu_{\alpha_{1,\beta_1}} \Big), \ldots, m^{(\beta_{{\hat{n}}})} \Big(s, \mu, \mu_{\alpha_{{\hat{n}},1}}, \ldots, \mu_{\alpha_{{\hat{n}},\beta_{\hat{n}}}} \Big) \bigg) \bigg\|_{n+k-1,\infty} \nonumber \\
&& \times \, \sup_{s \in [0,t]} \Big\| v^{({\hat{\beta}})}  \Big(s, \cdot, \mu, \mu_{{\hat{\alpha}}_1}, \ldots, \mu_{{\hat{\alpha}}_{\hat{\beta}}} \Big) \Big\|_{n +1 ,\infty} < +\infty, \nonumber
\end{eqnarray}
which implies that $G_{\lambda_k}(\cdot,\cdot, \mu, \mu_1, \ldots, \mu_k) \in L^{\infty}( [0,t], W^{n,\infty}(\bT^d))$. This completes the induction step by repeating the same argument as in the proof of Theorem \ref{lions paper pde result second}. 
\end{proof}
The following theorem is an analogue of Theorem \ref{ mk mk+1 pertub} for backward Kolmogorov equations. The computations in the proof follow the same ideas as those in the previous subsection, i.e. Theorem \ref{ thm mk lipschitz} and Theorem \ref{ mk mk+1 pertub}. Consequently, the proof is omitted for brevity. 
\begin{theorem} \label{ thm vk vk+1} 
Let  $k \in \bN $. Assume  \emph{\hregb{n+k+1}{k+1}} and \emph{\hlipb{n+k}{k+1}}, where $n \geq 2$. Suppose that $\xi \in W^{n+1,\infty}$.  Then 
\begin{eqnarray}  
&& \sup_{s \in [0,t]} \Big\| v^{(k)}(s, \cdot ,\mu_{k+1}, \mu_1, \ldots, \mu_k) - v^{(k)}(s, \cdot ,\mu, \mu_1, \ldots, \mu_k) - v^{(k+1)}(s,\cdot , \mu, \mu_1, \ldots, \mu_k, \mu_{k+1})\Big\|_{n+1,\infty} \nonumber \\
& \leq & C W_1(\mu, \mu_{k+1})^2,  \label{eq: linearisation formula vk}  \end{eqnarray} for any $\mu, \mu_1, \ldots, \mu_{k+1} \in \cP(\bT^d)$, for some constant $C>0$. 
\end{theorem}
We now establish the $k$th order linear derivative of $v$ in terms of $v^{(k)}$.
\begin{theorem} \label{ kth order derivative of v}
Let  $k \in \bN $. Assume  \emph{\hregb{n+k}{k}} and \emph{\hlipb{n+k-1}{k}}, where $n \geq 2$. Suppose that $\xi \in W^{n+1,\infty}$.  Then 
$$ v^{(k)} (0,x, \mu, \mu_1, \ldots, \mu_{k-1}, \delta_{y_k}) = \ld[v^{(k-1)}] (0,x, \mu, \mu_1, \ldots, \mu_{k-1}, y_k), $$ 
where the linear derivative $\ld[v^{(k-1)}]$ is taken with respect to $\mu$. Consequently, $\frac{ \delta^k v}{\delta m^k} (0,x, \mu, y_1, \ldots,  y_k)$ exists and is given by
\begin{equation*}  \frac{ \delta^k v}{\delta m^k} (0,x, \mu, y_1, \ldots,  y_k)=  v^{(k)} (0,x, \mu, \delta_{y_1}, \ldots, \delta_{y_k}).  \end{equation*}
\end{theorem}
\begin{proof}
Replacing $k$ by $k-1$ in Theorem \ref{ thm vk vk+1} gives
\begin{eqnarray}  
&& \sup_{s \in [0,t]} \Big\| v^{(k-1)}(s, \cdot ,\mu_{k}, \mu_1, \ldots, \mu_{k-1}) - v^{({k-1})}(s, \cdot ,\mu, \mu_1, \ldots, \mu_{k-1}) - v^{(k)}(s,\cdot , \mu, \mu_1, \ldots, \mu_{k-1}, \mu_{k})\Big\|_{n+1,\infty} \nonumber \\
& \leq & C W_1(\mu, \mu_{k})^2.  \nonumber  \end{eqnarray}
It follows from a similar argument as Lemma \ref{representation of v one} to show that
$$ v^{(k)} (0, x, \mu, \mu_1, \ldots, \mu_{k-1}, \mu_k) = \intrd v^{(k)} (0, x, \mu, \mu_1, \ldots, \mu_{k-1}, \delta_z) \, (\mu_k - \mu)(dz). $$
This proves the first equality. For the second equality, an inductive argument gives
\begin{eqnarray}
  v^{(k)} (0,x, \mu, \delta_{y_1}, \ldots, \delta_{y_k}) & = &    \ld[v^{(k-1)}] (0,x, \mu, \delta_{y_1}, \ldots, \delta_{y_{k-1}}, y_k) \nonumber \\
  &  = &  \frac{\delta^2 v^{(k-2)}}{\delta m^2} (0,x, \mu, \delta_{y_1}, \ldots, \delta_{y_{k-2}}, {y_{k-1}}, y_k) \nonumber \\
  & \vdots & \nonumber \\
  & = & \frac{ \delta^k v}{\delta m^k} (0,x, \mu, y_1, \ldots,  y_k). \nonumber 
\end{eqnarray}
\end{proof}
\subsection{Connection between higher order forward and backward equations} 
In this section, we follow the same approach as Section \ref{ backward first order } to show that integrals with respect to the signed measure $m^{(k)}(t, \mu, \mu_1, \ldots, \mu_k )$ can be re-expressed in terms of the signed measure $\mu_k- \mu$. 
\begin{theorem} \label{thm Ij} 
Let  $k \in \bN $. Assume  \emph{\hregb{n+k}{k}} and \emph{\hlipb{n+k-1}{k}}, where $n \geq 2$. Suppose that $\xi \in W^{ n+k,\infty}$. We define a sequence of functions $I^{(j)}(x, \mu, \mu_1, \ldots, \mu_{j-1}; \xi, t)$, $j \in \{1, \ldots, k\}$, by the following iteration:
\begin{equation}  I^{(1)}(x, \mu; \xi, t) := v(0,x, \mu; \xi ,t) + \int_{\bT^d}  \frac{\delta v}{\delta m}(0,z, \mu,x; \xi ,t)   \, \mu(dz),  \label{I1 def}  \end{equation} 
 \begin{eqnarray} I^{(j)}(x, \mu, \mu_1, \ldots, \mu_{j-1}; \xi, t) & := & -I^{(j-1)}(x, \mu,\mu_1, \ldots, \mu_{j-2}; \xi ,t) \nonumber \\
 && + \int_{\bT^d}  \frac{\delta I^{(j-1)}}{\delta m}(z, \mu,\mu_1, \ldots, \mu_{j-2}, x; \xi ,t)   \, (\mu_{j-1}- \mu) (dz),  \label{Ij iteration}  
 \end{eqnarray}
for $j \in \{2, \ldots, k \}$, where $\frac{\delta I^{(j-1)}}{\delta m}$ is taken with respect to $\mu$. Then the sequence is well-defined and $$ \intrd \xi(x) \, m^{(k)}(t, \mu, \mu_1, \ldots, \mu_k )(dx) = \intrd I^{(k)} (x, \mu,\mu_1, \ldots, \mu_{k-1}; \xi,t) \, (\mu_k - \mu)(dx). $$ 
\end{theorem}
\begin{proof}
By Theorem \ref{ kth order derivative of v}, the sequence $I^{(j)}$ is well-defined. To prove the equality, we proceed via an induction argument. The base step is established in \eqref{v phi connection 2}. For the inductive step, we assume that
$$ \intrd \xi(x) \, m^{(k-1)}(t, \mu, \mu_1, \ldots, \mu_{k-1} )(dx) = \intrd I^{(k-1)} (x, \mu, \mu_1, \ldots, \mu_{k-2}; \xi,t) \, (\mu_{k-1} - \mu)(dx). $$
By replacing $k$ by $k-1$ in Theorem \ref{ mk mk+1 pertub},  we have
\begin{eqnarray}  
&& \sup_{t \in [0,T]} \Big\| m^{(k-1)}(t, \mu_{k}, \mu_1, \ldots, \mu_{k-1}) - m^{(k-1)}(t, \mu, \mu_1, \ldots, \mu_{k-1}) - m^{(k)}(t, \mu, \mu_1, \ldots, \mu_{k-1}, \mu_{k})\Big\|_{-(n+k,\infty)} \nonumber \\
& \leq & C W_1(\mu, \mu_{k})^2 ,  \nonumber  \end{eqnarray} for any $\mu, \mu_1, \ldots, \mu_{k} \in \cP(\bT^d)$, for some constant $C>0$. Since $\xi \in W^{ n+k,\infty}$, it follows from the proof of Theorem \ref{thm m m-one} to observe that
\begin{equation}  \frac{d}{d \eps} \bigg|_{\eps =0^{+}} \intrd \xi(x) \, m^{(k-1)}(t, (1-\eps)\mu + \eps \mu_k, \mu_1, \ldots, \mu_{k-1} )(dx)  =  \intrd \xi(x) \, m^{(k)}(t, \mu, \mu_1, \ldots, \mu_{k-1} , \mu_k)(dx). \label{I k-1 first part} \end{equation} 
On the other hand, by the chain rule of differentiation,
\begin{eqnarray} 
&& \frac{d}{d \eps} \bigg|_{\eps =0^{+}} \intrd I^{(k-1)} (x, (1-\eps)\mu + \eps \mu_k, \mu_1, \ldots, \mu_{k-2}; \xi,t) \, \big(\mu_{k-1} - ( (1-\eps)\mu + \eps \mu_k) \big)(dx) \nonumber \\
& = & \intrd \intrd  \ld[I^{(k-1)}] (x, \mu,\mu_1, \ldots, \mu_{k-2},z; \xi,t) \, (\mu_k- \mu)(dz) \, (\mu_{k-1}- \mu)(dx) \nonumber \\
&& - \intrd I^{(k-1)} (x, \mu,\mu_1, \ldots, \mu_{k-2}; \xi,t) \, (\mu_{k} - \mu )(dx) \nonumber \\
& = & \intrd \bigg[ \intrd  \ld[I^{(k-1)}] (z, \mu,\mu_1, \ldots, \mu_{k-2},x; \xi,t) \, (\mu_{k-1}- \mu)(dz)  \nonumber \\
&& - I^{(k-1)} (x, \mu, \mu_1, \ldots, \mu_{k-2}; \xi,t) \bigg] \, (\mu_{k} - \mu )(dx) \nonumber \\
&=&  \intrd  I^{(k)} (x, \mu,\mu_1, \ldots, \mu_{k-1}; \xi,t) \, (\mu_{k} - \mu )(dx). \label{I k-1 second part} 
\end{eqnarray}
The proof is complete by combining \eqref{I k-1 first part} and \eqref{I k-1 second part}.
\end{proof}
\section{Regularity of higher order derivatives in measure of \texorpdfstring{$\cU$}{U}} 

\subsection{Definitions and notations for iteration in multi-indices in the class \texorpdfstring{$\Delta_k$}{Delta-k}}
In order to obtain a general formula for the $k$th order linear derivative of $\Phi$, we proceed with another iteration argument. Therefore, we need to introduce another class $\Delta_k$ of multi-indices.
\begin{definition}[Class $\Delta_k$ of multi-indices] \label{multi indices delta k} 
For any $k \in \bN$, the class $\Delta_k$ contains all multi-indices of the form
\begin{equation} {\Lambda}:=  \bigg({\hat{n}}, ( \beta_j )_{j=1}^{{\hat{n}}}, ( {\alpha_{i,j}} )_{\substack{1 \leq i \leq {\hat{n}} \\ 1 \leq j \leq \beta_{i}}} \bigg), \label{Lambda def} \end{equation}
where ${\hat{n}}$ and $\beta_j$  are non-negative integers and $ \alpha_{i,j}$,   $1 \leq i \leq {\hat{n}}$,  $ 1 \leq j \leq \beta_{i}$,  are positive integers satisfying
\begin{enumerate}[(i)]
\item $$ {\hat{n}} \leq k, \quad \quad 1 \leq \alpha_{i,1}< \ldots<  \alpha_{i,\beta_{i}} \leq k, \quad \quad \beta_1, \ldots, \beta_{\hat{n}} \leq k,$$
\item \begin{equation} \sum_{i=1}^{\hat{n}} \beta_i  = k, \label{eq: sum condition 2}  \end{equation}
\item for any $i,i' \in \{1, \ldots, \hat{n} \},$ \begin{equation} \Big\{ \alpha_{i,1}, \ldots, \alpha_{i, \beta_i} \Big\} \cap  \Big\{ \alpha_{i',1}, \ldots, \alpha_{i', \beta_{i'}} \Big\}  = \emptyset. \label{eq: emptyset condition 2}  \end{equation}
\end{enumerate}
In particular, $o(\Lambda)$ is called the \emph{order} of $\Lambda$ defined by
$$ o(\Lambda):={\hat{n}}.$$ 
Moreover, for any $(\Lambda^{(1)}, \ldots, \Lambda^{(q)}) \in (\Delta_k)^q$, we define the \emph{magnitude} of $(\Lambda^{(1)}, \ldots, \Lambda^{(q)})$ by
$$ m\big( (\Lambda^{(1)}, \ldots, \Lambda^{(q)}) \big) := q.$$ If $\Lambda= \Lambda^{(i)}$, for some $i \in \{1, \ldots, q \}$, we write
$$ \Lambda \in e \big( (\Lambda^{(1)}, \ldots, \Lambda^{(q)}) \big):= \{ \Lambda^{(1)}, \ldots, \Lambda^{(q)} \}.$$ 
\end{definition}
Next, we introduce the recurrence map $Q_k$ for multi-indices in $\Delta_k$, followed by the sequence of multi-dimensional vectors $\Lambda_k $ of elements in $\Delta_k$.
\begin{definition}[Recurrence map $Q_k$] \label{ def Qk} 
Let $\Lambda \in \Delta_k$ be given by the form \eqref{Lambda def}. We define a recurrence map $Q_k$ by 
\begin{eqnarray}
(\Delta_{k+1})^{o(\Lambda)+1} \ni Q_k(\Lambda) &:=& \bigg( \Big( {\hat{n}}+1, ( \beta_1, \ldots, \beta_{\hat{n}}, 1 ), ({\alpha}_{1,1}, \ldots, {\alpha}_{{\hat{n}}, \beta_{\hat{n}}},k+1)\Big), \nonumber \\
&& \Big( {\hat{n}}, ( \beta_1, \ldots, \beta_{p-1}, \beta_p +1, \beta_{p+1}, \ldots, \beta_{\hat{n}}), \nonumber \\
&& ( \alpha_{1,1}, \ldots, \alpha_{p-1, \beta_{p-1}}, \alpha_{p,1}, \ldots, \alpha_{p, \beta_p}, k+1, \alpha_{p+1,1}, \ldots, \alpha_{{\hat{n}}, \beta_{\hat{n}}} ) \Big)_{1 \leq p \leq {\hat{n}}} \bigg). \nonumber 
\end{eqnarray}
\end{definition}
\begin{definition}[Multi-dimensional vectors $\Lambda_k $ of elements in $\Delta_k$] \label{ def  Lambda k} 
We first define
\begin{eqnarray}
\Lambda_{1} & : = & \Big( 1, (1), (1) \Big) \in \Delta_1. \nonumber
\end{eqnarray}
For every $k \geq 2$,  we define a multi-dimensional vector $\Lambda_{k+1}$ of elements in $\Delta_{k+1}$ by the recurrence relation
\begin{eqnarray}
\Lambda_{k+1} & : = & \Big( Q_k(\Lambda^{(1)}_k), \ldots, Q_k(\Lambda^{(m(\Lambda_k))}_k ) \Big), \nonumber
\end{eqnarray}
for $\Lambda_{k} = \big(\Lambda^{(1)}_k, \ldots,  \Lambda^{(m(\Lambda_k))}_k \big)$.
\end{definition}
\subsection{Analysis of higher order linear derivatives of \texorpdfstring{ $\cU$}{U}}
We begin by establishing a higher-order analogue of Theorem \ref{thm m m-one}.
\begin{lemma} \label{ lemma differentiation higher order } Let  $k \in \bN \setminus \{1 \}$. Assume  \emph{\hregb{n+k}{k}}, \emph{\hlipb{n+k-1}{k}} and \\ \emph{\hregphi{n+k}{k-1}}, where $n \geq 2$. Then, for $\hat{n}, \beta \leq k-1$ and $i \in \{1, \ldots, \hat{n} \}$, 
\begin{eqnarray}
&& \frac{d}{d \eps} \bigg|_{\eps=0^{+}} \intrd \frac{\delta^{\hat{n}} \Phi}{\delta m^{\hat{n}}} (m)(y_1, \ldots, y_{\hat{n}})  \big( m^{(\beta)} (t, (1- \eps) \mu + \eps {\mu}_{k}, \mu_1, \ldots, \mu_{\beta}) \big)(dy_i) \nonumber \\
& = & \int_{\bT^d}  \frac{\delta^{\hat{n}} \Phi}{\delta m^{\hat{n}}}(m) (y_1, \ldots, y_{\hat{n}})  \, m^{(\beta+1)}(t,\mu,\m_1, \ldots, \mu_{\beta}, \mu_k) (dy_i),
\end{eqnarray} 
for every $m,\mu, \mu_1, \ldots, \mu_{\beta}, \mu_k \in \cP(\bT^d).$
\end{lemma}
\begin{proof}
Since $\beta \leq k-1$, the condition \hintb{n+k}{k} implies \hintb{n+\beta+1}{\beta+1}. Similarly, the condition \hlipb{n+k-1}{k} implies \hlipb{n+\beta}{\beta+1}. By Theorem \ref{ mk mk+1 pertub},  we have
\begin{eqnarray}  
&& \sup_{t \in [0,T]} \Big\| m^{(\beta)}(t, \mu_{k}, \mu_1, \ldots, \mu_{\beta}) - m^{(\beta)}(t, \mu, \mu_1, \ldots, \mu_{\beta}) - m^{(\beta+1)}(t, \mu, \mu_1, \ldots, \mu_{\beta}, \mu_{k})\Big\|_{-(n+\beta+1,\infty)} \nonumber \\
& \leq & C W_1(\mu, \mu_{k})^2,  \nonumber  \end{eqnarray} for any $\mu, \mu_1, \ldots, \mu_{\beta}, \mu_{k} \in \cP(\bT^d)$, for some constant $C>0$. On the other hand, the condition  \hregphi{n+k}{k-1} implies \hregphi{n+\beta+1}{k-1}. The rest of the proof is identical to the proof of Theorem \ref{thm m m-one}. 
\end{proof}
We are now in a position to prove the main result of the paper. Clearly, one can obtain the minimal condition by setting $n=2$ (as in the introduction).
\begin{theorem} \label{main result} 
Let  $k \in \bN$ and $n \geq 2$. Assume \emph{\hregb{n+k}{k}}, \emph{\hlipb{n+k-1}{k}}, \emph{\hlipphi[k]} and \emph{\hregphi{n+k}{k}}. 
Then $\frac{\delta^k \cU}{\delta m^k}$ exists and is given by
\begin{eqnarray}
     &&    \frac{\delta^k \cU}{\delta m^k} ( t,\mu)(z_1, \ldots , z_k)  \nonumber \\
     & = &  \sum_{\Lambda = \big( \hat{n}, (\beta_j), (\alpha_{i,j}) \big) \in e(\Lambda_k)} \bigg[ \frac{ \delta^{{{\hat{n}}}} \Phi}{ \delta m^{{{\hat{n}}}}}( m(t, \mu)) \bigg( m^{(\beta_1)} \Big(t, \mu, \delta_{z_{\alpha_{1,1}}}, \ldots, \delta_{z_{\alpha_{1,\beta_1}}} \Big), \ldots, m^{(\beta_{{\hat{n}}})} \Big(t, \mu, \delta_{z_{\alpha_{{\hat{n}},1}}}, \ldots, \delta_{z_{\alpha_{{\hat{n}},\beta_{\hat{n}}}}} \Big) \bigg) \bigg].   \nonumber 
       \end{eqnarray}
        In particular, if we also assume \emph{\hintphi{n+k-1}{k}}, then
$$\sup_{z_1, \ldots, z_k \in \bT^d}  \sup_{\mu \in \cP(\bT^d)} \sup_{t \in [0,T]} \bigg|  \frac{\delta^k \cU}{\delta m^k} ( t,\mu)(z_1, \ldots , z_k) \bigg| <+ \infty. $$
\end{theorem}
\begin{proof}
We proceed by induction on $k$. We first prove the statement for $k=1$. By Corollary \ref{Corr existence first order}, we know that $\ld[\cU]$ exists. Therefore, by \eqref{heuristics verified}, 
\begin{equation*} \intrd \ld[\cU](t, \mu)(y) \, (\hmu- \mu)(dy) =  \frac{d}{d \eps} \bigg|_{\eps=0^{+}} \Phi \big( m(t, (1- \eps) \mu + \eps \hat{\mu}) \big) = \int_{\bT^d} \frac{\delta \Phi}{\delta m}(m(t, \mu))(y) \, m^{(1)}(t,\mu,\hat{\mu}) (dy).
         \end{equation*}
         By the normalisation convention of $\ld[\cU]$, 
        \begin{equation} \intrd \ld[\cU](t, \mu)(y) \, \hmu(dy) = \int_{\bT^d} \frac{\delta \Phi}{\delta m}(m(t, \mu))(y) \, m^{(1)}(t,\mu,\hat{\mu}) (dy).   \label{normalisation argument part i} \end{equation} 
         Therefore, putting $\hmu:= \delta_{z_1}$ gives
           \begin{equation} \ld[\cU](t, \mu)(z_1) = \int_{\bT^d} \frac{\delta \Phi}{\delta m}(m(t, \mu))(y) \, m^{(1)}(t,\mu,\delta_{z_1}) (dy).  \label{normalisation argument part ii} \end{equation}
We now assume that this statement holds for $k-1$.  Therefore, for any $\eps >0$, we have
       \begin{eqnarray}
     &&    \frac{\delta^{{k-1}} \cU}{\delta m^{k-1}} ( t,(1-\eps) \mu + \eps \mu_k)(z_1, \ldots , z_{k-1})  \nonumber \\
     & = &  \sum_{\Lambda \in e(\Lambda_{k-1})} \bigg[ \frac{ \delta^{{{\hat{n}}}} \Phi}{ \delta m^{{{\hat{n}}}}}( m(t, (1-\eps) \mu + \eps \mu_k)) \bigg( m^{(\beta_1)} \Big(t, (1-\eps) \mu + \eps \mu_k, \delta_{z_{\alpha_{1,1}}}, \ldots, \delta_{z_{\alpha_{1,\beta_1}}} \Big), \ldots, \nonumber \\
     && m^{(\beta_{{\hat{n}}})} \Big(t, (1-\eps) \mu + \eps \mu_k,  \delta_{z_{\alpha_{{\hat{n}},1}}}, \ldots, 
     \delta_{z_{\alpha_{{\hat{n}},\beta_{\hat{n}}}}} \Big) \bigg) \bigg].   \nonumber 
       \end{eqnarray}
       By the chain rule of differentiation,
       \begin{eqnarray}
      &&  \frac{d}{d  \eps} \bigg[  \frac{\delta^{{k-1}} \cU}{\delta m^{k-1}} ( t,(1-\eps) \mu + \eps \mu_k)(z_1, \ldots , z_{k-1}) \bigg] \nonumber \\
      & = &  \sum_{\Lambda \in e( \Lambda_{k-1})} \bigg( \frac{d}{d  \eps} \bigg[ \frac{ \delta^{{{\hat{n}}}} \Phi}{ \delta m^{{{\hat{n}}}}}( m(t, (1-\eps) \mu + \eps \mu_k)) \bigg( m^{(\beta_1)} \Big(t, (1-\hat{\eps}) \mu + \hat{\eps} \mu_k, \delta_{z_{\alpha_{1,1}}}, \ldots, \delta_{z_{\alpha_{1,\beta_1}}} \Big), \ldots, \nonumber \\
     && m^{(\beta_{{\hat{n}}})} \Big(t, (1-\hat{\eps}) \mu + \hat{\eps} \mu_k,  \delta_{z_{\alpha_{{\hat{n}},1}}}, \ldots, 
     \delta_{z_{\alpha_{{\hat{n}},\beta_{\hat{n}}}}} \Big) \bigg) \bigg] \bigg) \bigg|_{\hat{\eps}= \eps}    \nonumber \\
     && + \sum_{\Lambda \in e( \Lambda_{k-1})} \sum_{i=1}^{\hat{n}}  \bigg( \frac{d}{d  \eps} \bigg[ \frac{ \delta^{{{\hat{n}}}} \Phi}{ \delta m^{{{\hat{n}}}}}( m(t, (1-\hat{\eps}) \mu + \hat{\eps} \mu_k)) \bigg( m^{(\beta_1)} \Big(t, (1-\hat{\eps}) \mu + \hat{\eps} \mu_k, \delta_{z_{\alpha_{1,1}}}, \ldots, \delta_{z_{\alpha_{1,\beta_1}}} \Big), \ldots, \nonumber \\
     && m^{(\beta_{{i-1}})} \Big(t, (1-\hat{\eps}) \mu + \hat{\eps} \mu_k,  \delta_{z_{\alpha_{{{i-1}},1}}}, \ldots, 
     \delta_{z_{\alpha_{{{i-1}},\beta_{i-1}}}} \Big), m^{(\beta_{{i}})} \Big(t, (1-{\eps}) \mu + {\eps} \mu_k,  \delta_{z_{\alpha_{{{i}},1}}}, \ldots, 
     \delta_{z_{\alpha_{{{i}},\beta_{i}}}} \Big), \nonumber \\
     && m^{(\beta_{{i+1}})} \Big(t, (1-\hat{\eps}) \mu + \hat{\eps} \mu_k,  \delta_{z_{\alpha_{{{i+1}},1}}}, \ldots, 
     \delta_{z_{\alpha_{{{i+1}},\beta_{i+1}}}} \Big), \ldots, m^{(\beta_{{\hat{n}}})} \Big(t, (1-\hat{\eps}) \mu + \hat{\eps} \mu_k,  \delta_{z_{\alpha_{{\hat{n}},1}}}, \ldots, 
     \delta_{z_{\alpha_{{\hat{n}},\beta_{\hat{n}}}}} \Big) \bigg) \bigg] \bigg) \bigg|_{\hat{\eps}= \eps}.    \nonumber \\ 
       \end{eqnarray}
       Therefore, the right-hand derivative at $\eps =0$ exists and is given by
       \begin{eqnarray}
      &&  \frac{d}{d  \eps} \bigg|_{\eps=0^{+}} \bigg[  \frac{\delta^{{k-1}} \cU}{\delta m^{k-1}} ( t,(1-\eps) \mu + \eps \mu_k)(z_1, \ldots , z_{k-1}) \bigg] \nonumber \\
      & = &  \sum_{\Lambda \in e( \Lambda_{k-1})} \bigg( \frac{d}{d  \eps}\bigg|_{\eps=0^{+}}  \bigg[ \frac{ \delta^{{{\hat{n}}}} \Phi}{ \delta m^{{{\hat{n}}}}}( m(t, (1-\eps) \mu + \eps \mu_k)) \bigg( m^{(\beta_1)} \Big(t, (1-\hat{\eps}) \mu + \hat{\eps} \mu_k, \delta_{z_{\alpha_{1,1}}}, \ldots, \delta_{z_{\alpha_{1,\beta_1}}} \Big), \ldots, \nonumber \\
     && m^{(\beta_{{\hat{n}}})} \Big(t, (1-\hat{\eps}) \mu + \hat{\eps} \mu_k,  \delta_{z_{\alpha_{{\hat{n}},1}}}, \ldots, 
     \delta_{z_{\alpha_{{\hat{n}},\beta_{\hat{n}}}}} \Big) \bigg) \bigg] \bigg) \bigg|_{\hat{\eps}= 0}    \nonumber \\
     && + \sum_{\Lambda \in e( \Lambda_{k-1})} \sum_{i=1}^{\hat{n}}  \bigg( \frac{d}{d  \eps} \bigg|_{\eps=0^{+}} \bigg[ \frac{ \delta^{{{\hat{n}}}} \Phi}{ \delta m^{{{\hat{n}}}}}( m(t, (1-\hat{\eps}) \mu + \hat{\eps} \mu_k)) \bigg( m^{(\beta_1)} \Big(t, (1-\hat{\eps}) \mu + \hat{\eps} \mu_k, \delta_{z_{\alpha_{1,1}}}, \ldots, \delta_{z_{\alpha_{1,\beta_1}}} \Big), \ldots, \nonumber \\
     && m^{(\beta_{{i-1}})} \Big(t, (1-\hat{\eps}) \mu + \hat{\eps} \mu_k,  \delta_{z_{\alpha_{{{i-1}},1}}}, \ldots, 
     \delta_{z_{\alpha_{{{i-1}},\beta_{i-1}}}} \Big), m^{(\beta_{{i}})} \Big(t, (1-{\eps}) \mu + {\eps} \mu_k,  \delta_{z_{\alpha_{{{i}},1}}}, \ldots, 
     \delta_{z_{\alpha_{{{i}},\beta_{i}}}} \Big), \nonumber \\
     && m^{(\beta_{{i+1}})} \Big(t, (1-\hat{\eps}) \mu + \hat{\eps} \mu_k,  \delta_{z_{\alpha_{{{i+1}},1}}}, \ldots, 
     \delta_{z_{\alpha_{{{i+1}},\beta_{i+1}}}} \Big), \ldots, m^{(\beta_{{\hat{n}}})} \Big(t, (1-\hat{\eps}) \mu + \hat{\eps} \mu_k,  \delta_{z_{\alpha_{{\hat{n}},1}}}, \ldots, 
     \delta_{z_{\alpha_{{\hat{n}},\beta_{\hat{n}}}}} \Big) \bigg) \bigg] \bigg) \bigg|_{\hat{\eps}= 0}. \nonumber \\ &&    \label{right hand derivative long expression}  
       \end{eqnarray}
      By  the assumptions  \hregphi{n+k}{k} (which implies \hregphi{n}{k-1}) and \hlipphi[k], we can repeat the same argument as in Theorem \ref{thm m m-one}. For any signed measures $m_1, \ldots, m_{\hat{n}}$ and $\hat{n} \leq k-1$, 
      \begin{eqnarray}
        &&    \frac{d}{d  \eps} \bigg|_{\eps=0^{+}} \bigg[  \frac{\delta^{\hat{n}} \Phi}{\delta m^{\hat{n}}} (m(t,(1-{\eps}) \mu + {\eps} \mu_k))(m_1, \ldots, m_{\hat{n}}) \bigg]  \nonumber \\
        & = & \frac{d}{d  \eps} \bigg|_{\eps=0^{+}} \bigg[ \intrd \ldots \intrd \frac{\delta^{\hat{n}} \Phi}{\delta m^{\hat{n}}} (m(t,(1-{\eps}) \mu + {\eps} \mu_k)) (y_1, \ldots, y_{\hat{n}}) \, m_1(dy_1) \ldots \, m_{\hat{n}} (dy_{\hat{n}}) \bigg] \nonumber \\
        & = & \intrd \ldots \intrd  \intrd  \frac{\delta^{\hat{n}+1} \Phi}{\delta m^{\hat{n}+1}} (m(t,\mu)) (y_1, \ldots, y_{\hat{n}},y_{\hat{n}+1}) (m^{(1)}(t, \mu, \mu_k))(dy_{\hat{n}+1}) \, m_1(dy_1) \ldots \, m_{\hat{n}} (dy_{\hat{n}}) \nonumber \\
        & = & \frac{\delta^{\hat{n}+1} \Phi}{\delta m^{\hat{n}+1}} (m(t,\mu)) \Big( m_1, \ldots, m_{\hat{n}} , m^{(1)}(t, \mu, \mu_k) \Big). \label{first part n hat deriv general }
      \end{eqnarray}
      The second part of \eqref{right hand derivative long expression}  can be computed by Lemma \ref{ lemma differentiation higher order }. Therefore, by \eqref{right hand derivative long expression},   \eqref{first part n hat deriv general } and Lemma \ref{ lemma differentiation higher order },  
       \begin{eqnarray}
      &&  \frac{d}{d  \eps} \bigg|_{\eps=0^{+}} \bigg[  \frac{\delta^{{k-1}} \cU}{\delta m^{k-1}} ( t,(1-\eps) \mu + \eps \mu_k)(z_1, \ldots , z_{k-1}) \bigg] \nonumber \\
      & = &  \sum_{\Lambda \in e( \Lambda_{k-1})}    \bigg[ \frac{ \delta^{{{\hat{n}+1}}} \Phi}{ \delta m^{{{\hat{n}+1}}}}( m(t, \mu)) \bigg( m^{(\beta_1)} \Big(t, \mu, \delta_{z_{\alpha_{1,1}}}, \ldots, \delta_{z_{\alpha_{1,\beta_1}}} \Big), \ldots,  m^{(\beta_{{\hat{n}}})} \Big(t, \mu,  \delta_{z_{\alpha_{{\hat{n}},1}}}, \ldots, 
     \delta_{z_{\alpha_{{\hat{n}},\beta_{\hat{n}}}}} \Big) , \nonumber \\
     && m^{(1)}(t, \mu, \mu_k) \bigg) \bigg]  \nonumber \\
     && + \sum_{\Lambda \in e( \Lambda_{k-1})} \sum_{i=1}^{\hat{n}}   \bigg[ \frac{ \delta^{{{\hat{n}}}} \Phi}{ \delta m^{{{\hat{n}}}}}( m(t, \mu)) \bigg( m^{(\beta_1)} \Big(t, \mu, \delta_{z_{\alpha_{1,1}}}, \ldots, \delta_{z_{\alpha_{1,\beta_1}}} \Big), \ldots, \nonumber \\
     && m^{(\beta_{{i-1}})} \Big(t, \mu,  \delta_{z_{\alpha_{{{i-1}},1}}}, \ldots, 
     \delta_{z_{\alpha_{{{i-1}},\beta_{i-1}}}} \Big), m^{(\beta_{{i}}+1)} \Big(t, \mu,  \delta_{z_{\alpha_{{{i}},1}}}, \ldots, 
     \delta_{z_{\alpha_{{{i}},\beta_{i}}}} , \mu_k \Big), \nonumber \\
     && m^{(\beta_{{i+1}})} \Big(t, \mu,  \delta_{z_{\alpha_{{{i+1}},1}}}, \ldots, 
     \delta_{z_{\alpha_{{{i+1}},\beta_{i+1}}}} \Big), \ldots, m^{(\beta_{{\hat{n}}})} \Big(t, \mu,   \delta_{z_{\alpha_{{\hat{n}},1}}}, \ldots, 
     \delta_{z_{\alpha_{{\hat{n}},\beta_{\hat{n}}}}} \Big) \bigg) \bigg] . \label{ kth der U}  
       \end{eqnarray}
     For each $ \Lambda \in e(\Lambda_{k-1}) $ and $i \in \{1, \ldots, \hat{n} \}$, we define functions $\Theta^{(1)}, \Theta^{(2)}_{\Lambda,i} : \bT^d \to \bR$ by 
       \begin{eqnarray}
            \Theta^{(1)}(x) & :=  &  \sum_{\Lambda \in e( \Lambda_{k-1})}   \intrd \ldots \intrd  \frac{ \delta^{{{\hat{n}+1}}} \Phi}{ \delta m^{{{\hat{n}+1}}}}( m(t, \mu))(z_1, \ldots, z_{\hat{n}},x) \nonumber \\
            &&   m^{(\beta_1)} \Big(t, \mu, \delta_{z_{\alpha_{1,1}}}, \ldots, \delta_{z_{\alpha_{1,\beta_1}}} \Big) (dz_1)  \ldots  m^{(\beta_{{\hat{n}}})} \Big(t, \mu,  \delta_{z_{\alpha_{{\hat{n}},1}}}, \ldots, 
     \delta_{z_{\alpha_{{\hat{n}},\beta_{\hat{n}}}}} \Big) (dz_{\hat{n}}) \, , \nonumber \\
     \Theta^{(2)}_{\Lambda,i} (x) & :=  & \intrd \ldots \intrd \frac{ \delta^{{{\hat{n}}}} \Phi}{ \delta m^{{{\hat{n}}}}}( m(t, \mu)) (z_1, \ldots, z_{i-1}, x, z_{i+1}, \ldots, z_{\hat{n}}) \nonumber \\
     && m^{(\beta_1)} \Big(t, \mu, \delta_{z_{\alpha_{1,1}}}, \ldots, \delta_{z_{\alpha_{1,\beta_1}}} \Big)(dz_1) \ldots m^{(\beta_{{i-1}})} \Big(t, \mu,  \delta_{z_{\alpha_{{{i-1}},1}}}, \ldots, 
     \delta_{z_{\alpha_{{{i-1}},\beta_{i-1}}}} \Big)(dz_{i-1}) \nonumber \\
     && m^{(\beta_{{i+1}})} \Big(t, \mu,  \delta_{z_{\alpha_{{{i+1}},1}}}, \ldots, 
     \delta_{z_{\alpha_{{{i+1}},\beta_{i+1}}}} \Big)(dz_{i+1}) \ldots m^{(\beta_{{\hat{n}}})} \Big(t, \mu,  \delta_{z_{\alpha_{{\hat{n}},1}}}, \ldots, 
     \delta_{z_{\alpha_{{\hat{n}},\beta_{\hat{n}}}}} \Big) (dz_{\hat{n}}). \nonumber 
       \end{eqnarray}
  By the assumption  \hregphi{n+k}{k},  it is clear that $\Theta^{(1)}, \Theta^{(2)}_{\Lambda,i} \in W^{n+k,\infty}.$ Therefore, using the notations \eqref{I1 def} and \eqref{Ij iteration}, Theorem \ref{thm Ij} implies that
  \begin{eqnarray}
       &&  \frac{d}{d  \eps} \bigg|_{\eps=0^{+}} \bigg[  \frac{\delta^{{k-1}} \cU}{\delta m^{k-1}} ( t,(1-\eps) \mu + \eps \mu_k)(z_1, \ldots , z_{k-1}) \bigg] \nonumber \\
       & = & \intrd I^{(1)} (x, \mu; \Theta^{(1)},t) \, (\mu_k-\mu)(dx) +  \sum_{\Lambda \in e( \Lambda_{k-1})} \sum_{i=1}^{\hat{n}} \intrd  I^{(\beta_i+1)} (x, \mu, \delta_{z_{\alpha_{i,1}}}, \ldots, \delta_{z_{\alpha_{i,\beta_i}}}; \Theta^{(2)}_{\Lambda,i} ,t) \, (\mu_k- \mu)(dx), \nonumber 
  \end{eqnarray}
  which shows that $\frac{\delta^{{k}} \cU}{\delta m^{k}} $ exists and is given by
  \begin{eqnarray}
       &&  \  \frac{\delta^{{k}} \cU}{\delta m^{k}} ( t,\mu)(z_1, \ldots , z_{k-1},x)  \nonumber \\
       & = &  I^{(1)} (x, \mu; \Theta^{(1)},t) +  \sum_{\Lambda \in e( \Lambda_{k-1})} \sum_{i=1}^{\hat{n}} I^{(\beta_i+1)} (x, \mu, \delta_{z_{\alpha_{i,1}}}, \ldots, \delta_{z_{\alpha_{i,\beta_i}}}; \Theta^{(2)}_{\Lambda,i} ,t). \nonumber 
  \end{eqnarray}
     By adopting the same normalisation argument as \eqref{normalisation argument part i} and \eqref{normalisation argument part ii}, formula \eqref{ kth der U} gives
  \begin{eqnarray}
      &&   \frac{\delta^{{k}} \cU}{\delta m^{k}} ( t, \mu )(z_1, \ldots , z_{k})  \nonumber \\
      & =&  \sum_{\Lambda \in e( \Lambda_{k-1})}    \bigg[ \frac{ \delta^{{{\hat{n}+1}}} \Phi}{ \delta m^{{{\hat{n}+1}}}}( m(t, \mu)) \bigg( m^{(\beta_1)} \Big(t, \mu, \delta_{z_{\alpha_{1,1}}}, \ldots, \delta_{z_{\alpha_{1,\beta_1}}} \Big), \ldots,  m^{(\beta_{{\hat{n}}})} \Big(t, \mu,  \delta_{z_{\alpha_{{\hat{n}},1}}}, \ldots, 
     \delta_{z_{\alpha_{{\hat{n}},\beta_{\hat{n}}}}} \Big) , \nonumber \\
     && m^{(1)}(t, \mu, \delta_{z_k}) \bigg) \bigg]  \nonumber \\
     && + \sum_{\Lambda \in e( \Lambda_{k-1})} \sum_{i=1}^{\hat{n}}   \bigg[ \frac{ \delta^{{{\hat{n}}}} \Phi}{ \delta m^{{{\hat{n}}}}}( m(t, \mu)) \bigg( m^{(\beta_1)} \Big(t, \mu, \delta_{z_{\alpha_{1,1}}}, \ldots, \delta_{z_{\alpha_{1,\beta_1}}} \Big), \ldots, \nonumber \\
     && m^{(\beta_{{i-1}})} \Big(t, \mu,  \delta_{z_{\alpha_{{{i-1}},1}}}, \ldots, 
     \delta_{z_{\alpha_{{{i-1}},\beta_{i-1}}}} \Big), m^{(\beta_{{i}}+1)} \Big(t, \mu,  \delta_{z_{\alpha_{{{i}},1}}}, \ldots, 
     \delta_{z_{\alpha_{{{i}},\beta_{i}}}} , \delta_{z_k} \Big), \nonumber \\
     && m^{(\beta_{{i+1}})} \Big(t, \mu,  \delta_{z_{\alpha_{{{i+1}},1}}}, \ldots, 
     \delta_{z_{\alpha_{{{i+1}},\beta_{i+1}}}} \Big), \ldots, m^{(\beta_{{\hat{n}}})} \Big(t, \mu,   \delta_{z_{\alpha_{{\hat{n}},1}}}, \ldots, 
     \delta_{z_{\alpha_{{\hat{n}},\beta_{\hat{n}}}}} \Big) \bigg) \bigg]  \nonumber  \\
     & = &  \sum_{\Lambda \in e(\Lambda_k)} \bigg[ \frac{ \delta^{{{\hat{n}}}} \Phi}{ \delta m^{{{\hat{n}}}}}( m(t, \mu)) \bigg( m^{(\beta_1)} \Big(t, \mu, \delta_{z_{\alpha_{1,1}}}, \ldots, \delta_{z_{\alpha_{1,\beta_1}}} \Big), \ldots, m^{(\beta_{{\hat{n}}})} \Big(t, \mu, \delta_{z_{\alpha_{{\hat{n}},1}}}, \ldots, \delta_{z_{\alpha_{{\hat{n}},\beta_{\hat{n}}}}} \Big) \bigg) \bigg].   \nonumber 
       \end{eqnarray}
Finally, if we also assume \hintphi{n+k-1}{k}, then by  Theorem \ref{ thm mk lipschitz}, for any $\Lambda \in e(\Lambda_k)$, 
\begin{eqnarray} && \sup_{z_1, \ldots, z_k \in \bT^d} \sup_{\mu \in \cP(\bT^d)} \sup_{t \in [0,T]} \bigg|  \frac{ \delta^{{{\hat{n}}}} \Phi}{ \delta m^{{{\hat{n}}}}}( m(t, \mu)) \bigg( m^{(\beta_1)} \Big(t, \mu, \delta_{z_{\alpha_{1,1}}}, \ldots, \delta_{z_{\alpha_{1,\beta_1}}} \Big), \ldots, m^{(\beta_{{\hat{n}}})} \Big(t, \mu, \delta_{z_{\alpha_{{\hat{n}},1}}}, \ldots, \delta_{z_{\alpha_{{\hat{n}},\beta_{\hat{n}}}}} \Big) \bigg) \bigg| \nonumber \\
& \leq & C \sup_{z_1, \ldots, z_k \in \bT^d}  \sup_{\mu \in \cP(\bT^d)}  \sup_{t \in [0,T]} \prod_{i=1}^{\hat{n}} \Big\| m^{(\beta_i)} \Big(t, \mu, \delta_{z_{\alpha_{i,1}}}, \ldots, \delta_{z_{\alpha_{i,\beta_i}}} \Big) \Big\|_{-(n+k-1,\infty)} \nonumber \\
& \leq & C \sup_{z_1, \ldots, z_k \in \bT^d}  \sup_{\mu \in \cP(\bT^d)}  \sup_{t \in [0,T]} \prod_{i=1}^{\hat{n}} \Big\| m^{(\beta_i)} \Big(t, \mu, \delta_{z_{\alpha_{i,1}}}, \ldots, \delta_{z_{\alpha_{i,\beta_i}}} \Big) \Big\|_{-(n+\beta_i-1,\infty)}  < +\infty. \nonumber 
\end{eqnarray}
       \end{proof}
       \section*{Acknowledgements}
The author is indebted to Prof. Pierre Cardaliaguet and  Dr. {\L}ukasz Szpruch for useful suggestions in various occasions, and to Prof. Fran\c{c}ois Delarue for the help in developing the proof of Proposition 2.2.

\end{document}